\documentclass{article}

\usepackage{makeidx}
\usepackage{latexsym}
\usepackage{amsfonts}
\usepackage{amssymb}
\usepackage{amsmath}
\usepackage{amstext}
\usepackage{amsthm}
\usepackage{mathrsfs}
\usepackage{mathabx}

\newtheorem{theorem}{Theorem}

\newtheorem{lemma}[theorem]{Lemma}

\newtheorem{corollary}[theorem]{Corollary}

\newtheorem{proposition}[theorem]{Proposition}

\begin{document}

\title{Breaking ties in collective decision making\footnote{Daniela Bubboloni was supported by GNSAGA of INdAM.
}}
\author{\textbf{Daniela Bubboloni}\\
{\small {Dipartimento di Matematica e Informatica ``Ulisse Dini''} }\\
\vspace{-6mm}\\
{\small {Universit\`{a} degli Studi di Firenze} }\\
\vspace{-6mm}\\
{\small {viale Morgagni 67/a, 50134, Firenze, Italy}}\\
\vspace{-6mm}\\
{\small {e-mail: daniela.bubboloni@unifi.it}}\\
\vspace{-6mm}\\
{\small tel: +39 055 2759667}\\
\vspace{-6mm}\\
{\small {https://orcid.org/0000-0002-1639-9525}}
 \and \textbf{Michele Gori}
 \\
{\small {Dipartimento di Scienze per l'Economia e  l'Impresa} }\\
\vspace{-6mm}\\
{\small {Universit\`{a} degli Studi di Firenze} }\\
\vspace{-6mm}\\
{\small {via delle Pandette 9, 50127, Firenze, Italy}}\\
\vspace{-6mm}\\
{\small {e-mail: michele.gori@unifi.it}}\\
\vspace{-6mm}\\
{\small tel: +39 055 2759707}\\
\vspace{-6mm}\\
{\small {https://orcid.org/0000-0003-3274-041X}}
}

\maketitle

\begin{abstract}
\noindent Many classical social preference (multiwinner social choice) correspondences are resolute only when two alternatives and an odd number of individuals are considered. Thus, they generally admit several resolute refinements, each of them naturally interpreted as a tie-breaking rule. In this paper
we find out conditions which make a social preference (multiwinner social choice) correspondence admit a resolute refinement fulfilling suitable weak versions of the anonymity and neutrality principles, as well as reversal symmetry (immunity to the reversal bias).
\end{abstract}

\vspace{4mm}


\vspace{2mm}


\section{Introduction}

Consider a committee having $h\ge 2$ members whose purpose is to determine a ranking of $n \ge 2$ alternatives and assume that committee members are supposed to express their preferences via a ranking of the alternatives. A preference profile is a list of individual preferences, one for each committee member, and a social preference correspondence ({\sc spc}) is a correspondence from the set of preference profiles to the set of alternatives rankings. Each {\sc spc} represents then a particular decision process which
determines a set of social preferences, whatever individual preferences the committee members express.

In the literature many {\sc spc}s have been proposed and studied. Most of them satisfy two requirements which are considered strongly desirable by social choice theorists, namely anonymity and neutrality. A {\sc spc} is said anonymous if the identities of individuals are irrelevant to determine the social outcome, that is, it selects the same social outcome for any pair of preference profiles such that we get one from the other by permuting individual names; neutral if alternatives are equally treated, that is, for every pair of preference profiles such that we get one from the other by permuting alternative names, the social outcomes associated with them coincide up to the considered permutation.

Since in many cases collective decision processes are required to select a unique outcome, an important role in social choice theory is played by resolute {\sc spc}s, namely those
{\sc spc}s associating a singleton with any preference profile. Unfortunately, classical {\sc spc}s are not resolute in general.
As a consequence, if the members of a committee agree to use a classical {\sc spc} to aggregate their preferences and a unique outcome is needed, then they also need to find an agreement on which tie-breaking rule to use when two or more rankings are selected by the chosen {\sc spc}.

The concept of tie-breaking rule can be naturally formalized in terms of refinement of a {\sc spc}. Let $C$ and $C'$ be two {\sc spc}s. We say that $C'$ is a refinement of $C$ if, for every preference profile, the set of social preferences selected by $C'$ is a subset of the set of social preferences selected by $C$. Thus, refinements of $C$ can be thought as methods to reduce the ambiguity in the choice made by $C$. In particular, resolute refinements of $C$ completely eliminate the ambiguity leading to a unique social outcome, so that they can be identified with the possible tie-breaking rules for $C$.
In general, if $C$  is not resolute, it admits many resolute refinements. Thus, an interesting issue to address is to understand whether it is possible to find resolute refinements of $C$ which satisfy suitable properties. In particular, since it is immediate to understand that resolute refinements of even anonymous and neutral {\sc spc}s  are not generally anonymous and neutral, one may wonder whether anonymous and neutral resolute refinements of a given {\sc spc} do exist. Unfortunately, as proved by Bubboloni and Gori
(2014,  Theorem 5), the existence of an anonymous, neutral and resolute {\sc spc} is equivalent to the strong arithmetical condition $\gcd(h,n!)=1$ which is rarely satisfied in practical situations\footnote{The condition $\gcd(h,n!)=1$ was first introduced by
Moulin (1983, Theorem 1, p.25) as a necessary and sufficient condition for the existence of anonymous, neutral and efficient social choice functions; Bubboloni and Gori
(2014, Theorem 15) generalize Moulin's result considering in addition the qualified majority principle; Campbell and Kelly
(2015) show that if $n > h$ (so that $\gcd(h,n!)\neq 1$) and anonymous, neutral and
resolute  social choice functions exist, then those  functions exhibit even more undesirable behaviours than
inefficiency. The existence of anonymous, neutral and resolute {\sc spc}s satisfying further properties is also studied by Bubboloni and Gori (2015), who focus on the majority principle, and Do\u gan and Giritligil (2015), who instead focus on monotonicity properties.}. As a consequence, given a {\sc spc}, we cannot in general get any anonymous and neutral resolute refinement of it.

In many concrete cases, however, anonymity and neutrality are too strong requirements.
That happens, for instance, for those committees having a president who is more influential than the other committee members or when a committee evaluates job candidates giving the female candidates an advantage over the male ones. Indeed, in such situations, individual and alternative names are not immaterial. For such a reason, we propose generalized versions of anonymity and neutrality which can take into account possible distinctions among individuals and among alternatives. Consider $S_h$, the group of permutations of the set $H=\{1,\ldots, h\}$ of committee member names and $S_n$, the group of permutations of the set $N=\{1,\ldots, n\}$ of alternative names.
Given a subgroup $V$ of $S_h$ we say that a {\sc spc} is $V$-anonymous if it selects the same social outcome for any pair of preference profiles such that we get one from the other by permuting individual names according to a permutation in $V$; given a subgroup $W$ of $S_n$ we say that a {\sc spc} is $W$-neutral if, for every pair of preference profiles such that we get one from the other by permuting alternative names according to a permutation in $W$, it associates with them social outcomes which coincide up to the considered permutation. The special groups of permutations $V$ and $W$ need to be determined on the basis of the specific features of the situation at hand. Of course, $S_h$-anonymity corresponds to anonymity and $S_n$-neutrality corresponds neutrality. As a consequence, each anonymous (neutral) {\sc spc} is $V$-anonymous ($W$-neutral) for every $V$ ($W$). It is also worth noting that, depending on the particular structure of $V$ and $W$, it is possible to get $V$-anonymous, $W$-neutral and resolute {\sc spc}s even when no anonymous, neutral and resolute {\sc spc} exists\footnote{For instance, in the trivial case where $V$ and $W$ are singletons whose unique element is the identity permutation, every resolute {\sc spc} is $V$-anonymous and $W$-neutral.}. That implies that, in a situation where only suitable weak versions of anonymity and neutrality are sensible, there may be room for finding a resolute {\sc spc} satisfying such versions of anonymity and neutrality, possibly being the refinement of a given {\sc spc}.

In this paper we propose conditions on $V$ and $W$ which are necessary and sufficient to make a $V$-anonymous and
$W$-neutral {\sc spc} admit a $V$-anonymous and $W$-neutral resolute refinement.
Those conditions are summarized by an arithmetical relation between the size $|W|$ of  $W$ and a
special number $\gamma(V)$ associated with $V$ (defined in \eqref{type-number}). Namely, given a $V$-anonymous and $W$-neutral {\sc spc} $C$ (like, for instance, the Borda, the Copeland, the Minimax and the Kemeny {\sc spc}s which are anonymous and neutral and then $V$-anonymous and $W$-neutral), we have that $C$ admits a $V$-anonymous and $W$-neutral resolute refinement if and only if\footnote{Throughout the paper, given integers $x_1,\dots, x_s$, for some $s\in \mathbb{N}$, we denote by $\gcd(x_1,\dots,x_s)$ their greatest common divisor and by $\mathrm{lcm}(x_1,\dots,x_s)$ their least common multiple. }
\begin{equation}\label{intro0}
\gcd(\gamma(V),|W|)=1
\end{equation}
(Theorem \ref{application1}).
The computation of the numbers $\gamma(V)$ and $|W|$ is, in principle, hard. Fortunately, it becomes easy for many subgroups of interest in the applications (see Section \ref{app}).  For example, consider a situation where individuals can be divided into disjoint subcommittees, say $H_1,\dots,H_r$, where individuals have the same decision power and that alternatives can be divided into disjoint subclasses, say $N_1,\dots,N_s$, where alternatives have the same exogenous importance.
Assume first that such subcommittees (subclasses) can be ranked is such a way that, for every pair of subcommittees (subclasses), each individual (alternative) belonging to the subcommittee (subclass) having higher rank is more influential (more important) than any individual (alternative) belonging to the other subcommittee (subclass)\footnote{The presence of a president or  gender discrimination previously mentioned are special instances.}.
In this case, it is natural to require a collective decision process not to distinguish among individuals (alternatives) belonging to the same subcommittee (subclass).
Thus, two groups $V$ and $W$ naturally arise: $V$ is given by those permutations of individual names mapping every $H_i$ into itself; $W$ is given by those permutations of alternative names mapping every $N_j$ into itself. Clearly $|W|=\prod_{j=1}^s |N_j|!$ and it can be shown that
\begin{equation}\label{describe}
\gamma(V)=\gcd(|H_1|,\ldots,|H_r|),
\end{equation}
so that condition \eqref{intro0} becomes\footnote{Note that anonymity corresponds to a have $H$ as a unique subcommittee and neutrality corresponds to a have $N$ as a unique subclass. In this special case \eqref{comment} reduces to the already commented condition $\gcd(h,n!)=1$.}
\begin{equation}\label{comment}
\gcd\left(|H_1|,\ldots,|H_r|,\prod_{j=1}^s |N_j|!\right)=1.
\end{equation}
Assume now that subclasses can be ranked as before while no ranking of subcommittees is available. However, it is known that subcommittees having the same size must have the same impact in the final decision. This particular situation leads to consider a collective decision process which does not distinguish among individuals (alternatives) belonging to the same subcommittee (subclass) and also does not distinguish among subcommittees having the same size.
Thus, $V$ is now given by those permutations of individual names mapping every $H_i$ into one of the subcommittee having its same size (possibly itself), while  $W$ is the same as before. We prove then that
\[
\gamma(V)=\gcd(|H_1|t_1,\ldots,|H_r|t_r),
\]
where $t_i$ counts the number of subcommittees of size $|H_i|$,
so that condition \eqref{intro0} becomes
\[
\gcd\left(|H_1|t_1,\ldots,|H_r|t_r,\prod_{j=1}^s |N_j|!\right)=1.
\]

We further deepen the analysis of the resolute refinements of a given {\sc spc} by considering the property of reversal symmetry,
a property first introduced by Saari (1994). Recall that the reversal of a ranking of alternatives is the ranking obtained by making the best alternative
the worst, the second best alternative the second worst, and so on, and that a {\sc spc} is said reversal symmetric if, for any pair of preference profiles such that one is obtained by the other by reversing each individual preference, it associates with one of them a set of social preferences if and only if it associates with the other one the set of their reversal. We prove that the condition
\begin{equation}\label{intro1}
\gcd(\gamma(V),\mathrm{lcm}(|W|,2))=1,
\end{equation}
along with other technical conditions, is necessary and sufficient to make any $V$-anonymous, $W$-neutral and reversal symmetric {\sc spc} admit a $V$-anonymous, $W$-neutral and reversal symmetric resolute refinement (Theorem \ref{application2}).
In particular, from that result, it is easily deduced that the Borda and the Copeland {\sc spc}s admit a $V$-anonymous, $W$-neutral and reversal symmetric resolute refinement if and only if \eqref{intro1} holds true.

Of course, the study of the resolute refinements and their properties is meaningful also in other frameworks. For such a reason we follow up on this issue focusing on the so-called $k$-multiwinner social choice correspondences ({\sc $k$-scc}s), that is, those procedures which associate with any preference profile a family of sets of $k$ alternatives to be interpreted as the family of all the sets of alternatives that can be considered the best $k$ alternatives for the society. This framework, which extends the classical and well-established single winner framework (corresponding to the case $k=1$), has been explored for about thirty years and constitutes nowadays an interesting and growing research area. We refer to the recent papers by Elkind et al. (2017) and Faliszewski et al. (2017) for further information on that topic.

We study the problem to determine whether a {\sc $k$-scc} admits $V$-anonymous and $W$-neutral resolute refinements\footnote{Those concepts are easily adapted to {\sc $k$-scc}s. See Section $4$.} and whether some of them are immune to the reversal bias too. Recall that a $k$-{\sc scc} is immune to the reversal bias if it never associates the same singleton with a preference profile and its reversal (see Saari and Barney, 2003, and Bubboloni and Gori, 2016a). We prove that if \eqref{intro0} holds true, then any  $V$-anonymous and $W$-neutral {\sc $k$-scc} admits a  $V$-anonymous and $W$-neutral resolute refinement (Theorem \ref{application3}). Assuming \eqref{intro1}, we also find out conditions to make every $V$-anonymous, $W$-neutral and immune to the reversal bias {\sc $k$-scc} admit a resolute refinement having the same properties (Theorem \ref{application4}). Remarkably, that analysis put in evidence the role played by the number $k$ of selected alternatives and the number $n$ of alternatives pointing out that the existence of an  immune to the reversal bias resolute refinement is guaranteed when  the pair $(n,k)$ belongs to a specific set\footnote{We stress that equality in \eqref{describe} along with the described results about {\sc $k$-scc}s (Theorem \ref{application3} and \ref{application4}) imply as an immediate consequence the main result in Bubboloni and Gori (2016b), namely Theorem 8 therein.}.

We emphasize that the methods developed in the paper also allow to deal with situations in which forms of symmetry conceptually different from $V$-anonymity and $W$-neutrality are required. To that purpose we examine, in detail, the case of a committee facing the problem of electing a subcommittee of size $k$ among $n$ members of the committee who run as candidates. In this particular situation, a collective decision process should make no distinction among individuals who are candidates and among individuals who are not candidates and, at the same time, it must carefully take into account the fact that the alternatives to choose from are now a subset of the set of  individuals. In Section \ref{case-forth}, we prove that if $\gcd(h,n)=1$, then any anonymous and neutral {\sc $k$-scc} admits a resolute refinement consistent with the above described requirements.

Note that, at least in the single winner case, two special resolute refinements are sometimes used. They are built using two simple methods to break ties.
The first method, proposed by Moulin (1988), is based on a tie-breaking agenda, that is, an exogenously given ranking of the alternatives: in ambiguity case, it is chosen the alternative of the social outcome which is best ranked in the agenda. The second one is instead based on the preferences of one of the individuals appointed as tie-breaker: in ambiguity case, it is chosen the alternative of the social outcome which is best ranked by the tie-breaker.
Of course, the resolute refinements built through a tie-breaking agenda fail in general to be neutral while the one built through a tie-breaker fail in general to be anonymous.
Our work allows to cast these two special types of tie-breaking rules within a more general theory (Section 8).\footnote{Some remarks about the tie-breaking agenda and the role of the tie-breaker are proposed in Do\u gan and Giritligil (2015) and Bubboloni and Gori (2016b), where the problem of finding resolute refinements of {\sc 1-scc}s is considered.}

Note also that some contributions dealing with weak versions of anonymity and neutrality are present in the literature. Assuming there are only two alternatives, Perry and Powers (2008) calculate the number of resolute {\sc spc}s that are anonymous and neutral and the number of {\sc spc}s satisfying a weak version of anonymity (that is, all individuals but one are anonymous) and neutrality; Powers (2010) shows that a {\sc spc} satisfies the previously described weak anonymity, neutrality and Maskin monotonicity if and only if it is close to {\sc spc} obeying an absolute qualified majority;  Quesada (2013) identifies seven axioms (among which are weak versions of anonymity and neutrality) characterizing the rules that are either the relative majority rule or the relative majority rule where a given individual can break the ties; Campbell and Kelly (2011, 2013) show that the relative majority is implied both by a suitable weak version of anonymity, neutrality and monotonicity, as well as by limited neutrality, anonymity and monotonicity. In the general case for the number of alternatives, some observations about different levels of anonymity and neutrality can be found in the paper by Kelly (1991). Bubboloni and Gori (2015, 2016b) consider, only for {\sc 1-scc}s, weak versions of anonymity and neutrality, determined by partitions of individuals and alternatives.
More recently, a wide literature about {\sc $k$-scc}s which are representation-focused on some attributes of the alternatives (sex, age, profession, ethnicity) has been growing. See, for instance, Lang and Skowron (2016), Bredereck et al. (2017), Celis et al. (2018).

The techniques used in the paper are based on the algebraic approach focused on the theory of finite permutation groups and the notion of action of a group on a set\footnote{It is worth mentioning that Kelly (1991) discusses the role of symmetry in the arrovian framework through suitable subgroups of the symmetric group.
Within the topological approach to social choices developed by Chichilnisky (1980), several algebraic concepts, in particular that of symmetric group, are crucial for proving many results.} proposed by E\u gecio\u glu (2009) and E\u gecio\u glu and Giritligil (2013) for analysing the set of preference profiles and later developed by Bubboloni and Gori (2014, 2015, 2016b), which we mainly refer to, and Do\u gan and Giritligil (2015).
However, generalizing the ideas and results in Bubboloni and Gori (2015, 2016b) for managing the problem of finding resolute refinements of {\sc spc}s
and {\sc $k$-scc} is in no way a trivial adaptation. On the contrary, some proofs require complex and tricky arguments and, in particular, a deep consideration of the concept of regular group, first introduced in Bubboloni and Gori (2015), is needed.

The paper is organized as follows. In Sections \ref{prelim}-\ref{scc} we give the definitions, the notation and the basic results needed to understand the model. In Section \ref{maingen} the main results are stated and developed. Section \ref{regular-gr} is devoted to a careful analysis of the concept of regular group. Section
\ref{gen-an-neut} explores $V$-anonymity and $W$-neutrality as an application of the main results of the general theory. Section \ref{app} is about some concrete applications of $V$-anonymity and $W$-neutrality, as described in this introduction. Finally, appendices \ref{sec-tec} and \ref{app-C} contain the proofs of the main results.

\section{Preliminaries}\label{prelim}

Throughout the paper, we assume $0\not\in \mathbb{N}$ and we set $\mathbb{N}_0=\mathbb{N}\cup\{0\}$. For $k\in \mathbb{N}$,  the set $\{1,\ldots,k\}$
is denoted by $\ldbrack k \rdbrack$.
Given a finite set $X$ we denote by $|X|$ its size. A subset of $X$ of size $k$ is called a $k$-subset of $X$. We denote by
$\mathbb{P}(X)$ the set of the subsets of $X$ and by $\mathbb{P}_k(X)$ the set of the $k$-subsets of $X$.

\subsection{Relations}

Let $X$ be a nonempty and finite set. A relation on $X$ is a subset of $X^2$.
The set of the relations on $X$ is denoted by $\mathbf{R}(X)$.
Fix $R\in \mathbf{R}(X)$. Given $x,y\in X$, we usually write $x\succeq_R y$ instead of $(x,y)\in R$ and  $x\succ_R y$ instead of $(x,y)\in R$ and $(y,x)\notin R$.
We say that $R$ is complete if, for every $x,y\in X$, $x\succeq_R y$ or $y\succeq_R x$; reflexive if, for every $x\in X$, $x\succeq_R x$; irreflexive if, for every $x\in X$, $x\not\succeq_R x$; antisymmetric if, for every $x,y\in X$, $x\succeq_R y$ and $y\succeq_R x$ imply $x=y$; asymmetric if, for every $x,y\in X$, $x\succeq_R y$ implies $y\not \succeq_R x$; transitive if, for every $x,y,z\in X$, $x\succeq_Ry$ and $y\succeq_R z$ imply $x\succeq_R z$; acyclic if, for every sequence $x_1,\ldots,x_s$ of $s\ge 2$ distinct elements of $X$ such that $x_i\succeq_R x_{i+1}$ for all $i\in\ldbrack s-1 \rdbrack$, we have that $x_s \not\succeq_R x_1$. It is well-known that if $R$ is transitive and irreflexive (antisymmetric, asymmetric), then $R$ is acyclic.
Complete and transitive relations on $X$ are called orders on $X$. The set of  orders on $X$ is denoted by $\mathbf{O}(X)$.
Complete, transitive and antisymmetric relations on $X$ are called linear orders on $X$. The set of linear orders on $X$ is denoted by $\mathbf{L}(X)$.

\subsection{Groups, permutations and relations}\label{plo}

In the paper we make use of basic finite group theory  and use standard notation. Our general reference is Jacobson (1974). For completeness we  give here the main notions and notation we are going to use.
Let $G$ be a finite group.
Given $g\in G$ we denote by $|g|$ the order of $g$.
If $U$ is a subgroup of $G$, we use the notation $U\leq G$.
Given $g,v\in G$ and $U\leq G$, the conjugate of $g$ by $v$ is $g^v=vgv^{-1}$ and the conjugate of $U$ by $v$ is the subgroup $U^v=\{g^v\in G: g\in U\}$.
We say that $g_1,g_2\in G$ are conjugate if there exists $v\in G$ such that $g_2= g_1^v.$ The group constituted only by the neutral element is called the trivial group.

Let $n\in\mathbb{N}$. The set of bijective functions from $\ldbrack n \rdbrack$ to itself is denoted by $S_n$.  $S_n$ is a group with product defined, for every $\sigma_1,\sigma_2\in S_n$, by the composition\footnote{Given $A$, $B$ and $C$ sets and $f:A\to B$ and $g:B\to C$ functions, we denote by $gf$ the right-to-left composition of $f$ and $g$, that is, the function from $A$ to $C$ defined, for every $a\in A$, as $gf(a)=g(f(a))$.
} $\sigma_1\sigma_2\in S_n$. The neutral element of $S_n$ is given by the identity function on $\ldbrack n \rdbrack$, denoted by $id$.
$S_n$ is called the symmetric group on $\ldbrack n \rdbrack$ and its elements are called permutations.
A permutation $\sigma\in S_n\setminus\{id\}$ is the product of disjoint cycles of lengths greater than or equal to $2$ uniquely determined by $\sigma$, up to reordering, and called the {\it constituents} of $\sigma$. The {\it order reversing permutation} in $S_n$ is the permutation $\rho_0\in S_n$ defined, for every $r\in \ldbrack n \rdbrack$, as $\rho_0(r)=n-r+1$.  Obviously, we have $|\rho_0|=2$ and thus  $\Omega=\{id, \rho_0\}$ is a subgroup of $S_n$.
Note that $\rho_0$ has exactly  one fixed point when $n$ is odd and no fixed point when $n$ is even. Note also that $\Omega$ is an abelian group which admits as unique subgroups $\{id\}$ and $\Omega$.

Let $\sigma\in S_n$. Given $W\in \mathbb{P}(\ldbrack n \rdbrack)$, we denote the image of $W$ through $\sigma$ by $\sigma W$ (instead of $\sigma(W)$).
Moreover, given $\mathbb{W}\subseteq  \mathbb{P}(\ldbrack n \rdbrack)$, we define $\sigma \mathbb{W}=\{\sigma W \in \mathbb{P}(\ldbrack n \rdbrack): W\in \mathbb{W}\}$.
Note that, for every  $\sigma_1, \sigma_2\in S_n$,
$W\in \mathbb{P}(\ldbrack n \rdbrack)$ and
$\mathbb{W}\subseteq  \mathbb{P}(\ldbrack n \rdbrack)$, we have that
$(\sigma_2\sigma_1)W=\sigma_2(\sigma_1 W)$ and $(\sigma_2\sigma_1)\mathbb{W}=\sigma_2(\sigma_1 \mathbb{W})$. Thus brackets can be omitted in this type of writings.
Given $R\in\mathbf{R}(\ldbrack n \rdbrack)$, we set
\begin{equation}\label{defpro}
\sigma R=\big\{(\sigma(x),\sigma(y))\in \ldbrack n \rdbrack^2 :(x,y)\in R\big\},
\quad R\rho_0=\big\{(y,x)\in \ldbrack n \rdbrack^2 :(x,y)\in R\big\},
\end{equation}
and $R\; id=R$. In other words, for every $x,y\in \ldbrack n \rdbrack$, $x\succeq_{R}y$ if and only if $\sigma(x)\succeq_{\sigma R}\sigma(y)$; $x\succeq_{R}y$ if and only if
$y\succeq_{R\rho_0}x$.
Given $\mathbf{Q}\subseteq \mathbf{R}(\ldbrack n \rdbrack)$ and $\rho\in \Omega$, we also set
\[
\sigma \mathbf{Q}=\{\sigma R\in \mathbf{R}(\ldbrack n \rdbrack): R\in \mathbf{Q}\},
\quad
\mathbf{Q} \rho=\{R \rho\in \mathbf{R}(\ldbrack n \rdbrack): R\in  \mathbf{Q}\}.
\]
Note that, for every $\sigma_1,\sigma_2\in S_n$, $\rho_1,\rho_2\in\Omega$, $R\in \mathbf{R}(\ldbrack n \rdbrack)$ and
$\mathbf{Q}\subseteq \mathbf{R}(\ldbrack n \rdbrack)$, we have that
$(\sigma_2\sigma_1)R=\sigma_2(\sigma_1 R)$, $R(\rho_1\rho_2)=(R\rho_1)\rho_2$, $(\sigma_1 R)\rho_1=\sigma_1(R\rho_1)$ and $(\sigma_2\sigma_1)\mathbf{Q}=\sigma_2(\sigma_1 \mathbf{Q})$, $\mathbf{Q}(\rho_1\rho_2)=(\mathbf{Q}\rho_1)\rho_2$, $(\sigma_1 \mathbf{Q})\rho_1=\sigma_1(\mathbf{Q}\rho_1)$.
Those properties allow to avoid brackets when writing such kinds of products.

\section{Preference relations and preference profiles}\label{model}

From now on, let $n\in \mathbb{N}$ with $n\ge 2$ be fixed, and let $N=\ldbrack n \rdbrack$ be the set of names of alternatives.

A {\it preference relation} on $N$ is a linear order on $N$. Let $q\in \mathbf{L}(N)$ be a preference relation. If
 $x,y\in N$ are alternatives, we interpret the writing $x\succeq_{q}y$ by saying that $x$ is at least as good as $y$ (according to $q$), and the writing $x\succ_{q}y$ by saying that $x$ is preferred to $y$ (according to $q$).  Note that, since $q$ is a linear order,  $x\succ_{q}y$ is equivalent to $x\neq y$ and $x\succeq_{q}y$.
 It is well-known that there exists a unique numbering $x_1, x_2,\dots,x_n$ of the distinct elements in $N$ such that, once set $R=\{(x_i, x_{i+1})\in N^2: i\in\{1,\dots, n-1\}\}\in\mathbf{R}(N)$, we have $q\supseteq R$ and $q$ is the only linear order containing the relation $R$.
 Thus, we can completely identify $q$ with its subset $R$, which we write in the form $x_1\succ_{q} x_2 \succ_{q}\dots \succ_{q} x_n.$
We refer to $r$ as the rank of $x_r$ in $q$.
Since the map from $\ldbrack n \rdbrack$ to $\ldbrack n \rdbrack$ which associates with any $r\in  \ldbrack n \rdbrack$ the alternative $x_r$ is a bijection and thus an element of $S_n,$ we also have that $q$ is completely identified with such permutation, which we continue to call $q$. Explicitly, if $q$ is interpreted in $S_n$, then
$q(r)=x_r$ for all $r\in \ldbrack n \rdbrack.$
In this way, we have established a well-known and remarkable identification of $\mathbf{L}(N)$ with $S_n$.
Note now that if  $\psi\in S_n$, then the relations $\psi q$ and $q\rho_0$, as defined in \eqref{defpro}, are the linear orders given by
$\psi(x_1)\succ_{\psi q} \psi(x_2) \succ_{\psi q}\dots \succ_{\psi q} \psi(x_n)$
and
$
x_n\succ_{q\rho_0} x_{n-1} \succ_{q\rho_0}\dots \succ_{q\rho_0} x_1.
$
In particular, for every $\psi\in S_n$ and  $\rho\in \Omega$, $\psi q$ and $q\rho$ can be interpreted as products of permutations in the symmetric group $S_n.$ As a consequence,
 thanks to the cancellation law in $S_n$, we have that, for every $\psi_1,\psi_2\in S_n$ and $\rho_1,\rho_2\in \Omega$, $\psi_1 q=\psi_2 q$  implies $\psi_1=\psi_2$ and $q\rho_1=q\rho_2$ implies $\rho_1=\rho_2$.
Moreover, by elementary properties of the symmetric group, we have that $\psi\mathbf{L}(N)\rho=\mathbf{L}(N)$ for all $\psi\in S_n$ and  $\rho\in \Omega$.

From now on, let $h\in \mathbb{N}$ with $h\ge 2$ be fixed, and let $H=\ldbrack h \rdbrack$ be the set of names of individuals.
A {\it preference  profile} is an element of $\mathbf{L}(N)^h$. The set $\mathbf{L}(N)^h$ is denoted by $\mathcal{P}$.
If $p\in\mathcal{P}$ and $i\in H$, the $i$-th component $p_i$ of $p$  represents the preferences of individual $i$.

Let us set now
\[
G=S_h\times S_n\times \Omega.
\]
Then $G$ is a group through component-wise multiplication, that is, defining, for every $(\varphi_1,\psi_1,\rho_1)\in G$ and $(\varphi_2,\psi_2,\rho_2)\in G$, $(\varphi_1,\psi_1,\rho_1)(\varphi_2,\psi_2,\rho_2)=
(\varphi_1\varphi_2,\psi_1\psi_2,\rho_1\rho_2).$
For every $(\varphi,\psi,\rho)\in G$ and $p\in \mathcal{P}$, define $p^{(\varphi,\psi,\rho)} \in \mathcal{P}$ as the preference profile such that, for every $i\in H$,
\begin{equation}\label{def-azione}
(p^{(\varphi,\psi,\rho)})_i=\psi p_{\varphi^{-1}(i)}\rho.
\end{equation}
Thus, the preference profile $p^{(\varphi,\psi,\rho)}$ is obtained by $p$ according to the following rules (to be applied in any order): for every $i\in H$, individual $i$ is renamed $\varphi(i)$; for every $x\in N$, alternative $x$ is renamed $\psi(x)$; for every $r\in\ldbrack n \rdbrack$, alternatives whose rank is $r$ are moved to rank $\rho(r)$.

For further details and examples on these issues the reader is referred to Bubboloni and Gori (2016b, Sec. 2.2 and 2.3).

\section{Social preference and social choice correspondences}\label{scc}

A {\it social preference correspondence} ({\sc spc}) is a
correspondence from $\mathcal{P}$ to $\mathbf{L}(N)$.
The set of the {\sc spc}s is denoted by $\mathfrak{P}$. Thus, if $C\in \mathfrak{P}$ and $p\in \mathcal{P}$, then $C(p)$ is a subset of $\mathbf{L}(N)$.

From now on, let $k\in \ldbrack n-1 \rdbrack$ be fixed. A {\it $k$-multiwinner social choice correspondence} ({\sc $k$-scc})
is a correspondence from $\mathcal{P}$ to $\mathbb{P}_k(N)$.
The set of the {\sc $k$-scc}s is denoted by $\mathfrak{C}_k$. Thus, if $C\in \mathfrak{C}_k$ and $p\in \mathcal{P}$, then $C(p)$ is a set of $k$-subsets of $N$.

We say that $C\in\mathfrak{P}$ ($C\in\mathfrak{C}_k$) is {\it decisive} if, for every $p\in\mathcal{P}$, $C(p)\neq\varnothing$; {\it resolute} if, for every
$p\in\mathcal{P}$, $|C(p)|=1$. We say that $C'\in\mathfrak{P}$ ($C'\in\mathfrak{C}_k$) is a {\it refinement} of $C\in\mathfrak{P}$ ($C\in\mathfrak{C}_k$) if, for every $p\in\mathcal{P}$, $C'(p)\subseteq C(p)$. Of course, $C$ admits a resolute refinement if and only if $C$ is decisive; $C$ admits a unique resolute refinement if and only if $C$ is resolute.

Let  $U$ be a subgroup of $G$. We say that $C\in\mathfrak{P}$ is
$U$-{\it symmetric} if, for every $p\in \mathcal{P}$ and $(\varphi,\psi,\rho)\in U$,
\begin{equation}\label{symU}
C(p^{(\varphi,\psi,\rho)})=\psi C(p) \rho;
\end{equation}
we say that $C\in\mathfrak{P}$ ($C\in\mathfrak{C}_k$) is $U$-{\it consistent} if, for every $p\in \mathcal{P}$ and $(\varphi,\psi,\rho)\in U$,
\begin{equation}\label{sccU1}
C(p^{(\varphi,\psi,\rho)})=\psi C(p)\quad\mbox{ if } \rho=id,
\end{equation}
\begin{equation}\label{sccU2}
C(p^{(\varphi,\psi,\rho)})\neq\psi  C(p) \quad\mbox{ if } \rho=\rho_0 \mbox{ and } |C(p)|=1.
\end{equation}
We stress that the writings in \eqref{symU}, \eqref{sccU1} and \eqref{sccU2} are meaningful due to the definitions of products between permutations and sets (sets of sets, relations, sets of relations) and their properties described in Section \ref{plo}. The  set of $U$-symmetric
{\sc spc}s is denoted by $\mathfrak{P}^{*U}$; the set of $U$-consistent {\sc spc}s ({\sc $k$-scc}s) is denoted by $\mathfrak{P}^U$ ($\mathfrak{C}_{k}^U$).
Observe that we do not introduce the concept of $U$-symmetry for {\sc $k$-scc}s.  Note also that, if $U'\le U\le G$, then $\mathfrak{P}^{*U}\subseteq \mathfrak{P}^{*U'}$, $\mathfrak{P}^U\subseteq \mathfrak{P}^{U'}$ and $\mathfrak{C}_k^U\subseteq \mathfrak{C}_k^{U'}$. Some basic links between the concepts of symmetry and consistency are given by the following proposition whose proof  is in Appendix \ref{sec-tec}.
\begin{proposition}\label{sym-cons} Let $U\le G$. Then the following facts hold true:
\begin{itemize}
\item[$(i)$]$\mathfrak{P}^{*U}\subseteq \mathfrak{P}^{U}$.
\item[$(ii)$] If $U\le S_h\times S_n\times \{id\}$, then $\mathfrak{P}^{*U}=\mathfrak{P}^{U}$.
\end{itemize}
\end{proposition}

The concepts of symmetry and consistency with respect to a subgroup $U$ of $G$ include some classic requirements for {\sc spc}s ({\sc $k$-scc}s). Indeed, we have that $C\in\mathfrak{P}$ ($C\in\mathfrak{C}_k$) is anonymous if and only if it is $S_h\times \{id\}\times \{id\}$-consistent; $C\in\mathfrak{P}$ ($C\in\mathfrak{C}_k$) is neutral if and only if it is $\{id\}\times S_n\times \{id\}$-consistent; $C\in\mathfrak{P}$ ($C\in\mathfrak{C}_k$) is immune to the reversal bias if and only if it is $\{id\}\times \{id\}\times \Omega$-consistent; $C\in\mathfrak{P}$ is reversal symmetric if and only if it is $\{id\}\times \{id\}\times \Omega$-symmetric. Moreover, any combination of the properties above mentioned can be interpreted in terms of $U$-symmetry or $U$-consistency where the subgroup $U$ of $G$ is naturally built as described by the next propositions. Their proof is given in  Appendix \ref{sec-tec}.
In what follows, given $U_1$ and $U_2$ subgroups of $G$, we denote by $\langle U_1,U_2\rangle$ the subgroup of $G$ generated by $U_1$ and $U_2$\footnote{See Jacobson (1974, Section 1.5).}.

\begin{proposition}\label{U,V} Let $U_1,U_2\le G$. Then $\mathfrak{P}^{*U_2}\cap \mathfrak{P}^{*U_2}=\mathfrak{P}^{*\langle U_1,U_2\rangle}$.
\end{proposition}

\begin{proposition}\label{U,V-corr}
 Let $U_1,U_2\leq G$ such that, for every $i\in\{1,2\}$, $U_i=Z_i\times R_i$ for some
 $Z_i\le S_h\times S_n$ and $R_i\le \Omega$. Then
$\mathfrak{P}^{U_1}\cap \mathfrak{P}^{U_2}=\mathfrak{P}^{\langle U_1,U_2\rangle}$ and $\mathfrak{C}_k^{U_1}\cap \mathfrak{C}_k^{U_2}=\mathfrak{C}_k^{\langle U_1,U_2\rangle}$.
\end{proposition}

In particular, $C\in\mathfrak{P}$  is anonymous, neutral and reversal symmetric if and only if $C$ is $G$-symmetric; $C\in\mathfrak{P}$ ($C\in\mathfrak{C}_k$) is anonymous, neutral and immune to the reversal bias if and only if $C$ is $G$-consistent. Classical  {\sc spc}s provide examples of
{\sc spc}s which are $G$-symmetric or at least $S_h\times S_n\times \{id\}$-symmetric. For instance, the Borda and the Copeland {\sc spc}s are $G$-symmetric  while the Minimax {\sc spc} is generally only $S_h\times S_n\times \{id\}$-symmetric.\footnote{See Bubboloni and Gori (2016b).}
Note also that the {\sc spc}  associating with every $p\in \mathcal{P}$ the set $\mathbf{L}(N)$ is surely
$G$-symmetric. Such a {\sc spc} is called the  trivial {\sc spc}.

We finally observe that there is a natural way to construct a {\sc $k$-scc} starting from a  {\sc spc}.
Indeed, given $C\in\mathfrak{P}$, we  consider the {\sc $k$-scc} $C_k$ defined, for every $p\in\mathcal{P}$, by
\[
C_k(p)=\Big\{\{q(r):1\leq r\le k\}\in \mathbb{P}_k(N): q\in C(p)\Big\}.
\]
Note that $\{q(r):1\leq r\le k\}$ is the set of alternatives ranked by $q$ in the first $k$ positions.
$C_k$ is called the  {\sc $k$-scc} {\it induced} by $C$.
Of course, given $C\in\mathfrak{P}$, $C$ is decisive if and only if  $C_k$ is decisive; if $C$ is resolute, then $C_k$ is resolute.
Moreover if $C'\in\mathfrak{P}$ is a refinement of $C$, then $C'_k$ is a refinement of $C_k.$  The following proposition, whose proof  is in Appendix \ref{sec-tec}, expresses the main basic property of the induced {\sc $k$-scc} with respect to symmetry.

\begin{proposition}\label{sym-cons2} Let $U\le G$.
If $C\in \mathfrak{P}^{*U}$, then $C_k\in\mathfrak{C}_k^{U}$.
\end{proposition}
In particular, the {\sc $k$-scc}s induced by the Borda, the Copeland and the trivial {\sc spc}s are $G$-consistent,  while the Minimax {\sc $k$-scc} is $S_h\times S_n\times \{id\}$-consistent. Note also that the {\sc $k$-scc} induced by the trivial {\sc spc}, associates with every $p\in \mathcal{P}$ the set of all the possible $k$-subsets of $N$.

\subsection{Social methods}

Let us introduce now a final new concept which will be very important in the sequel.
A {\it social method} is a function from $\mathcal{P}$ to $\mathbf{R}(N)$. The set of social methods is denoted by $\mathfrak{M}$. Let
$R\in \mathfrak{M}$. We say that $R$ is acyclic (transitive, complete etc.) if, for every $p\in \mathcal{P},$ the relation $R(p)$ is acyclic (transitive, complete etc.).
The {\sc spc} associated with $R$, denoted by $C^R$, is defined, for every $p\in \mathcal{P}$, by
\[
C^R(p)=\{q\in\mathbf{L}(N): R(p)\subseteq q\}.
\]
Note that, by the well-known Szpilrajn's extension theorem, $C^R$ is decisive if and only if $R$ is acyclic.

Given  $U\leq G$, we say that $R$ is $U$-{\it symmetric} if, for every $p\in \mathcal{P}$ and $(\varphi,\psi,\rho)\in U$,
\[
R(p^{(\varphi,\psi,\rho)})=\psi R(p) \rho.
\]
We stress that the above writing is meaningful due to the definitions of products between permutations and relations described in Section \ref{plo}.
The set of $U$-symmetric social methods is denoted by $\mathfrak{M}^{*U}.$

Note that if $U'\leq U$ and $R$ is $U$-symmetric, then $R$ is also $U'$-symmetric.
Moreover, as proved in Appendix \ref{sec-tec}, the next propositions hold true.

\begin{proposition}\label{sym-rel} Let $U\le G$. If $R\in \mathfrak{M}^{*U}$, then $C^R\in \mathfrak{P}^{*U}$.
\end{proposition}

\begin{proposition}\label{U,V-method} Let  $U_1,U_2\le G.$ Then $\mathfrak{M}^{*U_1}\cap \mathfrak{M}^{*U_2}=\mathfrak{M}^{*\langle U_1,U_2\rangle}.$
\end{proposition}

As an example, consider the social method $R$ defined, for every $p\in \mathcal{P}$, by
\[
R(p)=\{(x,y)\in N^2:\forall i\in H,\, x\succ_{p_i} y\},
\]
and note that $R\in \mathfrak{M}^{*G}$.
The {\sc spc} associated with is called the Pareto {\sc spc} and, as well-known, is decisive. The $k$-{\sc scc} induced by the Pareto {\sc spc} is called the Pareto $k$-{\sc scc}.
By Proposition \ref{sym-rel}, the Pareto {\sc spc} is $G$-symmetric and, by Proposition \ref{sym-cons2},
the Pareto $k$-{\sc scc} is $G$-consistent.

\section{Main results}\label{maingen}

Given a {\sc spc} or a {\sc $k$-scc}, our main purpose is to find a resolute refinement of it satisfying suitable symmetry and consistency properties. More precisely, we try to find an answer to the next three questions:
\begin{itemize}
\item Given a {\sc spc} $C$ and $U\le G$, can we find a resolute refinement of $C$ which is $U$-symmetric?
\item Given a {\sc spc} $C$ and $U\le G$, can we find a resolute refinement of $C$ which is $U$-consistent?
\item Given a $k$-{\sc scc} $C$ and $U\le G$, can we find a resolute refinement of $C$ which is $U$-consistent?
\end{itemize}
Indeed, for each of the above questions, we provide conditions on $C$ and $U$ which allow to give an affirmative answer.
We emphasize that, assuming $C$ to be trivial, the three questions above are about the existence of resolute $U$-symmetric {\sc spc}s, $U$-consistent {\sc spc}s and $U$-consistent {\sc $k$-scc}s.

In order to describe our main results, we need to recall the crucial concept of regular subgroup of $G$, introduced in
Bubboloni and Gori (2015).
A subgroup $U$  of $G$ is said to be {\it regular} if, for every $p\in\mathcal{P}$, there exists $\psi_*\in S_n$  conjugate
to $\rho_0$ such that
\begin{equation} \label{reg-def}
\mathrm{Stab}_U(p)\subseteq \left(S_h\times \{id\}\times \{id\}\right)\cup \left( S_h\times \{\psi_*\}\times \{\rho_0\}\right).
\end{equation}
where $\mathrm{Stab}_U(p)=\{(\varphi,\psi,\rho)\in U: p^{(\varphi,\psi,\rho)}=p\}$.

We also recall the fundamental result about those groups proved in Bubboloni and Gori (2015, Theorem 7):
\begin{equation} \label{teorema 7} \hbox{ There exists a } U\hbox{-symmetric resolute {\sc spc} if and only if } U \hbox{ is regular.}
\end{equation}
Thus, if we want to focus on $U$-symmetric resolute refinements of a given {\sc spc}, we necessarily have to assume $U$  regular.
The next result follows from some technical adjustments of the proof of Theorem $11$ in Bubboloni and Gori (2015). Its proof is proposed in Appendix \ref{app-C}.

\begin{theorem}\label{main3}
Let $U$ be a regular subgroup of $G$ and $C$ be a decisive and $U$-symmetric {\sc spc}. Then the three following conditions are equivalent:
\begin{itemize}
\item[$(i)$] $C$ admits a $U$-symmetric resolute refinement;
\item[$(ii)$] there exists an irreflexive and acyclic $U$-symmetric social method $R$ such that $C^R$ refines $C$;
\item[$(iii)$] there exists an irreflexive and acyclic social method $R$ such that $C^R$ refines $C$ and the following condition is satisfied:
\begin{itemize}
\item[$(a)$] for every $p\in\mathcal{P}$, $x,y\in N$ and $(\varphi,\psi,\rho_0)\in  \mathrm{Stab}_U(p)$, we have that $(x,y)\in R(p)$ if and only if $(\psi(y),\psi(x))\in R(p)$.
\end{itemize}
\end{itemize}
\end{theorem}
Note that, for every $U\le G$, there are decisive and $U$-symmetric {\sc spc}s. Indeed, the trivial, the Pareto, the Borda and the Kemeny  {\sc spc}s are decisive and $G$-symmetric so that they are $U$-symmetric too.

\begin{corollary}\label{corollary-newnew}
Let $U$ be a regular subgroup of $G$ with $U\le S_h\times S_n\times \{id\}$ and $C$ be a decisive and $U$-symmetric {\sc spc}. Then $C$ admits a $U$-symmetric resolute refinement.
\end{corollary}

\begin{proof}
For every $p\in\mathcal{P}$, we have $C(p)\neq\varnothing$. Pick $q_p\in C(p)$ and define $R:\mathcal{P}\to \mathbf{R}(N)$ setting, for every $p\in \mathcal{P}$, $R(p)=q_p\setminus \Delta$, where $\Delta=\{(x,x): x\in N\}.$
Then $R$ is an irreflexive and acyclic social method. Moreover, $C^R(p)=\{q\in\mathbf{L}(N): q_p\setminus \Delta\subseteq q\}=\{q_p\}$, so that $C^R$ refines $C$. Since $(a)$ trivially holds, we have that  condition (iii) in Theorem \ref{main3} is satisfied. Therefore, by Theorem \ref{main3}, $C$ admits a $U$-symmetric resolute refinement.
\end{proof}

The concept of regular group works properly even to analyse consistency.
Theorems \ref{main2}  below is proved in Appendix \ref{app-C} by widely extending the theory developed in Bubboloni and Gori (2016b) to the framework here considered.

\begin{theorem}\label{main2}
Let $U$ be a regular subgroup of $G$  and $C$ be a decisive and $U$-consistent {\sc spc}. Then $C$ admits a $U$-consistent resolute refinement.
\end{theorem}

Theorems  \ref{main1} below, also proved in Appendix \ref{app-C},  is a largely unexpected result. It establishes a deep link between the number $n$ of alternatives and the number $k$ of winners to be selected in order to make each decisive and $U$-consistent {\sc $k$-scc} admit  a $U$-consistent resolute refinement. It points out that there is a substantial difference between {\sc spc}s and {\sc $k$-scc}s with respect to the existence of $U$-consistent resolute refinements. Indeed, while for {\sc spc}s such existence is guaranteed for every $U$-consistent {\sc spc} $C$ once $U$ is regular (Theorem \ref{main2}), for  {\sc $k$-scc}s one has to require the fulfilment of further conditions in addition to the regularity of $U$.

\begin{theorem}\label{main1}
Let $U$ be a regular subgroup of $G$. Then the two following facts are equivalent:
\begin{itemize}
	\item[$(i)$] every decisive and $U$-consistent {\sc $k$-scc} admits a $U$-consistent resolute refinement;
	\item[$(ii)$] one of the following conditions is satisfied:
	\begin{itemize}
		\item[$(a)$] for every $(\varphi,\psi,\rho_0)\in U$, $\psi$ is not a conjugate of $\rho_0$,
		\item[$(b)$] $n\le 3$,
		\item[$(c)$] $k\in\{1,n-1\}$,
		\item[$(d)$] $n$ is even and $k$ is odd.
	\end{itemize}
\end{itemize}
\end{theorem}
Since condition $(a)$ in Theorem \ref{main1} certainly holds when $U\le S_h\times S_n\times \{id\},$ we easily obtain the following useful consequence.

\begin{corollary}\label{cor-newnew1}
Let $U$ be a regular subgroup of $G$, with $U\le S_h\times S_n\times \{id\}$,  and $C$ be a decisive and $U$-consistent {\sc $k$-scc}. Then $C$ admits a $U$-consistent resolute refinement.
\end{corollary}

Observe that condition $(a)$ in Theorem \ref{main1} can be satisfied also by
 regular subgroups  not included in $S_h\times S_n\times \{id\}$. Consider the group
$\{id\}\times\{id\}\times \Omega$ and note that, for every $p\in \mathcal{P}$, $\mathrm{Stab}_U(p)=\{(id,id,id)\}$\footnote{ Indeed $p^{(id, id, \rho_0)}=p$ implies $p_1\rho_0=p_1$, which by cancellation law in the symmetric group, gives the contradiction $\rho_0=id.$}.
Thus $\{id\}\times\{id\}\times \Omega$ is regular and, since $id$ is not conjugate to $\rho_0$, $(a)$ trivially holds.

We end the section by stating some simple results about the induced  {\sc $k$-scc}s.

\begin{proposition}\label{main4}
Let $U$ be a regular subgroup of $G$, with $U\le S_h\times S_n\times \{id\}$,  and $C$ be a decisive and $U$-consistent {\sc spc}. Then the induced  {\sc $k$-scc} $C_k$ admits a $U$-consistent resolute refinement.
\end{proposition}
\begin{proof} Since $U\le S_h\times S_n\times \{id\}$, we have that $C$ is $U$-symmetric and thus, by Proposition \ref{sym-cons2}, $C_k$ is a decisive $U$-consistent {\sc $k$-scc}. Thus, by Corollary \ref{cor-newnew1}, $C_k$ admits a $U$-consistent resolute refinement.
\end{proof}

\begin{proposition}\label{cor-induc}
Let $U\leq G$ and $C$ be a decisive {\sc spc}. If $C$ admits a resolute $U$-symmetric refinement, then $U$ is regular and the induced  {\sc $k$-scc} $C_k$ admits a $U$-consistent resolute refinement.
\end{proposition}
\begin{proof} Let $f$ be a resolute $U$-symmetric refinement of $C$. By \eqref{teorema 7}, $U$ is regular. Moreover, by Proposition \ref{sym-cons2}, $f_k$ is $U$-consistent. On the other hand, clearly, $f_k$ is  resolute and refines $C_k.$
\end{proof}

\section{Regular groups}\label{regular-gr}

Because of the results presented in Section \ref{maingen}, the importance of the regular subgroups of $G$ is evident.
 In this section we propose some theorems which provide a way to test whether a subgroup $U$ of $G$ of the type $V\times W\times \{id\}$ or $V\times W\times \Omega$, with $V\leq S_h$ and $W\leq S_n$, is regular or not. These types of groups are particularly relevant for applications. Regular groups of a different structure will be treated in Section
\ref{case-forth}.

We start with some preliminary definitions.
Let $m\in\mathbb{N}$ be fixed in this section. Consider  $\sigma\in S_m$. For every $x\in \ldbrack m \rdbrack$, the $\sigma$-orbit $
x^{\langle\sigma\rangle}$ of $x$ is defined by
$
x^{\langle\sigma\rangle}=\{\sigma^t(x)\in \ldbrack m \rdbrack: t\in \mathbb{N}\}.
$
It is well-known that $|x^{\langle\sigma\rangle}|=s$ where $s=\min\{t\in\mathbb{N}: \sigma^t(x)=x\}$. If $x$ is fixed by $\sigma,$ then $x^{\langle\sigma\rangle}=\{x\}.$
If instead $x$ is moved by $\sigma$ and the constituent of $\sigma$ moving $x$ is the cycle $(x_1\cdots x_s)$, then
$x^{\langle\sigma\rangle}=\{x_1,\dots, x_s\}.$
The set
$
O(\sigma)=\{x^{\langle\sigma\rangle}: x\in \ldbrack m \rdbrack\}
$
of the $\sigma$-orbits is a partition of $\ldbrack m \rdbrack$, and we denote its size by $r(\sigma)$. A system of representatives of  the $\sigma$-orbits is a set $\{x_1,\dots, x_{r(\sigma)}\}\in\mathbb{P}_{r(\sigma)}(\ldbrack m \rdbrack) $ such that $O(\sigma)=\{x_1^{\langle\sigma\rangle},\dots, x_{r(\sigma)}^{\langle\sigma\rangle}\}$.

Next we recall the well-known number theoretical concept of partition. A {\it partition} of $m$ is an unordered list $T=[m_1,\dots,m_r]$ where $r\in\mathbb{N}$, for every
$j\in \{1,\dots,r\}$, $m_j\in\mathbb{N}$ and $m=\sum_{j=1}^{r}m_j.$
The numbers $m_1,\ldots, m_r$ are called the terms of $T$. The set of partitions of $m$ is denoted by
$\Pi_m$.

 Consider $T\in \Pi_m$ and assume that $T$ admits $s\in\mathbb{N}$ distinct terms, say $m_1<\dots < m_s$, and assume that, for every
$j\in \ldbrack s \rdbrack$, $m_j$ appears $t_j\geq 1$ times in $T$. Then we use the notation $T=[m_1^{t_1},..., m_s^{t_s}]$ (where $t_j$ is omitted when it equals 1). We say that  $[m_1^{t_1},..., m_s^{t_s}]$ is the normal form of $T$ and that $t_j$ is the  multiplicity of $m_j$.
For instance, $[2,1,3,1]\in \Pi_7$ has normal form $[1^2, 2,3].$
We recall the well-known surjective function
\[
\mbox{$T:S_m\to \Pi_m$},\quad \sigma\mapsto T(\sigma)=\left[|x_1^{\langle\sigma\rangle}|,\ldots, |x_{r(\sigma)}^{\langle\sigma\rangle}|\right],
\]
where $\{x_1,\dots, x_{r(\sigma)}\}\in\mathbb{P}_{r(\sigma)}(\ldbrack m \rdbrack) $ is a system of representatives of the $\sigma$-orbits. For every $\sigma\in S_m$, $T(\sigma)$ is called the type of $\sigma$. In other words $T(\sigma)$ is the unordered list of the sizes of the orbits of the group generated by $\sigma$ on the set $\ldbrack m \rdbrack.$
Note  that the number of  terms equal to $1$ in $T(\sigma)$ counts the fixed points of $\sigma$ while the number of terms different from $1$ counts the constituents of $\sigma$. Moreover, $|\sigma|= \mathrm{lcm }(T(\sigma))$.  For instance, if $\sigma=(123)(456)(78)\in S_9$, then $r(\sigma)=4$, the type of $\sigma$ is $T(\sigma)=[1,2,3,3]\in \Pi_9$, $|\sigma|=\mathrm{lcm }[1,2,3,3]=6$ and a system of representatives of the $\sigma$-orbits is $\{1,4,7,9\}\in\mathbb{P}_{4}(\ldbrack 9 \rdbrack)$.
The  theoretical importance of the concept of type relies on the fact that two permutations are conjugate if and only if they have the same type.

Given $U\le S_m$, we define the {\it type number} of $U$ by
\begin{equation}\label{type-number}
\gamma(U)=\mathrm{lcm}\{\gcd(T(\sigma)): \sigma\in U\}.
\end{equation}
We describe some basic properties of the type number after having introduced some arithmetic notation. Let $x,y\in \mathbb{N}$. If $x$ divides $y$, we write $x\mid y.$ If $\pi\in \mathbb{N}$ is a prime number we denote by  $x_\pi=\max \{\pi^a:a\in\mathbb{N}_0,\, \pi^a\mid x\}$ the $\pi$-part of $x$.

\begin{lemma}\label{gamma-prop} Let $U,V\le S_m$. Then the following facts hold true.
 \begin{itemize}
 \item[$(i)$] If $U\leq V$, then $\gamma(U)\mid \gamma(V).$
\item[$(ii)$] $\gamma(U)\mid m$.
\item[$(iii)$] If $U$ contains an $m$-cycle, then $\gamma(U)=m$. In particular, $\gamma(S_m)= m$.
 \end{itemize}
\end{lemma}
\begin{proof}$(i)$ The set of integers $\{\gcd(T(\sigma)): \sigma\in U\}$ is included in the set of integers $\{\gcd(T(\sigma)): \sigma\in V\}$ and thus the least common multiple of the first divides that of the second.

$(ii)$ Let $\sigma\in U$ and let $T(\sigma)=[m_1,\dots, m_r]$. If $d=\gcd(T(\sigma))$, we have that $d\mid m_j$ for all $j\in \ldbrack r \rdbrack$ and since $\sum_{j=1}^r m_j=m$ we have $d\mid m$. Thus $m$ is a common multiple for the integers in $\{\gcd(T(\sigma)): \sigma\in U\}$, which implies $\gamma(U)\mid m.$

$(iii)$ Let $\sigma\in U$ be an $m$-cycle. Then $T(\sigma)=[m]$ and $\gcd(T(\sigma))=m$. Thus $m\mid \gamma(U)$. Since by $(ii)$ we also have $\gamma(U)\mid m$
we conclude that $\gamma(U)=m$.
\end{proof}

\begin{theorem}\label{reg-id}
Let $V\le S_h$ and $W\le S_n$. Then $V\times W\times \{id\}$ is regular if and only if
\[
\gcd(\gamma(V),|W|)=1.
\]
\end{theorem}

\begin{proof} Let  $U=V\times W\times \{id\}$. Assume first that $\gcd(\gamma(V),|W|)=1$. Assume further, by contradiction, that $U$ is not regular. Then, by Theorem 12 in Bubboloni and Gori (2015), there exist $(\varphi,\psi,id)\in U$ and a prime $\pi$ such that $|\psi|_\pi>1$ and $|\psi|_\pi \mid \gcd(T(\varphi))$. Then $\pi \mid \gamma(V)$ and, by Lagrange Theorem,  $\pi \mid |W|$ so that $\pi \mid\gcd(\gamma(V),|W|)=1$, a contradiction.

Assume next that $U$ is regular. We show that if $\pi $ is a prime dividing $|W|$, then $\pi\nmid \gamma(V)$. Let $\pi \mid |W|.$ Then, by Cauchy Theorem, there exists $\psi\in W$ with $|\psi|=\pi$. But, for every $\varphi\in V$, we have $(\varphi,  \psi,id)\in U$ and, of course, $|\psi|_{\pi}=\pi$. Thus, by Theorem 12 in Bubboloni and Gori (2015), we have that,  for every $\varphi\in V$, $\pi\nmid\gcd(T(\varphi))$ so that $\pi \nmid \mathrm{lcm}\{\gcd(T(\varphi)): \varphi\in V\}=\gamma(V)$.
\end{proof}

\begin{theorem}\label{reg-omega}
Let $V\le S_h$ and $W\le S_n$. Then $V\times W\times \Omega$ is regular if and only if
\[
\gcd(\gamma(V),\mathrm{lcm}(|W|,2))=1.
\]
\end{theorem}

\begin{proof} Let  $U=V\times W\times \Omega$. Assume first that $\gcd(\gamma(V),\mathrm{lcm}(|W|,2))=1$. Assume further, by contradiction, that $U$ is not regular. Then, by Theorem 12 in Bubboloni and Gori (2015), there exist $(\varphi,\psi,id)\in U$ and a prime $\pi$ such that $|\psi|_\pi>1$ and $|\psi|_\pi \mid \gcd(T(\varphi))$, or there exist $(\varphi,\psi,\rho_0)\in U$, with $\psi^2=id$ and $\psi$ not conjugate of $\rho_0,$ such that $2 \mid \gcd(T(\varphi)).$
In the first case, by Lagrange Theorem, we have $\pi\mid |W|$ as well as $\pi\mid\gamma(V)$ and thus $\pi\mid \gcd(\gamma(V),\mathrm{lcm}(|W|,2))=1$, a contradiction.
In the second case we have $2\mid \gamma(V)$, which implies the contradiction $2\mid \gcd(\gamma(V),\mathrm{lcm}(|W|,2))=1$.

Assume next that $U$ is regular. We show first that if $\pi$ is a prime dividing $|W|$, then $\pi\nmid \gamma(V)$. Let $\pi \mid |W|$. Then, by Cauchy Theorem, there exists $\psi\in W$ with $|\psi|=\pi$. But, for every $\varphi\in V$, we have $(\varphi,  \psi,id)\in U$ and, of course, $|\psi|_{\pi}=\pi$. Thus, by Theorem 12 in Bubboloni and Gori (2015), we get $\pi\nmid\gcd(T(\varphi))$ and so also $\pi \nmid \mathrm{lcm}\{\gcd(T(\varphi)): \varphi\in V\}=\gamma(V).$
We are then left with proving that $2\nmid \gamma(V)$, that is, that $2\nmid\gcd(T(\varphi))$ for all $\varphi\in V$. Let $\varphi\in V$ and consider $(\varphi, id, \rho_0)\in U.$ Since $id^2=id$ but $id$ is not a conjugate of $\rho_0$, by Theorem 12 in Bubboloni and Gori (2015), we get $2\nmid \gcd(T(\varphi))$.
\end{proof}

\section{Generalized anonymity and neutrality}\label{gen-an-neut}

Consider $V\le S_h$ and $W\le S_n$. Given $C\in \mathfrak{P}$ ($C\in\mathfrak{C}_k$), we say that $C$ is $V$-anonymous if $C$ is $V\times  \{id \} \times \{id \}$-consistent;
$W$-neutral if $C$ is $ \{id \} \times W \times \{id \}$-consistent.
Thus, $C$ is $V$-anonymous if permuting individual names according to permutations in $V$ has no effect on the final outcome; $C$ is $W$-neutral if the unique effect of permuting alternative names according to a permutation in $W$ is that alternative names are accordingly permuted in the final outcome.
Note that, the concepts of $S_h$-anonymity and $S_n$-neutrality correspond to the classical concepts of anonymity and neutrality, respectively.
 Note also that every $C\in \mathfrak{P}$ ($C\in\mathfrak{C}_k$) is $\{id\}$-anonymous and $\{id\}$-neutral.

We also stress that, by Propositions \ref{U,V} and \ref{U,V-corr},  $C\in \mathfrak{P}$ ($C\in \mathfrak{C}_k)$ is $V$-anonymous and  $W$-neutral if and only if $C$ is $V\times W\times \{id\}$-consistent;
$C\in \mathfrak{P}$ is $V$-anonymous, $W$-neutral and reversal symmetric if and only if $C$ is $V\times W\times \Omega$-symmetric; $C\in \mathfrak{P}$ ($C\in \mathfrak{C}_k)$ is $V$-anonymous, $W$-neutral and immune to the reversal bias if and only if $C$ is $V\times W\times \Omega$-consistent.

Below we provide some simple but very important consequences of the theory developed in the previous sections.

\begin{theorem}\label{application1}
Let $V\le S_h$, $W\le S_n$ and $C$ be a decisive, $V$-anonymous and $W$-neutral {\sc spc}. Then the two following conditions are equivalent:
\begin{itemize}
	\item[$(i)$] $C$ admits a $V$-anonymous and $W$-neutral resolute refinement;
\item [$(ii)$] $\gcd(\gamma(V),|W|)=1$.
\end{itemize}
\end{theorem}

\begin{proof}
$(i)\Rightarrow (ii)$. Assume that $(i)$ holds true. Then by  \eqref{teorema 7}, we have that the group $V\times W\times \{id\}$
is regular, so that, by Theorem \ref{reg-id}, $\gcd(\gamma(V),|W|)=1$.

$(ii)\Rightarrow (i)$. Assume that $(ii)$ holds true. Then, by Theorem \ref{reg-id}, the group $V\times W\times \{id\}$ is regular. Then we can apply Theorem \ref{main2}.
\end{proof}

Note that condition $\gcd(\gamma(V),|W|)=1$ above is trivially satisfied if one between $V$ and $W$ is trivial.
Observe also that the above theorem generalizes the classic result about the existence of a resolute anonymous and neutral {\sc spc}. Namely, taking $V=S_h$, $W=S_n$ and using Lemma \ref{gamma-prop}\,$(iii)$, we immediately have that there exists a resolute, anonymous and neutral {\sc spc} if and only if $\gcd(h,n!)=1$.

\begin{theorem}\label{application2}
Let  $V\le S_h$, $W\le S_n$ and $C$ be a decisive, $V$-anonymous, $W$-neutral and reversal symmetric {\sc spc}. Then the two following conditions are equivalent:
\begin{itemize}
	\item[$(i)$] $C$ admits a $V$-anonymous, $W$-neutral and reversal symmetric resolute refinement;
\item [$(ii)$] $\gcd(\gamma(V),\mathrm{lcm}(|W|,2))=1$ and there exists an irreflexive acyclic social method $R$ such that $C^R$ refines $C$ and, for $U=V\times W\times \Omega$, the following condition is satisfied:
\begin{itemize}
\item[$(a)$] for every $p\in\mathcal{P}$, $x,y\in N$ and $(\varphi,\psi,\rho_0)\in Stab_U(p)$, we have that $(x,y)\in R(p)$ if and only if $(\psi(y),\psi(x))\in R(p)$.
\end{itemize}
\end{itemize}
\end{theorem}

\begin{proof}
$(i)\Rightarrow (ii)$. Assume that $(i)$ holds true. Then by  \eqref{teorema 7}, we have that the group $U$
is regular, so that, by Theorem \ref{reg-omega}, $\gcd(\gamma(V),\mathrm{lcm}(|W|,2))=1$. Moreover, using now Theorem \ref{main3} we conclude the proof.

$(ii)\Rightarrow (i)$. Assume that $(ii)$ holds true. Then, by Theorem \ref{reg-omega}, we know that the group $U$ is regular. Therefore we can apply Theorem \ref{main3}.
\end{proof}

\begin{theorem}\label{application-spc}
Let $V\le S_h$, $W\le S_n$ and $C$ be a decisive, $V$-anonymous, $W$-neutral and immune to the reversal bias {\sc spc}. If
$\gcd(\gamma(V),\mathrm{lcm}(|W|,2))=1$, then
$C$ admits a $V$-anonymous, $W$-neutral and immune to the reversal bias resolute refinement.
\end{theorem}

\begin{proof} By Theorem \ref{reg-omega}, the group $U=V\times W\times \Omega$ is regular. Then
Theorem \ref{main2} applies.
\end{proof}

\begin{theorem}\label{application3}
Let $V\le S_h$, $W\le S_n$  and $C$ be a decisive, $V$-anonymous and $W$-neutral {\sc $k$-scc}. If
$\gcd(\gamma(V),|W|)=1$, then
$C$ admits a $V$-anonymous and $W$-neutral resolute refinement.
\end{theorem}

\begin{proof}
Apply Theorem \ref{reg-id} and  Corollary \ref{cor-newnew1}.
\end{proof}

\begin{theorem}\label{application4}
Let $V\le S_h$, $W\le S_n$ and assume that
$\gcd(\gamma(V),\mathrm{lcm}(|W|,2))=1$. Then the two following facts are equivalent:
\begin{itemize}
\item[$(i)$] every decisive, $V$-anonymous, $W$-neutral and immune to the reversal bias {\sc $k$-scc} admits a $V$-anonymous, $W$-neutral and immune to the reversal bias resolute refinement;
\item[$(ii)$] one of the following conditions is satisfied:
\begin{itemize}
		\item[$(a)$] for every $\psi\in W$, $\psi$ is not a conjugate of $\rho_0$,
		\item[$(b)$] $n\le 3$,
		\item[$(c)$] $k\in\{1,n-1\}$,
		\item[$(d)$] $n$ is even and $k$ is odd.
	\end{itemize}
\end{itemize}
\end{theorem}

\begin{proof}
Apply Theorem \ref{reg-omega} and Theorem \ref{main1}.
\end{proof}

\section{Some applications}\label{app}

Consider a committee having $h$ members which is facing the problem of selecting a set of $k$ alternatives within a set of $n$ alternatives.
The committee members need a suitable procedure for determining the $k$ alternatives from their preferences. If they agree to express their preferences via a linear order on the set of alternatives, they are in fact looking for a resolute $k$-{\sc scc}. Suppose now that there is a general agreement on the fact that the resolute $k$-{\sc scc} to be used should obey to certain fundamental principles which can be summarized in the requirement of being a refinement of a given decisive, anonymous, neutral and immune to the reversal bias $k$-{\sc scc} $C$.\footnote{Thus, $C$ might be, for instance, the trivial, the Pareto, the Borda or the Kemeny  $k$-{\sc scc}. Recall that, those $k$-{\sc scc} are
 $U$-consistent for all $U\le G$.}

In what follows we are going to discuss four particular qualifications of the above described situation
where the theory we developed can be applied and turns out to be useful to the committee.
We stress that the proposed applications only focus on $k$-{\sc scc}s. Suitable adaptations to the framework of {\sc spc}s are natural and left to the reader.

\subsection{Grouping individuals and alternatives }\label{case-1}

Assume that the particular structure of the committee naturally allows to group committee members into subcommittees and that the typology of the alternatives allows to group them into subclasses. Suppose that committee members also agree on the fact that the desired resolute refinement of $C$ should not distinguish among individuals in the same subcommittee and should equally treat alternatives in the same subclass.

Let $r,s\ge 1$ and assume that $H_1,\ldots,H_r\subseteq H$ is the list of subcommittees and that $N_1,\ldots, N_s\subseteq N$ is the list of subclasses so that
$\{H_1,\ldots,H_r\}$ is a partition of $H$
and  $\{N_1,\ldots, N_s\}$ is a partition of $N$. Then, the committee needs to find a resolute refinement of $C$ which is $V$-anonymous and $W$-neutral, where
\[
V=\{\varphi\in S_h: \forall j\in\ldbrack r \rdbrack,\;  \varphi(H_j)=H_j\},
\]
\begin{equation}\label{W-sub}
W=\{\psi\in S_n: \forall j\in\ldbrack s \rdbrack,\;  \psi(N_j)=N_j\}.
\end{equation}
Let us prove now that
\begin{equation}\label{direct-gam-1}
\gamma(V)=\gcd\left(|H_1|,\ldots, |H_r|\right)
\end{equation}
Indeed, consider $\sigma\in V$ and its type $T(\sigma).$ Let $x\in H_j$, for some $j\in\ldbrack r \rdbrack$. Since $\sigma(H_j)=H_j$, we have that $x^{\langle\sigma\rangle}\subseteq H_j$. Thus the sets $H_j$ are union of $\sigma$-orbits.
Hence, the partition $T(\sigma)$ of $h$ induces a partition of $|H_j|$ for all
$j\in \ldbrack r \rdbrack$. It follows that, for every  $j\in \ldbrack r \rdbrack$, $\gcd(T(\sigma))$ divides $|H_j|$ and thus
$\gcd(T(\sigma))$ divides $ \gcd\left(|H_1|,\ldots, |H_r|\right)$. Then we also have that  $\gamma(V)=\mathrm{lcm}\{\gcd(T(\sigma)): \sigma\in V\}$ divides $\gcd\left(|H_1|,\ldots, |H_r|\right)$. If $\gcd\left(|H_1|,\ldots, |H_r|\right)=1,$ then we also have $\gamma(V)=1=\gcd\left(|H_1|,\ldots, |H_r|\right)$. Assume then that
$\gcd\left(|H_1|,\ldots, |H_r|\right)\neq 1.$ Then, in particular, $|H_j|\neq 1$ for all $j\in \ldbrack r \rdbrack$ and it is meaningful to define a cycle
 $\sigma_j$ on its elements. Consider now $\sigma\in S_h$ defined by $\sigma=\sigma_1\cdots \sigma_r.$ Then $T(\sigma)=[|H_1|,\ldots,|H_r|]$ and $\gcd(T(\sigma))=\gcd\left(|H_1|,\ldots, |H_r|\right)$ divides $\gamma(V).$ Thus we obtain \eqref{direct-gam-1}.

Clearly, we also have
\begin{equation}\label{direct-1}
|W|=\prod_{j=1}^s |N_j|!
\end{equation}
Thus, by \eqref{direct-gam-1}, \eqref{direct-1} and Theorem \ref{application3}, we deduce that
if
\begin{equation}\label{condition1}
\gcd\left(|H_1|,\ldots, |H_r|,\prod_{j=1}^s |N_j|!\right)=1,
\end{equation}
then $C$ admits a $V$-anonymous and $W$-neutral resolute refinement.
Moreover, by \eqref{direct-gam-1}, \eqref{direct-1} and Theorem \ref{application4}, we deduce that
if \eqref{condition1} holds true, not all the subclasses are singletons and one of the following conditions is satisfied:
\begin{itemize}
		\item[$(a)$] there exist at least two subclasses of odd size,
		\item[$(b)$] $n\le 3$,
		\item[$(c)$] $k\in\{1,n-1\}$,
		\item[$(d)$] $n$ is even and $k$ is odd,
\end{itemize}
then $C$ admits a $V$-anonymous, $W$-neutral and immune to the reversal bias resolute refinement.

\subsection{A variation on groupings of individuals and alternatives}\label{sec-app}

Assume that the particular structure of the committee naturally allows to group committee members into subcommittees and that the typology of the alternatives allows to group them into subclasses as in Section \ref{case-1}. Suppose, this time, that committee members also agree on the fact that the desired resolute refinement of $C$ should not distinguish among individuals in the same subcommittee, should not distinguish among subcommittees having the same size and should equally treat alternatives in the same subclass.

Using the same notation used in Section \ref{case-1}, let $H_1,\ldots,H_r$ be the list of subcommittees and $N_1,\ldots, N_s\subseteq N$ be the list of subclasses. Then, the committee needs to find a resolute refinement of $C$ which is $V$-anonymous and $W$-neutral, where
\[
V=\{\varphi\in S_h: \forall j\in \ldbrack r \rdbrack, \exists k\in \ldbrack r \rdbrack\mbox{ such that }\varphi(H_j)=H_k\}.
\]
and $W$ is as in \eqref{W-sub}.
Defining now, for every $j\in \ldbrack r \rdbrack$,
\[
t_j=|\{k\in \ldbrack r \rdbrack: |H_k|=|H_j|\}|
\]
we claim that
\begin{equation}\label{direct-gam-1bis}
\gamma(V)=\gcd\left(|H_1|t_1,\ldots, |H_r|t_r\right)
\end{equation}
Indeed, consider $\sigma\in V$. Note that, since $\sigma$ is a bijection, for every $j\in \ldbrack r \rdbrack$, we have $|\sigma(H_j)|=|H_j|$ and thus $\sigma$ maps an element $x\in H_j$ into an element belonging to $U_j=\bigcup_{|H_i|=|H_j|} H_i.$
Thus every $U_j$  is a union of $\sigma$-orbits and the distinct sets $U_j$ give a partition of $H.$. Hence,  the partition $T(\sigma)$ of $h$ induces a partition of $|U_j|=|H_j|t_j$ for all $j\in \ldbrack r \rdbrack$. It follows that, for every  $j\in \ldbrack r \rdbrack$, $\gcd(T(\sigma))\mid |U_j$ and thus
$\gcd(T(\sigma))$ divides $\gcd\left(|U_1,\ldots, |U_r|\right)$ for all $\sigma\in V$. Then we also have that  $\gamma(V)=\mathrm{lcm}\{\gcd(T(\sigma)): \sigma\in V\}$ divides $\gcd\left(|U_1|,\ldots, |U_r|\right)$.
 If $\gcd\left(|U_1|,\ldots, |U_r|\right)=1,$ then we also have $\gamma(V)=1=\gcd\left(|U_1|,\ldots, |U_r|\right)$ and we have finished.
Assume then that  $\gcd\left(|U_1|,\ldots, |U_r|\right)\neq 1.$  Then, in particular, $|U_j|\neq 1$ for all $j\in \ldbrack r \rdbrack$ and it is meaningful to define a special cycle
 $\sigma_j$ on its elements in the following way. Let $H_{i_1},\dots H_{i_t}$ be the distinct sets of size $|H_j|=\ell\geq 1$ whose union is $U_j$ and order the elements inside each of them, say,  $H_{i_1}=\{h^1_{i_1},\dots , h^{\ell}_{i_1}\},\dots , H_{i_t}=\{h^1_{i_t},\dots , h^{\ell}_{i_t}\}$. Define $\sigma_j$ as the $|U_j|$-cycle on the elements of $U_j$ given by $\sigma_j=(h^1_{i_1}\ h^1_{i_2}\ \dots h^1_{i_t}\ h^2_{i_1}\ h^2_{i_2}\ \dots h^2_{i_t}\ \dots h^{\ell}_{i_1}\ h^{\ell}_{i_2}\ \dots h^{\ell}_{i_t})$. It is immediately checked that $\sigma_j(H_{i_{k}})=H_{i_{k+1}}$ for all $k\in \ldbrack t-1 \rdbrack$ while $\sigma_j(H_{i_{t}})=H_{i_{1}}$.
Consider now $\sigma\in S_h$ defined by $\sigma=\sigma_1\cdots \sigma_r.$ Then $\sigma\in V$ and $T(\sigma)=[|U_1|,\ldots,|U_r|]$ and $\gcd(T(\sigma))=\gcd\left(|U_1|,\ldots, |U_r|\right)$ divides $\gamma(V).$ Thus we obtain \eqref{direct-gam-1bis}.

Thus, by \eqref{direct-gam-1bis}, \eqref{direct-1} and Theorem \ref{application3}, we deduce that
if
\begin{equation}\label{condition3}
\gcd\left(|H_1|t_1,\ldots, |H_r|t_r,\prod_{j=1}^s |N_j|!\right)=1,
\end{equation}
then $C$ admits a $V$-anonymous and $W$-neutral resolute refinement.
Moreover, by \eqref{direct-gam-1bis}, \eqref{direct-1} and Theorem \ref{application4}, we deduce that
if \eqref{condition3} holds true, not all the subclasses are singletons and one of the following conditions is satisfied:
\begin{itemize}
		\item[$(a)$] there exist at least two subclasses of odd size,
		\item[$(b)$] $n\le 3$,
		\item[$(c)$] $k\in\{1,n-1\}$,
		\item[$(d)$] $n$ is even and $k$ is odd,
\end{itemize}
then $C$ admits a $V$-anonymous, $W$-neutral and immune to the reversal bias resolute refinement.

\subsection{Alternatives as a subset of individuals}\label{case-forth}

Assume that the committee is facing the problem of electing a subcommittee of size $k$ among those members of the committee who are running for the selection. In this particular situation we have that the alternatives to choose from are a subset of the set of individuals.
Suppose that committee members also agree on the fact that the desired resolute refinement of $C$ should not distinguish among individuals who are candidates, should not distinguish among individuals who are not candidates and should be immune to the reversal bias.

Assume then $n\leq h$ and name $1,\ldots,n$ the individuals running for the selection. Thus $N=\{1,\ldots,n\}\subseteq H$ is the set of
 alternatives to choose from and the procedure the committee is looking for corresponds to a $U$-consistent resolute refinement of $C$, where
\begin{equation}\label{diagonal}
U=\{(\varphi,\psi,\rho)\in G: \forall i\in\ldbrack n \rdbrack,\;  \varphi(i)=\psi(i)\}.
\end{equation}
Note that if $(\varphi,\psi,\rho)\in U$, then $\varphi(H\setminus N)=H\setminus N$. Note also that the above group $U$ is not necessarily regular. Consider, for instance, the case $h=4$ and $n=2$. Then $((1\,2)(3\, 4), (1\,2), id)\in U$ violates Theorem 12 in Bubboloni and Gori (2015). In the next proposition we see that, as usual, a suitable coprimality condition is equivalent to regularity.

\begin{proposition}\label{ind-alt} Let $n\leq h$. The group
$$U=\{(\varphi,\psi,\rho)\in G: \forall i\in\ldbrack n \rdbrack,\;  \varphi(i)=\psi(i)\}$$
is regular if and only if $\gcd(h,n)=1.$
\end{proposition}

\begin{proof} We first describe in a convenient way  the elements in $U.$ Let $(\varphi,\psi,\rho)\in U$
and let
\begin{equation}\label{shape-psi}
\psi=\prod_{i=1}^s\gamma_i,
\end{equation}
where the $\gamma_j$ are disjoint cycles, constituents of $\psi$ and $s\geq 0.$ The case $s=0$ is to be intended as taking the product over the empty set, and by definition is equal to $id.$
Since $\psi\in S_n$ is a permutation of the set $N\subseteq H$ and $\varphi\in S_h$ is a permutation of the set $H$ sharing with $\psi$ the same behaviour on $N$, we have that $\varphi=\psi\nu$ for some  permutation $\nu$ of the set $H\setminus N$. Thus there exist $r\geq 0$ disjoint cycles $\gamma_j$  on the set $H\setminus N$, with $j\in \{s+1,\dots, s+r\}$, such that
\begin{equation}\label{shape-phi}
\varphi=\prod_{i=1}^{s+r}\gamma_i
\end{equation}
and the above writing expresses $\varphi$ as product of its constituents. Note that if $r=0$, the set  $\{s+1,\dots, s+r\}$ is empty and no cycle has to be added.

We claim that the following property holds true:
\begin{equation}\label{property}
\hbox{if}\ (\varphi,\psi,\rho)\in U, \hbox{then}\ \gcd(T(\varphi))\mid \gcd(h,\gcd (T(\psi)))\mid \gcd(h,n).
\end{equation}
Indeed, let $(\varphi,\psi,\rho)\in U$. For $\psi$ we have the representation \eqref{shape-psi} with $s\geq 0$, and for $\varphi$ the representation \eqref{shape-phi} with $r\geq 0.$
If $\psi=id\in S_n$, then $\varphi$ admits at least one fixed point. Thus $\gcd(T(\psi))=1$ and  $\gcd(T(\varphi))=1$ and \eqref{property} is obvious. If $\psi\neq id$, we have $s\geq 1$ and
$$\gcd(T(\varphi))\mid \gcd(|\gamma_{i}|))_{i=1}^s=\gcd(T(\psi))\mid n.$$
On the other hand $\gcd(T(\varphi))\mid h,$ and thus \eqref{property} follows.
Assume now that $\gcd(h,n)=1$. Then, by \eqref{property}, we have $\gcd(T(\varphi))=1$.
In order to show that $U$ is regular, we show that the conditions $a)$ and $b)$ in Theorem 12 in Bubboloni and Gori (2015) are satisfied.
Let $(\varphi,\psi,id)\in U$, with $\psi\neq id$, and $\pi$ be a prime such that $|\psi|_{\pi}=\pi^a$, for some $a\in \mathbb{N}.$ Then obviously $\pi^a\nmid \gcd(T(\varphi))=1.$

Let next $(\varphi,\psi,\rho_0)\in U$, with $\psi^2=id$ and $\psi$ not a conjugate $\rho_0$. Then clearly $2\nmid \gcd(T(\varphi))=1.$
Assume now that $\gcd(h,n)\neq1$. Let $\pi$ be prime dividing $\gcd(h,n).$  Let $\psi\in S_n$ be a be a product of $n/\pi$ disjoint  $\pi$-cycles and $\varphi\in S_h$ be a product of $h/\pi$ disjoint $\pi$-cycles with  the first  $n/\pi$ of them equal to those forming $\psi.$ Then $(\varphi,\psi,id)\in U$ and $|\psi|_{\pi}=\pi\mid \gcd(T(\varphi))=\pi.$
Thus condition a) in Theorem 12 in Bubboloni and Gori (2015) is not satisfied. Hence $U$ is not regular.
\end{proof}

Thus, by Theorem \ref{main1}, we deduce that
if $\gcd(h,n)=1$ and one of the following conditions is satisfied:
\begin{itemize}
		\item[$(b)$] $n\le 3$,
		\item[$(c)$] $k\in\{1,n-1\}$,
		\item[$(d)$] $n$ is even and $k$ is odd,
\end{itemize}
then $C$ admits a $U$-consistent resolute refinement.\footnote{Note that condition $(a)$ in Theorem \ref{main1} cannot be satisfied because
$(\varphi,\rho_0,\rho_0)\in U$, for  any $\varphi\in S_h$ such that $\varphi(i)=\rho_0(i)$ for all $i\in \ldbrack n \rdbrack$.}

Note that the nature of the group $U$ defined by \eqref{diagonal} is very different from the groups of symmetry considered in Sections \ref{case-1} and \ref{sec-app}. In fact, $U$ cannot be expressed as a direct product of subgroups in $S_h$ and $S_n.$ Indeed, assume by contradiction that $U=V\times W\times \Omega$, for suitable $V\leq S_h$ and $W\leq S_n.$ Then for every $(\varphi,\psi,\rho)\in U$ we also have $(id,\psi,\rho)\in U$. But picking  $(\varphi,\psi,\rho)\in U$, with $\psi\neq id,$ we instead necessarily have  $\varphi\neq id$ and thus $(id,\psi,\rho)\notin U.$

\subsection{Committee members with diverse decision power}\label{order}

Assume that committee members have different decision power and that their decision power is described by an order on the set of individuals. Orders  model those situations in which it is reasonable to divide individuals into disjoint groups where all individuals have the same decision power and to rank such groups. Consider, for instance, a faculty committee composed by full and associate professors which is going to select some applicants for an academic position. There are in this case two well distinguished groups in the committee and a natural hierarchy between them ranking first the full professors and second the associate professors. Assume then that committee members agree to use a resolute refinement of
$C$ consistent with the decision power of individuals and which equally treat all the alternatives.

Consider then the relation $R\in\mathbf{O}(H)$ defined as follows: for every $x,y\in H$, we define $x\succeq_R y$ if and only if individual $x$ has a decision power which is at least as great as the one of individual $y$. The procedure the committee is looking for is a resolute refinement of $C$ which is $\mathrm{Aut}(R)$-anonymous and neutral, where $\mathrm{Aut}(R)$ is the group
\[
\mathrm{Aut}(R)=\{\sigma\in S_h: \forall x,y\in H,\, \sigma(x)\succeq_R\sigma(y) \mbox{ if and only if }x\succeq_R y\},
\]
called the group of automorphisms of $R$.

It is well-known that the relation
\[
I(R)=\{(x,y)\in H^2:(x,y)\in R \mbox{ and } (y,x)\in R\}
\]
is an equivalence relation in $H.$ Denote by $H_1,\ldots, H_r$, where $r\ge 1$, the equivalence classes of $I(R)$.
We claim that
\begin{equation}\label{direct}
\mathrm{Aut}(R)=
\{\varphi\in S_h: \forall j\in\ldbrack r \rdbrack,\;  \varphi(H_j)=H_j\}.
\end{equation}
Indeed, first order the sets $H_1,\ldots, H_r$ so that $x\in H_i$ and $y\in H_j$ with $i<j$ if and only if $x\succ_R y$.
The permutations in $K=\{\varphi\in S_h: \forall j\in\ldbrack r \rdbrack,\;  \varphi(H_j)=H_j\}$ clearly form a subgroup of $S_h.$
We show that  $K\leq \mathrm{Aut}(R)$. Pick $\varphi\in K$ and let $x,y\in H$ with $x\succeq_R y$. If we also have
$y\succeq_R x$, then $x$ and $y$ belong to the same $H_j$ and thus, by $K$ definition, $\varphi(x)$ and $\varphi(y)$ belong to $H_j$ too. Thus $\varphi(x)\succeq_R \varphi(y)$. If instead $y\not\succeq_R x$, then $x\succ_R y$. Thus we have $x\in H_i$  and $y\in H_{j}$
for some $i,j\in \ldbrack r \rdbrack$ with $i<j$. Thus, we have
$\varphi(x)\in \varphi(H_i)=H_i$ and $\varphi(y)\in \varphi(H_j)=H_j$ so that, $\varphi(x)\succ_R \varphi(y)$ and then, in particular,
$\varphi(x)\succeq_R \varphi(y)$.  Let now $x,y\in H$ with $\varphi(x)\succeq_R \varphi(y)$. Then, the same argument applies to $\varphi^{-1}$, giving $x\succeq_R y$.

We next show that $\mathrm{Aut}(R)\leq K.$ Let $\varphi\in \mathrm{Aut}(R)$ and  $j\in\ldbrack r \rdbrack.$ We  need to see that  $\varphi(H_j)=H_j$. To that purpose, since $\varphi$ is a bijection, it is enough to show that  $\varphi(H_j)\subseteq H_j.$ Pick $x\in H_j$. Since $H_1,\ldots, H_r$ are a partition of $H$, there exists $i\in\ldbrack r \rdbrack,$ such that $y=\varphi(x)\in H_i.$ Assume, by contradiction, that $i\neq j$. Then, we have that $x\succ_R y$ or $y\succ_R x.$ Assume that $x\succ_R y=\varphi(x)$. Since  $\varphi\in \mathrm{Aut}(R)$, applying  $\varphi$ to that relation, we get $\varphi(x)\succ_R \varphi^2(x)$ and thus, by transitivity $x\succ_R \varphi^2(x)$. Iterating this argument, we then find $x\succ_R \varphi^m(x)$, for every $m\in \mathbb{N}.$ Considering now $m$ such that $\varphi^m(x)=x,$ we get the contradiction $x\succ_R x.$
A similar argument shows the impossibility of $y\succ_R x.$

As a consequence,  by \eqref{direct-gam-1}, $\gamma(\mathrm{Aut}(R))=\gcd\left(|H_1|,\ldots, |H_r|\right)$. Thus, by Theorem \ref{application3},  we deduce that if
\begin{equation}\label{condition2}
\gcd\left(|H_1|,\ldots, |H_r|,|N|!\right)=1,
\end{equation}
then $C$ admits an $\mathrm{Aut}(R)$-anonymous, neutral and resolute refinement.
Moreover,
 by Theorem \ref{application4}, we deduce that
if \eqref{condition2} holds true and one of the following conditions is satisfied:
\begin{itemize}
		\item[$(b)$] $n\le 3$,
		\item[$(c)$] $k\in\{1,n-1\}$,
		\item[$(d)$] $n$ is even and $k$ is odd,
\end{itemize}
then $C$ admits an $\mathrm{Aut}(R)$-anonymous, neutral and immune to the reversal bias resolute refinement.\footnote{ Note that condition $(a)$ in Theorem \ref{application4} cannot be satisfied because obviously $S_n$ contains permutations which are conjugate of $\rho_0.$}

Assume now that the conditions for the existence of $\mathrm{Aut}(R)$-anonymous and neutral resolute refinements of $C$ are satisfied. In general, such refinements of $C$ are more than one. However, the information we get from $R$ goes further the mere knowledge of the indifference sets of $R$ and we can use this information for better selecting the resolute refinement.
In order to better explain this point, assume that the committee has a president, say individual 1, and assume that the president has a decision power which is greater than the one of all the other members of the committee. Assume further that all individuals but the president have the same decision power. Such a situation can be modelled by the order $R\in \mathbf{O}(H)$ defined by
\[
R=\{(1,y)\in H^2: y\in H\}\cup\{(x,y)\in H^2: x,y\in H\setminus \{1\}\}.
\]
In this case, the indifference classes of $I(R)$ are $H_1=\{1\}$ and $H_2=\{2,\ldots,h\}$ so that \eqref{condition2} holds and $C$ admits a resolute refinement which is $\mathrm{Aut}(R)$-anonymous and neutral. As already emphasised resolute refinements can be many and the problem of selecting one of them is certainly crucial. The special characteristics of the situation under consideration can be used to make this selection.
Assume that $k=1$. Then among those resolute refinements it is natural to choose the one obtained by letting the president break the ties by choosing his/her best alternative. This is exactly the classic tie-breaker method widely used in practical situations.
Another resolute refinement can be obtained, for instance, by selecting the alternative which is the worst one for the president. However,
this option cannot be considered consistent with the decision power among individuals and the role of the president.
Assuming now $k\geq 2$, there is no standard way to select a resolute refinement. However, looking at the particular characteristics of the decision problem some resolute refinements can be easily discarded and others can be reasonably taken into consideration.

\vspace{7mm}

\noindent {\Large{\bf References}}

\vspace{2mm}
\noindent Bredereck, R., Faliszewski, P., Igarashi, A., Lackner, M., Skowron, P., 2018. Multiwinner Elections with Diversity Constraints. Proceedings of the Thirty-Second AAAI Conference on Artificial Intelligence (AAAI-18), 933-940.
\vspace{2mm}

\noindent Bubboloni, D., Gori, M., 2014. Anonymous and neutral majority rules.
Social Choice and Welfare 43, 377-401.
\vspace{2mm}

\noindent Bubboloni, D., Gori, M., 2015. Symmetric majority rules. Mathematical Social Sciences 76, 73-86.
\vspace{2mm}

\noindent Bubboloni, D., Gori, M., 2016a. On the reversal bias of the Minimax social choice correspondence.
Mathematical Social Sciences 81, 53-61.
\vspace{2mm}

\noindent Bubboloni, D., Gori, M., 2016b. Resolute refinements of social choice correspondences.
Mathematical Social Sciences 84, 37-49.
\vspace{2mm}

\noindent Campbell, D.E., Kelly, J.S., 2011. Majority selection of one alternative from a binary
agenda. Economics Letters 110, 272-273.
\vspace{2mm}

\noindent Campbell, D.E., Kelly, J.S., 2013. Anonymity, monotonicity, and limited neutrality:
selecting a single alternative from a binary agenda. Economics Letters 118, 10-12.
\vspace{2mm}

\noindent Campbell, D.E., Kelly, J.S., 2015. The finer structure of resolute, neutral, and anonymous social choice
correspondences. Economics Letters 132, 109-111.
\vspace{2mm}

\noindent  Celis, L.E., Huang, L., Vishnoi, N.K., 2018. Multiwinner Voting with Fairness Constraints. Proceedings of the Twenty-Seventh International Joint Conference on Artificial Intelligence (IJCAI-18), 144-151.
\vspace{2mm}

\noindent Chichilnisky, G., 1980. Social choice and the topology of spaces of preferences. Advances in Mathematics 37, 165-176.
\vspace{2mm}

\noindent Do\u gan, O., Giritligil, A.E., 2015. Anonymous and neutral social choice: existence
results on resoluteness. Murat Sertel Center for Advanced Economic Studies,
Working Paper Series 2015–01.
\vspace{2mm}

\noindent E\u gecio\u glu, \" O., 2009. Uniform generation of anonymous and neutral preference profiles for social choice
rules. Monte Carlo Methods and Applications 15, 241-255.
\vspace{2mm}

\noindent E\u gecio\u glu, \" O., Giritligil, A.E., 2013. The Impartial, Anonymous, and Neutral Culture Model: A Probability Model for Sampling Public Preference Structures.
The Journal of Mathematical Sociology 37,  203-222.
\vspace{2mm}

\noindent Elkind, E., Faliszewski, P., Skowron, P., Slinko, A., 2017. Properties of Multiwinner Voting Rules. Social Choice and Welfare 48, 599-632.
\vspace{2mm}

\noindent Faliszewski, P., Skowron, P., Slinko, A., Talmon, N., 2017. Multiwinner Voting: A New Challenge for Social
Choice Theory. In Endriss, U., 2017, {\it Trends in Computational
Social Choice}, chapter 2, pages 27–47. AI Access Foundation.
\vspace{2mm}

\noindent Jacobson, N., 1974. {\it Basic Algebra I }. W.H. Freeman and Company, New York.
\vspace{2mm}

\noindent Kelly, J.S., 1991. Symmetry groups. Social Choice and Welfare 8, 89-95.
\vspace{2mm}

\noindent Lang, G., Skowron, P., 2016. Multi-Attribute Proportional Representation. Proceedings of the Thirtieth AAAI Conference on Artificial Intelligence (AAAI-16), 530-536.
\vspace{2mm}

\noindent Moulin, H., 1983. The Strategy of Social Choice. In: Advanced Textbooks in Economics, North Holland Publishing Company.
\vspace{2mm}

\noindent Moulin, H., 1988. Condorcets principle implies the no-show paradox. Journal of Economic Theory 45, 53-64.
\vspace{2mm}

\noindent Perry, J., Powers, R.C., 2008. Aggregation rules that satisfy anonymity and neutrality. Economics Letters 100, 108-110.
\vspace{2mm}

\noindent Powers, R.C., 2010. Maskin monotonic aggregation rules and partial anonymity. Economics Letters 106, 12-14.
\vspace{2mm}

\noindent Quesada, A., 2013. The majority rule with a chairman. Social Choice and Welfare 40, 679–691.
\vspace{2mm}

\noindent Saari, D.G., 1994. Geometry of Voting. In: Studies in Economic Theory, vol. 3. Springer.
\vspace{2mm}

\noindent Saari, D.G., Barney, S., 2003. Consequences of reversing preferences. The Mathematical Intelligencer 25, 17-31.

\newpage

\appendix

\noindent {\Large{\bf Appendix}}

\section{Symmetry and consistency}\label{sec-tec}

First of all, let us note that, for every $\sigma\in S_n$, $\rho\in\Omega$, $W_1,W_2\in \mathbb{P}(N)$, $\mathbb{W}_1,\mathbb{W}_2\subseteq  \mathbb{P}(N)$,
$R_1,R_2\in\mathbf{R}(N)$ and $\mathbf{Q}_1,\mathbf{Q}_2\subseteq \mathbf{R}(N)$, we have that
\begin{equation}\label{equal-new}
\begin{array}{l}
\sigma W_1\subseteq\sigma W_2 \Leftrightarrow W_1\subseteq W_2,\quad
\sigma \mathbb{W}_1\subseteq \sigma \mathbb{W}_2 \Leftrightarrow \mathbb{W}_1\subseteq \mathbb{W}_2,\\
\vspace{-2mm}\\
R_1\subseteq R_2\Leftrightarrow\sigma R_1\subseteq \sigma R_2\Leftrightarrow R_1\rho \subseteq R_2\rho,\quad
\mathbf{Q}_1\subseteq \mathbf{Q}_2\Leftrightarrow\sigma \mathbf{Q}_1\subseteq \sigma \mathbf{Q}_2\Leftrightarrow \mathbf{Q}_1\rho\subseteq \mathbf{Q}_2\rho.
\end{array}
\end{equation}
From those inclusions, analogous relations for equalities hold true. We freely use those properties in the rest of the paper. We emphasize their use when particularly remarkable.

Let $U\leq G.$ In Bubboloni and Gori (2015, Proposition 2) it is shown that the definition  \eqref{def-azione}, determines an action of $U$ on the set of preference profiles.
In particular, for every $p\in\mathcal{P}$ and $(\varphi_1,\psi_1,\rho_1),(\varphi_2,\psi_2,\rho_2)\in U$, we have
\begin{equation}\label{action-e}
p^{\,(\varphi_1\varphi_2,\psi_1\psi_2,\rho_1\rho_2)}= (p^{\,(\varphi_2,\psi_2,\rho_2)})^{(\varphi_1,\psi_1,\rho_1)}.
\end{equation}

The above equality is a main tool throughout the paper. To begin with it allows to prove the basic results on symmetry and consistency stated in Section \ref{scc}.

\begin{proof}[Proof of Proposition \ref{sym-cons}] $(i)$ Let $C\in \mathfrak{P}^{*U}$ and  $p\in \mathcal{P}$. Consider first $(\varphi,\psi,id)\in U$. Since $C$ is $U$-symmetric, then $C(p^{(\varphi,\psi,id)})=\psi C(p) id=\psi C(p)$ and \eqref{sccU1} is satisfied. In order to prove \eqref{sccU2}
assume, by contradiction, that $|C(p)|=1$ and that there exists $(\varphi,\psi,\rho_0)\in G$  such that
$C(p^{(\varphi,\psi,\rho_0)})=\psi  C(p).$ Then $C(p)=\{q\}$ for some $q\in\mathbf{L}(N)$ and,
by the $U$-symmetry of $C$, we get $\psi C(p)\rho_0=\psi  C(p).$ Thus, $C(p)\rho_0=C(p),$ that is, $q\rho_0=q$, which gives the contradiction $\rho_0=id.$

$(ii)$ From $(i)$ we know that $\mathfrak{P}^{*U}\subseteq \mathfrak{P}^{U}$. On the other hand if $C\in \mathfrak{P}^{U}$  we have that condition \eqref{sccU1}
is satisfied and, since $U\le S_h\times S_n\times \{id\}$, that equals condition \eqref{symU} so that $C\in \mathfrak{P}^{*U}.$
\end{proof}

\begin{proof}[Proof of Proposition \ref{U,V}] The fact that $\mathfrak{P}^{\langle U_1,U_2\rangle}\subseteq \mathfrak{P}^{*U_1}\cap \mathfrak{P}^{*U_2}$ is obvious. We show the other inclusion.
Let $C\in \mathfrak{P}^{*U_1}\cap \mathfrak{P}^{*U_2}$. Consider the set $$W=\{(\varphi,\psi,\rho)\in G: \forall p\in \mathcal{P}, C(p^{(\varphi,\psi,\rho)})=\psi C(p) \rho\}.$$
We show that $W$ is a subgroup of $G$. Let $(\varphi_1,\psi_1,\rho_1), (\varphi_2,\psi_2,\rho_2)\in W$ and we show that $(\varphi_1 \varphi_2,\psi_1 \psi_2,\rho_1\rho_2)\in W.$ Given $p\in \mathcal{P}$, by \eqref{action-e}  and recalling that $\Omega$ is abelian, we have
$$C(p^{(\varphi_1 \varphi_2,\psi_1 \psi_2,\rho_1\rho_2)})=C\left ((p^{\,(\varphi_2,\psi_2,\rho_2)})^{(\varphi_1,\psi_1,\rho_1)}\right)=\psi_1C(p^{\,(\varphi_2,\psi_2,\rho_2)})\rho_1=$$$$\psi_1\psi_2C(p)\rho_2\rho_1=\psi_1\psi_2 C(p)\rho_1\rho_2.$$
Since $W$ is a group and contains both $U_1$ and $U_2$, then we necessarily have $W\geq \langle U_1,U_2\rangle.$ But , by definition of $W$, $C\in \mathfrak{P}^{*W}$ and thus also $C\in \mathfrak{P}^{*\langle U_1,U_2\rangle}.$
\end{proof}

\begin{proof}[Proof of Proposition \ref{U,V-corr}]
The proof is formally the same of Proposition 10 in Bubboloni and Gori (2016b), simply taking into account \eqref{equal-new}. The interested reader can find the proof in Appendix \ref{final}.
\end{proof}

\begin{proof}[Proof of Proposition \ref{sym-cons2}]
Let $C\in \mathfrak{P}^{*U}$ and  $p\in \mathcal{P}$. Consider at first $(\varphi,\psi,id)\in U$. Since $C$ is $U$-symmetric then $C(p^{(\varphi,\psi,id)})=\psi C(p)$, that is, $C(p^{(\varphi,\psi,id)})=\{\psi q \in\mathbf{L}(N): q\in C(p)\}$.
Then
\[
C_k(p^{(\varphi,\psi,id)})=\Big\{\{(\psi q)(r):1\leq r\le k\} \in  \mathbb{P}_k(N): q\in C(p)\Big\}
\]
\[
=\Big\{\psi \{q(r):1\leq r\le k\}\in  \mathbb{P}_k(N): q\in C(p)\Big\}=\psi \Big\{\{q(r):1\leq r\le k\} \in  \mathbb{P}_k(N): q\in C(p)\Big\}=\psi C_k(p),
\]
so that $C_k$ satisfies \eqref{sccU1}. In order to show that $C_k$ satisfies \eqref{sccU2} too,
assume by contradiction that $|C_k(p)|=1$ and that there exists $(\varphi,\psi,\rho_0)\in G$  such that
$C_k(p^{(\varphi,\psi,\rho_0)})=\psi  C_k(p).$ Then there are $k$ distinct elements $x_1,\ldots, x_k$ of $N$ such that $C_k(p)=\{\{x_1,\dots, x_k\}\}$ and
\begin{equation}\label{absurd}
C_k(p^{(\varphi,\psi,\rho_0)})=\{\{\psi(x_1),\dots, \psi(x_k)\}\}.
\end{equation}
In particular, for every $q\in C(p),$ we have $\{q(1),\dots, q(k)\}=\{x_1,\dots, x_k\}$.
By the $U$-symmetry of $C$, we have $C(p^{(\varphi,\psi,\rho_0)})=\psi C(p)\rho_0$ and thus
\[
C_k(p^{(\varphi,\psi,\rho_0)})=\{\{q'(1),\dots, q'(k)\}\in\mathbb{P}_k(N) : q'\in \psi C(p)\rho_0 \}.
\]
Now if $q'\in \psi C(p)\rho_0$ there exists $q\in C(p)$ such that $q'=\psi q \rho_0$ and therefore $q'(1)=\psi q \rho_0(1)=\psi q(n).$ Since $n>k$ and $q$ is a bijection we have that
$q(n)\notin \{q(1),\dots, q(k)\}=\{x_1,\dots, x_k\}.$ Thus, since also $\psi$ is a bijection, we also have $q'(1)=\psi q(n)\notin \{\psi(x_1),\dots, \psi(x_k)\}.$ It follows that no element of $C_k(p^{(\varphi,\psi,\rho_0)})$ can coincide with $\{\psi(x_1),\dots, \psi(x_k)\},$ against \eqref{absurd}.
\end{proof}

\begin{proof}[Proof of Proposition \ref{sym-rel}] We show that,  for every $p\in\mathcal{P}$ and $(\varphi,\psi,\rho)\in U$, $C^R(p^{(\varphi,\psi,\rho)})=\psi C^R(p)\rho$.
Fix  $p\in\mathcal{P}$ and $(\varphi,\psi,\rho)\in U$ and prove first that $\psi C^R(p)\rho \subseteq C^R(p^{(\varphi,\psi,\rho)})$. Pick $q\in C^R(p)$ and show that
$\psi q\rho\in C^R(p^{(\varphi,\psi,\rho)}).$
We have $q\supseteq R(p)$, which implies $\psi q\rho \supseteq \psi R(p)\rho.$ Since $R$ is $U$-symmetric, we then have $\psi q\rho \supseteq R(p^{(\varphi,\psi,\rho)})$ and therefore $\psi q\rho\in C^R(p^{(\varphi,\psi,\rho)}).$

Let us prove now the other inclusion $C^R(p^{(\varphi,\psi,\rho)})\subseteq \psi C^R(p)\rho$. Let $\overline{p}=p^{(\varphi,\psi,\rho)}$ and note that $p=\overline{p}^{(\varphi^{-1},\psi^{-1},\rho)}$. Observe  that, since $U$ is a group and $(\varphi,\psi,\rho)\in U$ also its inverse $(\varphi,\psi,\rho)^{-1}=(\varphi^{-1},\psi^{-1},\rho)\in U.$
Thus, by the inclusion just proved, we get $\psi^{-1}C^R(\overline{p})\rho\subseteq C^R(\overline{p}^{(\varphi^{-1},\psi^{-1},\rho)})$, that is, $\psi^{-1}C^R(p^{(\varphi,\psi,\rho)})\rho\subseteq C^R(p)$ and we conclude applying $\psi$ on the left and $\rho$ on the right to both sides of that inclusion.
\end{proof}

\begin{proof}[Proof of Proposition \ref{U,V-method}]
Repeat word by word the proof of Proposition \ref{U,V} writing $\mathfrak{M}$ instead of $\mathfrak{P}$ and $R$ instead of $C$.
\end{proof}

\section{Proof of Theorems \ref{main3}, \ref{main2} and \ref{main1} }\label{app-C}

The proofs of Theorems \ref{main3}, \ref{main2} and \ref{main1} are definitely  technical and require some preliminary work. We underline that the results we are going to prove are more general as they provide a method to potentially build and count all the resolute refinements.

\subsection{The role of the orbit representatives}

Let $U\leq G$  and $p\in \mathcal{P}$. The set $p^U=\{p^g\in \mathcal{P}: g\in U\}$ is called the $U$-orbit of $p$ and
$\mathrm{Stab}_U(p)=\{g\in U : p^g=p \}\leq U$ is called the stabilizer of $p$ in $U$.
Recall that, for every  $g\in U$, we have $\mathrm{Stab}_U(p^{g})=\mathrm{Stab}_U(p)^g$.
The set  $\mathcal{P}^U=\{p^U:p\in\mathcal{P}\}$ of the $U$-orbits is a partition of $\mathcal{P}$.
We use $\mathcal{P}^U$ as set of indexes and denote its elements with $j$. A vector $(p^j)_{j\in\mathcal{P}^U}\in\times_{j\in \mathcal{P}^U}\mathcal{P}$
is called a system of representatives of the $U$-orbits if, for every $j\in\mathcal{P}^U$, $p^j\in j$.
The set of the systems of representatives of the $U$-orbits is denoted by $\mathfrak{S}(U)$.
If $(p^j)_{j\in \mathcal{P}^U}\in \mathfrak{S}(U)$, then, for every $p\in \mathcal{P}$, there exist $ j\in\mathcal{P}^U$ and $(\varphi,\psi,\rho)\in U$ such that $p=p^{j\,(\varphi,\psi,\rho)}$. Note that if $p^{j_1\,(\varphi_1,\psi_1,\rho_1)}=p^{j_2\,(\varphi_2,\psi_2,\rho_2)}$ for some $j_1,j_2\in \mathcal{P}^U$ and some $(\varphi_1,\psi_1,\rho_1), (\varphi_2,\psi_2,\rho_2)\in U$, then $j_1= j_2$ and, by \eqref{action-e},
$(\varphi_2^{-1}\varphi_1,\psi_2^{-1}\psi_1,\rho_2^{-1}\rho_1)\in \mathrm{Stab}_U(p^{j_1})$.

\bigskip

In this section we present some results explaining how a correspondence $C\in\mathfrak{P}^{*U}\cup \mathfrak{P}^{U}\cup\mathfrak{C}_k^{U}$ is determined by the values it assumes on a system of representatives of the $U$-orbits in $\mathcal{P}.$ To that purpose, the first step is to split $\mathcal{P}^U$ into two parts $\mathcal{P}_1^U$ and $\mathcal{P}_2^U$, where
\begin{eqnarray*}
&&\mathcal{P}_1^U=\left\{j\in\mathcal{P}^U: \forall p\in j,\, \mathrm{Stab}_U(p)\le S_h\times S_n\times \{id\}\right\},\\
&&\mathcal{P}_2^U=\left\{j\in\mathcal{P}^U: \forall p\in j, \,\mathrm{Stab}_U(p)\not \le S_h\times S_n\times \{id\}\right\}.
\end{eqnarray*}
Note that those sets are well defined because $U\cap (S_h\times S_n\times \{id\})$ is normal in $U$. Moreover  $\mathcal{P}_1^U\cup\mathcal{P}_2^U=\mathcal{P}^U$ and $\mathcal{P}_1^U\cap\mathcal{P}_2^U=\varnothing$.
In particular, $\mathcal{P}_1^U$ and $ \mathcal{P}_2^U$ cannot be both empty. Obviously, if $U\leq S_h\times S_n\times \{id\}$, then $\mathcal{P}_2^U=\varnothing$ and $\mathcal{P}^U=\mathcal{P}_1^U\neq \varnothing.$
We recall a part of Proposition 24 in  Bubboloni and Gori (2016b) which is interesting for our scope\footnote{Further details
can be found in Section 5.2 of Bubboloni and Gori (2016b).}:
\begin{equation}\label{PU2vuoto}
\mathcal{P}_2^U \neq \varnothing  \hbox{ if and only if there exists } (\varphi,\psi,\rho_0)\in U   \hbox{ such that } \psi   \hbox{ is a conjugate of }  \rho_0.
\end{equation}

\begin{proposition}\label{rappresentanti1} Let $U\leq G$ and $C,C'\in \mathfrak{P}^{*U}$. Assume that
there exists $(p^j)_{j\in\mathcal{P}^U }\in \mathfrak{S}(U)$ such that, for every  $j\in \mathcal{P}^U$, $C(p^j)=C'(p^j)$. Then $C=C'$.
\end{proposition}

\begin{proof} Let $p\in \mathcal{P}$ and show that $C(p)=C'(p)$. We know there exist $j\in\mathcal{P}^U$ and $(\varphi,\psi,\rho)\in U$ such that $p=p^{j\,(\varphi,\psi,\rho)}$.
Then, $C(p)=C(p^{j\,(\varphi,\psi,\rho)})=\psi C(p^{j})\rho=\psi C'(p^{j})\rho=C'(p^{j\,(\varphi,\psi,\rho)})=C'(p)$.
\end{proof}

\begin{proposition}\label{rappresentanti2} Let $U\leq S_h\times S_n \times \{id\}$  and
 $C,C'\in \mathfrak{P}^U$ ($C,C'\in \mathfrak{C}_k^U$). Assume that there exists $(p^j)_{j\in\mathcal{P}^U }\in \mathfrak{S}(U)$
 such that $C(p^j)=C'(p^j)$ for all $j\in \mathcal{P}^U$. Then $C=C'$.
\end{proposition}

The proof of the above result is the same as the one of Proposition 12 in Bubboloni and Gori (2016b). The interested reader can be find that proof in Appendix \ref{final}.

\begin{proposition}\label{rappresentanti3} Let  $U\leq G$ such that $U\not\le S_h\times S_n\times \{id\}$,
 $C,C'\in \mathfrak{P}^U$ ($C,C'\in \mathfrak{C}_k^U$). Assume that there exist $(p^j)_{j\in\mathcal{P}^U }\in \mathfrak{S}(U)$ and $(\varphi_*,\psi_*,\rho_0)\in U$
 such that $C(p^j)=C'(p^j)$ for all $j\in \mathcal{P}^U$ and $ C(p^{j\,(\varphi_*,\psi_*,\rho_0)})=C'(p^{j\,(\varphi_*,\psi_*,\rho_0)})$ for all $j\in \mathcal{P}_1^U$. Then $C=C'.$
\end{proposition}

The proof of the above result is formally the same as the one of Proposition 13 in Bubboloni and Gori (2016b). The interested reader can be find that proof in Appendix \ref{final}.

\subsection{Resolute {\sc spc}s and {\sc $k$-scc}s }

We have defined $C\in\mathfrak{P}$ ($C\in\mathfrak{C}_k$) resolute if, for every
$p\in\mathcal{P}$, $|C(p)|=1$. We now denote by $\mathfrak{F}$ ($\mathfrak{F}_k$) the set of resolute {\sc spc} ({\sc $k$-scc}). Obviously, we have $\mathfrak{F}\subseteq \mathfrak{P}$ and $\mathfrak{F}_k\subseteq \mathfrak{C}_k$ and, for every $U\leq G$, we can consider the following sets
\[
\mathfrak{F}^{*U}=\mathfrak{F}\cap \mathfrak{P}^{*U},\qquad \mathfrak{F}^U=\mathfrak{F}\cap \mathfrak{P}^U,\qquad \mathfrak{F}_k^U=\mathfrak{F}_k\cap  \mathfrak{C}_k^U.
\]
They describe, respectively, the set of the $U$-symmetric resolute {\sc spc},  the set of the $U$-consistent resolute {\sc spc}, the set of the $U$-consistent resolute {\sc $k$-scc}.

 Given $C\in \mathfrak{F}$, for every $p\in \mathcal{P}$, there exists a unique $q\in \mathbf{L}(N)$ such that $C(p)=\{q\}.$ Thus $C$ can be naturally identified with the {\it social preference function} ({\sc spf}) $f$ from $\mathcal{P}$ to $\mathbf{L}(N)$ defined, for every $p\in\mathcal{P}$,  by $f(p)=q.$
Similarly, given $C\in \mathfrak{F}_k$, for every $p\in \mathcal{P}$, there exists a unique $W\in \mathbb{P}_k(N)$ such that $C(p)=\{W\}.$ Thus $C$ can be naturally identified with the {\it $k$-multiwinner social choice function} ({\sc $k$-scf}) $f$ from $\mathcal{P}$ to $\mathbb{P}_k(N)$ defined, for every $p\in\mathcal{P}$,  by $f(p)=W.$
We will freely adopt those identifications and the language of functions in what follows. For that reason we will refer to $\mathfrak{F}^{*U}$ also as the set of the $U$-symmetric
{\sc spf}s; to $\mathfrak{F}^U$ also as the set of the $U$-consistent {\sc spf}s; to  $\mathfrak{F}_k^U$ also as the set of the $U$-consistent {\sc $k$-scf}s.

Let $C\in\mathfrak{P}$ ($C\in\mathfrak{C}_k$).  Denote the refinements of $C$ by $\mathfrak{P}_C$ ($\mathfrak{C}_{k,C}$). Then the resolute refinements of $C$ are the functions in the set $\mathfrak{F}_C=\mathfrak{F}\cap \mathfrak{P}_C$ ($\mathfrak{F}_{k,C}=\mathfrak{F}_k\cap \mathfrak{C}_{k,C}$).

Let now $U\leq G$. Given  $C\in\mathfrak{P}$, we consider the following sets
$$\mathfrak{F}^{*U}_C=\mathfrak{F}^{*U}\cap\mathfrak{F}_C,\qquad \mathfrak{F}_C^U=\mathfrak{F}^{U}\cap\mathfrak{F}_C.$$
They describe, respectively, the set of the $U$-symmetric {\sc spf}s which are refinements of $C$ and the set of the $U$-consistent {\sc spf}s which are refinements of $C.$
Given $C\in\mathfrak{C}_k$, we finally consider the set
$$\mathfrak{F}_{k,C}^U=\mathfrak{F}_{k}^U\cap\mathfrak{F}_{k,C}$$
describing the set of $U$-consistent {\sc $k$-scf}s which are refinements of $C$.

\subsection{Proof of Theorem \ref{main3}}
In order to approach the proof of Theorem \ref{main3} we need some preliminary facts.
Let $U\leq G$  and $S\in\mathfrak{P}$ be defined, for every $p\in\mathcal{P}$, by
\[
S(p)=\big\{q\in \mathbf{L}(N): \forall (\varphi,\psi,\rho)\in \mathrm{Stab}_U(p),\, \psi q \rho=q\big\}.\footnote{The {\sc spc} $S$ was first considered in Bubboloni and Gori (2015, Section 4.1) and there denoted by $S_1^U$.}
\]

\begin{proposition}\label{s} Let $U\leq G$. Then the following facts hold:
\begin{itemize}
\item[$(i)$] $S$ is decisive if and only if $U$ is regular.
\item[$(ii)$] $S\in \mathfrak{P}^{*U}.$
\item[$(iii)$] If $f\in \mathfrak{F}^{*U}$, then $f\in \mathfrak{F}_S$.
\end{itemize}
\end{proposition}

\begin{proof}$(i)$ Let $U$ be regular. Proving \eqref{teorema 7},  in Bubboloni and Gori (2015), the authors showed  that
\[
S(p)
=\left\{
\begin{array}{ll}
\mathbf{L}(N) & \mbox{if }\;\mathrm{Stab}_U(p)\leq S_h\times\{id\}\times\{id\}\\
\vspace{-2mm}\\
u  C_{S_n}(\rho_0)     & \mbox{if }\;\mathrm{Stab}_U(p)\not\leq S_h\times\{id\}\times\{id\},\\
\end{array}
\right.
\]
where $C_{S_n}(\rho_0)$ denotes the centralizer of $\rho_0$ in $S_n$ and $u\in S_n$ is such that  $\psi_*=u\rho_0 u^{-1}$, with $\psi_*\in S_n$ the permutation appearing in the definition \eqref{reg-def} of regularity.
In particular,
 \begin{equation}\label{ordS}
|S(p)|
=\left\{
\begin{array}{ll}
n! & \mbox{if }\;\mathrm{Stab}_U(p)\leq S_h\times\{id\}\times\{id\}\\
\vspace{-2mm}\\
  2^{\lfloor \frac{n}{2}\rfloor} \lfloor \frac{n}{2}\rfloor !     &   \mbox{if }\;\mathrm{Stab}_U(p)\not\leq S_h\times\{id\}\times\{id\}.\\
\end{array}\right.
\end{equation}
Thus the fact that $S$ is decisive immediately follows from \eqref{ordS}.

Assume now that $S$ is decisive. We need to show that, for every $p\in \mathcal{P}$, $(\varphi, \psi, id)\in \mathrm{Stab}_U(p)$ implies $\psi=id$, and $(\varphi, \psi, \rho_0)\in \mathrm{Stab}_U(p)$ implies $\psi=\psi_*$ for a suitable unique conjugate $\psi_*$ of $\rho_0.$
 Let $p\in \mathcal{P}$ and pick $q_0\in S(p)$. If $(\varphi, \psi, id)\in \mathrm{Stab}_U(p)$, then we have $\psi q_0=q_0$ and thus, by cancellation, $\psi =id.$ If $(\varphi, \psi, \rho_0)\in \mathrm{Stab}_U(p)$, then  we have $\psi q_0 \rho_0=q_0$, which implies $\psi=q_0\rho_0q_0^{-1}$ (recall that the linear order $q_0$ is identified with an element of $S_n$). Thus, $\psi_*=q_0\rho_0q_0^{-1}$ works.

$(ii)$ In order to show that $S$ is $U$-symmetric, we pick $(\varphi_1,\psi_1,\rho_1)\in U$ and see that
$S(p^{(\varphi_1,\psi_1,\rho_1)})=\psi_1S(p)\rho_1.$ Recall that
$\mathrm{Stab}_U(p^{(\varphi_1,\psi_1,\rho_1)})= \mathrm{Stab}_U(p)^{(\varphi_1,\psi_1,\rho_1)}.$
Thus, $q\in S(p^{(\varphi_1,\psi_1,\rho_1)})$ if and only if, for every $(\varphi,\psi,\rho)\in \mathrm{Stab}_U(p)$, $\psi_1\psi\psi_1^{-1}q\rho_1\rho\rho_1^{-1}=q$, which is equivalent to $\psi(\psi_1^{-1}q\rho_1)\rho=\psi_1^{-1}q\rho_1,$ that is, to $\psi_1^{-1}q\rho_1\in S(p)$ and thus to $q\in \psi_1S(p)\rho_1.$

$(iii)$ This is just Lemma 4 in Bubboloni and Gori (2015).
\end{proof}

By the above proposition, any $U$-symmetric {\sc spf} maps the profile $p$ into an element of $S(p)$. The next result shows that, conversely, one can construct $f\in \mathfrak{F}^{*U}$ just fixing a system  $(p^j)_{j\in \mathcal{P}^U}$ of representatives and assigning within $S(p^j)$ the value to be assumed on each $p^j$. Its proof  is formally equal to Proposition 5 in Bubboloni and Gori (2015) and can be find in Appendix \ref{final}.

\begin{proposition}\label{existence} Let $U\leq G$ be regular and  $(p^j)_{j\in \mathcal{P}^U}\in\mathfrak{S}(U)$. For every $j\in\mathcal{P}^U$, let $q_j\in S(p^j)$.
Then there exists a unique $f\in \mathfrak{F}^{*U}$ such that,
for every $j\in\mathcal{P}^U$, $f(p^j)=q_j$.
\end{proposition}

Given now $(p^j)_{j\in \mathcal{P}^U}\in\mathfrak{S}(U)$, let $$\Phi: \times_{j\in\mathcal{P}^U} S(p^j)	\to \mathfrak{F}^{*U}$$
be the function which associates with every  $(q_j)_{j\in \mathcal{P}^U}\in \times_{j\in\mathcal{P}^U} S(p^j)$ the unique $f\in\mathfrak{F}^{*U}$ defined in Proposition \ref{existence}.
Of course, $\Phi$ depends on $(p^j)_{j\in \mathcal{P}^U}$ but we do not emphasize that dependence in the notation. Note that $\Phi$ is injective.

\begin{theorem}\label{main}
Let $U\le G$ be regular, $(p^j)_{j\in \mathcal{P}^U}\in\mathfrak{S}(U)$ and $C\in\mathfrak{P}^{*U}.$ Then
\[
\mathfrak{F}^{*U}_C=\Phi\left( \times_{j\in\mathcal{P}^U} S(p^j)\cap C(p^j)\right)
\]
and
\[
|\mathfrak{F}^{*U}_C|=\prod_{j\in\mathcal{P}^U} |S(p^j)\cap C(p^j)|.
\]
\end{theorem}

\begin{proof}We first prove that $\Phi\left( \times_{j\in\mathcal{P}^U} S(p^j)\cap C(p^j)\right)\subseteq \mathfrak{F}^{*U}_C$.
Let $(q_j)_{j\in\mathcal{P}^U}\in \times_{j\in\mathcal{P}^U} S(p^j)\cap C(p^j)$ and $f=\Phi\left((q_j)_{j\in\mathcal{P}^U}\right)$. We show that
$f\in\mathfrak{F}^{*U}_C$. We know that $f\in \mathfrak{F}^{*U}$, so that have are left with showing that $f\in \mathfrak{F}_C$.
Given $p\in\mathcal{P}$, there exist $j\in \mathcal{P}^U$ and $(\varphi,\psi,\rho)\in U$ such that $p=p^{j\,(\varphi,\psi,\rho)}$. As we know that $q_j=f(p^j)\in   C(p^j)$,
by $U$-symmetry of $f$ and $C$, we have
$$
f(p)=f(p^{j\,(\varphi,\psi,\rho)})=\psi q_j \rho\in  \psi  C(p^j)\rho = C(p^{j\,(\varphi,\psi,\rho)})=C(p)
$$
as desired.

Let us next prove that $\mathfrak{F}^{*U}_C   \subseteq   \Phi\left( \times_{j\in\mathcal{P}^U} S(p^j)\cap C(p^j)\right)$.
Consider then $f\in \mathfrak{F}^{*U}_C$ and note that, by Proposition \ref{s}(iii), for every $j\in \mathcal{P}^U$, $f(p^j)\in S(p^j)\cap C(p^j)$. Then
$( f(p^j))_{j\in\mathcal{P}^U}\in  \times_{j\in\mathcal{P}^U} S(p^j)\cap C(p^j).$ Thus the function $f$ and the function $\Phi\left(( f(p^j))_{j\in\mathcal{P}^U}\right)$ are $U$-symmetric functions which coincide on a system of representatives. Hence, by Proposition \ref{rappresentanti1}, we obtain  $f=\Phi\left(( f(p^j))_{j\in\mathcal{P}^U}\right)$.

The last part of the statement is an immediate consequence of the fact that $\Phi$ is injective.
\end{proof}

In order to write down the proof of Theorem \ref{main3} we need a final technical lemma.
\begin{lemma}\label{for-their} Let $R\in \mathfrak{M}^{*U}$ and $p\in \mathcal{P}$. Then, for every $x,y\in N$ and $(\varphi,\psi,\rho_0)\in \mathrm{Stab}_U(p)$, $(x,y)\in R(p)$ if and only if $(\psi(y),\psi(x))\in R(p)$.
\end{lemma}
\begin{proof} Let $x,y\in N$ and $(\varphi,\psi,\rho_0)\in \mathrm{Stab}_U(p)$. Then, by the $U$-symmetry of $R$,  we have $R(p)=R(p^{(\varphi,\psi,\rho_0)}=\psi R(p)\rho_0.$ On the other hand, $x\succeq _{R(p)} y$ is equivalent to $\psi(y)\succeq _{\psi R(p)\rho_0}\psi(x)$ and thus to $\psi(y)\succeq _{R(p)} \psi(x)$.
\end{proof}

\begin{proof}[Proof of Theorem \ref{main3}] $(i)\Rightarrow (ii)$. The fact that $C\in \mathfrak{P}^{*U}$ admits a $U$-symmetric resolute refinement means that $\mathfrak{F}^{*U}_C\neq \varnothing$. Let then $f\in \mathfrak{F}^{*U}_C$ and define the social method $R:\mathcal{P}\to \mathbf{R}(N)$, by $R(p)=f(p)\setminus \Delta$
for all $p\in\mathcal{P}$, where $\Delta=\{(x,x): x\in N\}.$
Note that $\Delta$ is a  relation on $N$ and that
for every $\psi\in S_n$ and $\rho\in \Omega$, we have $\psi \Delta\rho=\Delta$, accordingly to the definitions given in Section \ref{plo}. We show that $R$ is irreflexive, acyclic, $U$-symmetric and that $C^R$ refines $C.$

Let $p\in \mathcal{P}.$ $R(p)$ is irreflexive by definition and surely acyclic since it refines the linear order $f(p)$.   Let $(\varphi, \psi, \rho)\in U$. Then, by the $U$-symmetry of $f$, we have
\[
R(p^{(\varphi, \psi, \rho)})=f(p^{(\varphi, \psi, \rho)})\setminus\Delta=(\psi f(p) \rho) \setminus(\psi \Delta\rho)=\psi(f(p)\setminus \Delta)\rho=\psi R(p)\rho.
\]
Thus $R\in\mathfrak{M}^{*U}$. Finally, observe that
\[
C^R(p)=\{q\in\mathbf{L}(N): R(p)\subseteq q\}=\{q\in\mathbf{L}(N): f(p)\setminus \Delta\subseteq q\}=\{f(p)\}\subseteq C(p).
\]
Thus $C^R$ refines $C.$

$(ii)\Rightarrow (iii)$ Let  $R$ be an irreflexive, acyclic, $U$-symmetric social method  such that $C^R$ refines $C$. Then, by Lemma \ref{for-their}, condition $(a)$  in (iii) is fulfilled.

$(iii)\Rightarrow (i)$ The proof is a remake of Bubboloni and Gori (2015, Section 8), which is essentially obtained replacing the minimal majority relation $R^{\nu(p)}(p)$ there, with the present relation $R(p)$. The interested reader can find the proof in Appendix \ref{final}.
\end{proof}

\subsection{Analysis of $\mathfrak{F}_C^U$ and $\mathfrak{F}_{k,C}^U$ when $U\leq S_h\times S_n\times \{id\}$}

\begin{proposition}\label{f-fu}
Let $U\le S_h\times S_n\times \{id\}$ be regular, $(p^j)_{j\in\mathcal{P}^U}\in\mathfrak{S}(U)$ and $C\in\mathfrak{P}^U$ $(C\in\mathfrak{C}^U_k)$. For every $j\in\mathcal{P}^U$, let $x_j\in C(p^j)$.
Then there exists a unique $f\in\mathfrak{F}^U_C$ ($f\in\mathfrak{F}^U_{k,C}$) such that, for every $j\in\mathcal{P}^U$, $f(p^j)=x_j$.
\end{proposition}

The proof of the above result is similar to the one of Proposition 17 in Bubboloni and Gori (2016b) and can be find in Appendix \ref{final}.

Let $U\le S_h\times S_n\times \{id\}$ be regular, $(p^j)_{j\in \mathcal{P}^U}\in\mathfrak{S}(U)$ and $C\in\mathfrak{P}^U$ ($C\in\mathfrak{C}_k^U$). Consider the function
\[
\Phi: \times_{j\in\mathcal{P}^U} C(p^j)	\to \mathfrak{F}^{U}_C,\quad
\left(
\Phi: \times_{j\in\mathcal{P}^U} C(p^j)	\to \mathfrak{F}^{U}_{k,C}
\right)
\]
which associates with every
$(x_j)_{j\in \mathcal{P}^U}\in \times_{j\in\mathcal{P}^U} C(p^j)$ the unique $f\in\mathfrak{F}^U_{C}$ ($f\in\mathfrak{F}^U_{k,C}$) defined in Proposition \ref{f-fu}.
Of course, $\Phi$ depends on $U$, $(p^j)_{j\in \mathcal{P}^U}$  and $C$ but we do not emphasize that dependence in the notation.  Note that we  used the same letter $\Phi$ to treat both {\sc spf}s and {\sc $k$-scf}s. Note also that $\Phi$ is injective.

\begin{theorem}\label{fu-min-2}Let $U\le S_h\times S_n\times \{id\}$ be regular, $(p^j)_{j\in\mathcal{P}^U}\in\mathfrak{S}(U)$ and
$C\in\mathfrak{P}^U$ $(C\in\mathfrak{C}_k^U)$. Then
\[
\mathfrak{F}^U_{C}=\Phi\left(\times_{j\in\mathcal{P}^U}C(p^j)	\right)\quad
\Big(
\mathfrak{F}^U_{k,C}=\Phi\left(\times_{j\in\mathcal{P}^U}C(p^j)	\right)
\Big)
\]
and
\[
|\mathfrak{F}^U_{C}|= \prod_{j\in\mathcal{P}^U}\left|C(p^j)\right|\quad
\Big(
|\mathfrak{F}^U_{k,C}|= \prod_{j\in\mathcal{P}^U}\left|C(p^j)\right|
\Big).
\]
In particular, if $C$ is decisive, then
$\mathfrak{F}^U_{C}\neq\varnothing$ $(\mathfrak{F}^U_{k,C}\neq\varnothing)$.
\end{theorem}

The proof of the above result is similar to the one of Theorem 18 in Bubboloni and Gori (2016b) and can be found in Appendix \ref{final}.

\subsection{Analysis of $\mathfrak{F}_C^U$ and $\mathfrak{F}_{k,C}^U$ when $U \nleq S_h\times S_n\times \{id\}$}

Let $U\le G$ be regular such that $U\nleq S_h\times S_n\times \{id\}$, $C\in\mathfrak{P}^U (C\in\mathfrak{C}_k^U)$.
Recall that if $p\in \mathcal{P}$ and  $y\in C(p)$, then $y\in \mathbf{L}(N)$ ($y\in\mathbb{P}_k(N)$). Moreover, given $\psi\in S_n$, the meaning of the writing
$\psi y$  is carefully explained  in Section \ref{plo}.

Fix $(p^j)_{j\in\mathcal{P}^U}\in\mathfrak{S}(U)$ and $(\varphi_*,\psi_*,\rho_0)\in U$.
Define, for every $j\in \mathcal{P}^U_1$, the set
\[
A^1_C(p^j)=\{(y,z)\in C(p^j) \times  C(p^{j\,(\varphi_*,\psi_*,\rho_0)}): z\neq \psi_*y\},
\]
and, for every $j\in \mathcal{P}^U_2$, the set
\[
A^2_C(p^j)=\left\{x\in C(p^j): \psi_jx\neq x\right\},
\]
where $\psi_j$ is the unique element in $S_n$ such that
\begin{equation}\label{psi-j}
\mathrm{Stab}_U(p^j)\subseteq (S_h\times \{id\}\times \{id\})\cup (S_h\times \{\psi_j\}\times \{\rho_0\}).
\end{equation}
Note that the uniqueness of $\psi_j$ is guaranteed by Lemma 16 (ii) in Bubboloni and Gori (2016).

Next if $\mathcal{P}_1^U\neq \varnothing$, then define
\[
A^1_C=\times_{j\in \mathcal{P}_1^U}A^1_C(p^j),
\]
and if $\mathcal{P}_2^U\neq \varnothing$, then define
\[
A^2_C=\times_{j\in\mathcal{P}_2^U}A^2_C(p^j).
\]
Recall that $\mathcal{P}_1^U$ and $ \mathcal{P}_2^U$ cannot be both empty. Thus at least one of the above set is always defined.
Of course, $A^1_C$ and $A^2_C$ depend also on $U$, $(p^j)_{j\in\mathcal{P}^U}$ and $(\varphi_*,\psi_*,\rho_0)$ but we do not emphasize that dependence in the notation.

The proof of the next result is similar to Proposition 19 in Bubboloni and Gori (2016b) and can be found in Appendix \ref{final}.

\begin{proposition}\label{fu-min-ex}
Let $U\le G$ be regular such that $U\not\le S_h\times S_n\times \{id\}$, $(p^j)_{j\in\mathcal{P}^U}\in\mathfrak{S}(U)$, $(\varphi_*,\psi_*,\rho_0)\in U$ and $C\in\mathfrak{P}^U$ ($C\in\mathfrak{C}^U_k$).
For every $j\in \mathcal{P}_1^U$, let
$(y_j,z_j)\in A^1_C(p^j)$
and, for every $j\in\mathcal{P}_2^U$, let
$x_j\in A^2_C(p^j)$.
Then there exists
a unique $f\in\mathfrak{F}^U_C$ ($f\in\mathfrak{F}^U_{k,C}$) such that
$f(p^j)=y_j$ and $f(p^{j\,(\varphi_*,\psi_*,\rho_0)})=z_j$ for all $j\in\mathcal{P}_1^U$, and
$f(p^j)=x_j$ for all  $j\in\mathcal{P}_2^U$.
\end{proposition}

Let $U\le G$ be regular such that $U\not\le S_h\times S_n\times \{id\}$, $(p^j)_{j\in\mathcal{P}^U}\in\mathfrak{S}(U)$, $(\varphi_*,\psi_*,\rho_0)\in U$ and $C\in\mathfrak{P}^U$ ($C\in\mathfrak{C}^U_k$).
\begin{itemize}
\item[-] If $\mathcal{P}_2^U=\varnothing$, then let $\Psi_1: A^1_C\to \mathfrak{F}^U_C$ ($\Psi_1: A^1_C\to \mathfrak{F}^U_{k,C}$) be the function which associates with every  $(y_j,z_j)_{j\in\mathcal{P}_1^U}\in A^1_C$, the unique $f\in\mathfrak{F}^U_C$ ($f\in\mathfrak{F}^U_{k,C}$) defined in Proposition \ref{fu-min-ex}.
	\item[-] If $\mathcal{P}_1^U=\varnothing$, then let
$\Psi_2: A^2_C\to \mathfrak{F}^U_C$ ($\Psi_2: A^2_C\to \mathfrak{F}^U_{k,C}$) be the function which associates with every  $(x_j)_{j\in \mathcal{P}_2^U}\in A^2_C$, the unique
$f\in\mathfrak{F}^U_C$ ($f\in\mathfrak{F}^U_{k,C}$) defined in Proposition \ref{fu-min-ex}.
\item[-] If $\mathcal{P}_1^U\neq \varnothing$ and $\mathcal{P}_2^U\neq \varnothing$, then let
$\Psi_3:A^1_C\times A^2_C \to \mathfrak{F}^U_C$ ($\Psi_3:A^1_C\times A^2_C \to \mathfrak{F}^U_{k,C}$) be the function
which associates with every  $((y_j,z_j)_{j\in\mathcal{P}_1^U},(x_j)_{j\in \mathcal{P}_2^U})\in A^1_C\times A^2_C$, the unique
$f\in\mathfrak{F}^U_C$ ($f\in\mathfrak{F}^U_{k,C}$) defined in Proposition \ref{fu-min-ex}.
\end{itemize}
Of course, $\Psi_1$, $\Psi_2$ and $\Psi_3$ depend on $U$, $(p^j)_{j\in\mathcal{P}^U}$, $(\varphi_*,\psi_*,\rho_0)$ and $C$ but we do not emphasize that dependence in the notation. Note also that $\Psi_1$, $\Psi_2$ and $\Psi_3$  are injective.

\begin{theorem}\label{fu-min-count2}
Let $U\le G$ be regular such that $U\not\le S_h\times S_n\times \{id\}$, $(p^j)_{j\in\mathcal{P}^U}\in\mathfrak{S}(U)$, $(\varphi_*,\psi_*,\rho_0)\in U$ and $C\in\mathfrak{P}^U$ $(C\in \mathfrak{C}_k^U)$.
Then
\[
\mathfrak{F}^U_C\; \big(\mathfrak{F}^U_{k,C}\big) =
\left\{
\begin{array}{ll}
\Psi_1(A^1_C) & \mbox{if }\mathcal{P}_2^U=\varnothing\\
\Psi_2(A^2_C) & \mbox{if }\mathcal{P}_1^U=\varnothing\\
\Psi_3(A^1_C\times A^2_C)& \mbox{if }\mathcal{P}_1^U\neq\varnothing\mbox{ and }\mathcal{P}_2^U\neq\varnothing\\
\end{array}
\right.
\]
and
\[
|\mathfrak{F}^U_C|\; \big( |\mathfrak{F}^U_{k,C}| \big)=
\left\{
\begin{array}{ll}
|A^1_C| & \mbox{if }\mathcal{P}_2^U=\varnothing\\
|A^2_C| & \mbox{if }\mathcal{P}_1^U=\varnothing\\
|A^1_C|\cdot |A^2_C|& \mbox{if }\mathcal{P}_1^U\neq\varnothing\mbox{ and }\mathcal{P}_2^U\neq\varnothing\\
\end{array}
\right.
\]
Moreover,  if $C$ is decisive, then we have that:
\begin{itemize}
\item for every $j\in \mathcal{P}^U_1$, $A^1_C(p^j)\neq \varnothing$,
\item $\mathfrak{F}^U_C\neq \varnothing$ $(\mathfrak{F}^U_{k,C}\neq\varnothing)$ if and only if, for every $j\in \mathcal{P}^U_2$, $A^2_C(p^j)\neq \varnothing$.
\end{itemize}
\end{theorem}

\begin{proof}
Let $C\in\mathfrak{P}^U$.
Assume first that $\mathcal{P}_1^U$ and  $\mathcal{P}_2^U$ are both nonempty.
Consider $f\in \mathfrak{F}^U_C$ and note that
\[
\left((f(p^j),f(p^{j\,(\varphi_*,\psi_*,\rho_0)}))_{j\in\mathcal{P}_1^U},(f(p^j))_{j\in \mathcal{P}_2^U}\right)\in A^1_{C}\times A^2_{C},
\]
and
\[
\Psi_3\left((f(p^j),f(p^{j\,(\varphi_*,\psi_*,\rho_0)}))_{j\in\mathcal{P}_1^U},(f(p^j))_{j\in \mathcal{P}_2^U}\right)=f.
\]
Thus $\Psi_3$ is bijective and we have $|\mathfrak{F}^U_C|=|A^1_C\times A^2_C|=|A^1_C|\cdot| A^2_C|$.
The case $\mathcal{P}_1^U=\varnothing$ and the case $\mathcal{P}_2^U=\varnothing$ are similar and then omitted.

Assume now that $C$ is decisive. Assume, by contradiction, that there exists $j\in \mathcal{P}^U_1$ such that
$A^1_C(p^j)= \varnothing.$ Thus, for every $y\in C(p^j)$ and $z\in C(p^{j\,(\varphi_*,\psi_*,\rho_0)})$ we have $z=\psi_* y$. Using decisiveness, fix $z\in C(p^{j\,(\varphi_*,\psi_*,\rho_0)})$. If $y_1,y_2\in C(p^j)$, we then have $z=\psi_* y_1=\psi_* y_2$ so that $y_1=y_2$.
It follows that $|C(p^j)|=1=|C(p^{j\,(\varphi_*,\psi_*,\rho_0)})|$ and that $C(p^{j\,(\varphi_*,\psi_*,\rho_0)})=\psi_*C(p^j)$, against $U$-consistency.
The last part of the theorem is trivial.

Let now consider $C\in\mathfrak{C}^U_k$. The theorem can be proved using formally the same argument.
\end{proof}

\subsection{Proof of Theorem \ref{main2}}

\begin{proof}[Proof of Theorem \ref{main2}]
Let first $U\leq S_h\times S_n\times \{id\}.$ Then, by Theorem \ref{fu-min-2}, we have that $\mathfrak{F}^U_{C}\neq\varnothing$.
Let next $U\nleq S_h\times S_n\times \{id\}.$ In order to show that also in this case we have $\mathfrak{F}^U_{C}\neq\varnothing$,
by Theorem \ref{fu-min-count2},
we need to prove that, for every $j\in \mathcal{P}^U_2$, $A^2_C(p^j)\neq \varnothing$.
Assume, by contradiction, that  there exists $j\in \mathcal{P}^U_2$ such that $A^2_C(p^j)= \varnothing$ and consider $\psi_j\in S_n$ as defined in \eqref{psi-j}. Then, for every $x\in C(p^j),$ we have $\psi_jx=x.$ Recall that $\psi_j$ is  a conjugate of $\rho_0$ and thus, in particular, $\psi_j\neq id$.
Since $C$ is decisive, we can pick $x \in C(p^j)$. Thus $x\in S_n$ and, using the cancellation law in $S_n$, we get the contradiction $\psi_j=id.$
\end{proof}

\subsection{Proof of Theorem \ref{main1}}

 Theorem \ref{fu-min-count2}, says that to guarantee the existence of a $U$-consistent resolute refinement for some $C\in \mathfrak{C}_k^U\cup \mathfrak{P}^U$ we need to satisfy the condition $A^2_C(p^j)\neq \varnothing$ for every $j\in \mathcal{P}^U_2$, where $(p^j)_{j\in\mathcal{P}^U}\in\mathfrak{S}(U)$.  We have seen  in the proof of Theorem \ref{main2}, that that condition always holds if $C\in\mathfrak{P}^U.$ On the other hand,  there is no reason for having it satisfied when $C\in \mathfrak{C}_k^U$. Indeed, in that context, $A^2_C(p^j)=\varnothing$ for some $j\in \mathcal{P}^U_2$ means that for $\psi_j\in S_n$ defined by \eqref{psi-j} we have
\begin{equation}\label{crisi}
\mbox{for every } x\in C(p^j),\, \psi_jx=x.
\end{equation}
In other words, $\psi_j$ fixes all the $k$-subsets of $N$ appearing in $C(p^j)$. Thus the elements of $C(p^j)$ need to be union of $\psi_j$-orbits and this does not constitute, in principle, a contradictory fact.

The situation is then more variegated  with respect to the case of the  {\sc spc}s and, in order to manage it,  we need some preliminary work. The next lemma is an easy but very useful starting point.

\begin{lemma}\label{singleton} Let $U\le G$ be regular such that $U\not\le S_h\times S_n\times \{id\}$, $(p^j)_{j\in\mathcal{P}^U}\in\mathfrak{S}(U)$ and $C\in \mathfrak{C}_k^U$. If $j\in\mathcal{P}_2^U$ is such that $|C(p^j)|=1$, then $A^2_C(p^j)\neq \varnothing$.
\end{lemma}
\begin{proof} Let $C(p^j)=\{x_1\}$ where $x_1$ is a $k$-subset of $N$ and assume, by contradiction, that $A^2_C(p^j)= \varnothing$. Thus, by \eqref{crisi},
we have  $\psi_j x_1=x_1$, with $\psi_j\in S_n$ defined in \eqref{psi-j}.
Since $j\in \mathcal{P}_2^U$, there exists  $(\varphi_1,\psi_1,\rho_0)\in \mathrm{Stab}_U(p^j)$ and, by the regularity of $U$, we have $\psi_1=\psi_j.$  Thus $\psi_1 x_1=x_1.$ It follows that $\psi_1^{-1}C(p^j)=C(p^j).$ On the other hand, since $U\not\le S_h\times S_n\times \{id\}$, there exists $(\varphi_*,\psi_*,\rho_0)\in U$ and,
by \eqref{action-e} and \eqref{sccU1}, we deduce that
\[
C(p^{j\,(\varphi_*,\psi_*,\rho_0)})
=C\left((p^{j\,(\varphi_1,\psi_1,\rho_0)})^{(\varphi_*\varphi_1^{-1},\psi_*\psi_1^{-1}, id)}\right)
\]
\[
=C(p^{j	\,(\varphi_*\varphi_1^{-1},\psi_*\psi_1^{-1}, id)})
=\psi_*\psi_1^{-1}C(p^j)=
\psi_*C(p^j),
\]
which contradicts \eqref{sccU2}.
\end{proof}

\begin{proof}[Proof of the implication $(ii)\Rightarrow(i)$ of Theorem \ref{main1}] Assume that one among $(a)$-$(d)$ holds and let $C\in\mathfrak{C}^U_k$ be decisive. We will prove that $\mathfrak{F}^U_{k,C}\neq\varnothing$.

If $(a)$ holds then, by \eqref{PU2vuoto}, we have
$\mathcal{P}^U_2=\varnothing$.  Hence, by Theorem \ref{fu-min-count2}, we immediately deduce $\mathfrak{F}^U_{k,C}\neq\varnothing$.
Assume then that $(a)$ does not hold but one among $(b)$-$(d)$ holds. Then, using the other implication in \eqref{PU2vuoto}, we have $\mathcal{P}^U_2\neq\varnothing$ and, by Theorem \ref{fu-min-count2}, we need to show that, for every $j\in\mathcal{P}_2^U$, $A^2_C(p^j)\neq \varnothing$.

Assume by contradiction that for some $j\in\mathcal{P}_2^U$, we have $A^2_C(p^j)= \varnothing$. Thus \eqref{crisi} holds true, for
$\psi_j\in S_n$ defined by\eqref{psi-j}. Recall now that $\psi_j$ is  a conjugate of $\rho_0$. Thus if $n$ is even all its orbits have size $2$; if $n$ is odd we have a unique orbit of size $1$ given by the only fixed point of $\psi_j$ and all the other orbits  have size 2.

Assume first that $k=1$. Then, for every $x\in C(p^j)$, $x$ is a singleton and the unique element in $x$ is a fixed point for  $\psi_j.$ Since, for $n$ even,  $\psi_j$ has no fixed point we deduce that $n$ is odd and $C(p^j)=\{\{x_1\}\}$ where $x_1$ is the only fixed point of $\psi_j$. Thus, by Lemma \ref{singleton}, we get the contradiction $A^2_C(p^j)\neq \varnothing.$

Assume now that $n\leq 3$. Because of the previous step we need to consider only the case $n=3$ and $k=2$. Thus, every $x\in C(p^j)$ is a $2$-subset of $N$ fixed by $\psi_j$. But $T_{\psi_j}=[1,2]$, that is, $\psi_j=(a\, b)$ for some distinct $a,b\in N$. Thus the only possibility is $x=\{a,b\}$ and hence $C(p^j)=\{\{a,b\}\}$ is a singleton. Again Lemma \ref{singleton} gives the internal contradiction $A^2_C(p^j)\neq \varnothing.$

Assume next that $n$ is even and $k$ is odd. Every $x\in C(p^j)$ is a $k$-subset of $N$  fixed by $\psi_j$. But since every orbit of $\psi_j$ has size $2$, any subset of $N$ fixed by $\psi_j$ has even size. Thus $C(p^j)=\varnothing$ against decisiveness.

Finally assume that $k=n-1$. If $n$ is even, then $k$ is odd and we conclude by the previous step. If $n$ is odd,  we have that every $x\in C(p^j)$  is a $(n-1)$-subset of $N$ fixed by $\psi_j.$ But the unique subset of $N$ of size $n-1$ fixed by $\psi_j$ is the union of the orbits of $\psi_j$ of size $2$, that is  $N\setminus \{x_1\}$, where $x_1$  is the unique fixed point of $\psi_j$. Thus, $|C(p^j)|=1$ and Lemma \ref{singleton} gives the contradiction $A^2_C(p^j)\neq \varnothing.$
\end{proof}

We are now left with proving the implication $(i)\Rightarrow (ii)$ of Theorem \ref{main1}. To that purpose, we need to introduce and study a special family of
{\sc $k$-scc}s.

Let $U\le G$ be regular. For every $p\in \mathcal{P}$, denote by $\psi_p$ the unique permutation in $S_n$ such that
\[
\mathrm{Stab}_U(p)\subseteq \left(S_h\times \{id\}\times \{id\}\right)\cup \left( S_h\times \{\psi_p\}\times \{\rho_0\}\right).
\]
The {\sc $k$-scc} $U_k$ associated with $U$ is defined, for every $p\in \mathcal{P}$, by
\[
U_k(p)=
\left\{
\begin{array}{lll}
 \mathbb{P}_k(N)\ & \mbox{if }\  \mathrm{Stab}_U(p)\leq S_h\times  \{id\} \times \{id\}\\
&\\
\{x\in  \mathbb{P}_k(N):  \psi_p\, x=x\}\ &  \mbox{if }\  \mathrm{Stab}_U(p)\nleq S_h\times \{id\}\times \{id\}\\
\end{array}
\right.
\]
The next two propositions show some important properties of $U_k$.

\begin{proposition}\label{decisiva} Let $U\le G$ be regular. The following facts are equivalent:
\begin{itemize}
\item[$(i)$] $U_k$ is decisive.
\item[$(ii)$] One of the following condition is satisfied:
\begin{itemize}
\item[$(a)$] $\mathcal{P}_2^U=\varnothing;$
\item[$(b)$] $n$ is odd;
\item[$(c)$]$n$ is even with $n\geq 4$ and $k$ is even.
\end{itemize}
\end{itemize}
\end{proposition}
\begin{proof}$(i)\Rightarrow (ii)$ Let $U_k$ be decisive, $\mathcal{P}_2^U\neq\varnothing$ and $n$ even. We show that $n\geq 4$ and that $k$ is even. Since $\mathcal{P}_2^U\neq\varnothing$, there exists $p\in \mathcal{P}$ such that $\mathrm{Stab}_U(p)\nleq S_h\times \{id\}\times \{id\}$ and thus there exists $(\varphi, \psi_p,\rho_0)\in \mathrm{Stab}_U(p).$ By
$U_k(p)\neq\varnothing,$ we deduce that there exists at least one $k$-subset $x$ of $N$ fixed by $\psi_p$. Since $1\leq k\leq n-1$, $x$ is a proper nonempty subset of $N$ which is union of $\psi_p$-orbits. Since $n$ is even, we have $n/2$ orbits all of size $2$. If $n=2$, then we have just one orbit and thus the only subsets of $N$ which are union of orbits are $\varnothing$ and $N$. It follows that $n\geq 4$ and that $k=|x|$ is even.

$(ii)\Rightarrow (i)$ Let $\mathcal{Q}=\{p\in\mathcal{P}:\mathrm{Stab}_U(p)\nleq S_h\times \{id\}\times \{id\}\}$. Since $\mathbb{P}_k(N)\neq \varnothing$, in order to show that $U_k$ is decisive, it is enough to see that for every $p\in \mathcal{Q}$, we have $\{x\in  \mathbb{P}_k(N):  \psi_p\, x=x\}\neq \varnothing.$
 If $\mathcal{P}_2^U=\varnothing,$ this is clear because necessarily we also have $\mathcal{Q}=\varnothing$. Let $n$ be odd and pick $p\in \mathcal{Q}$. Then $\psi_p$ has $\frac{n-1}{2}\geq 1$ orbits of size $2$ and one orbit of size $1$.
Assembling some of those orbits we can surely build a $k$-subset of $N$ fixed by $\psi_p$, whatever $k$ is. Thus $\{x\in  \mathbb{P}_k(N):  \psi_p\, x=x\}\neq \varnothing$.
Let finally $n$ be even with $n\geq 4$ and $k$ be even. Then we have $2\leq k\leq n-2$. Pick $p\in \mathcal{Q}$. Then $\psi_p$ has $\frac{n}{2}\geq 2$ orbits of size $2$.  Assembling some of those orbits we can surely build a $k$-subset of $N$ fixed by $\psi_p$. Thus, $\{x\in  \mathbb{P}_k(N):  \psi_p\, x=x\}\neq \varnothing.$
\end{proof}
In order to state the next result, let us first define the set
\begin{equation}\label{T}
T=\{(n,k)\in \mathbb{N}^2: n\le 3\}\cup\{(n,k)\in \mathbb{N}^2: k\in \{1,n-1\}\}\cup\{(n,k)\in \mathbb{N}^2: n\hbox{ is even, }\ k \hbox{ is odd}\}.
\end{equation}
\begin{proposition}\label{simmetrica}
 Let $U\le G$ be regular. Then the following facts hold:
\begin{itemize}
\item[$(i)$] If $(\varphi,\psi,\rho)\in U$ and $p\in \mathcal{P}$,  then $U_k(p^{(\varphi,\psi,\rho)})=\psi U_k(p). $
\item[$(ii)$] If $(n,k)\notin T$ then, for every $p\in \mathcal{P}$, $| U_k(p)|\geq 2$. In particular $U_k$ is decisive.
\item[$(iii)$]If $(n,k)\notin T$ then $U_k\in \mathfrak{C}_k^U$.
\end{itemize}
 \end{proposition}
\begin{proof}$(i)$ Let first $p$ be such that $\mathrm{Stab}_U(p)\leq S_h\times \{id\}\times \{id\}$. Then also $\mathrm{Stab}_U(p^{(\varphi,\psi,\rho)})\leq S_h\times \{id\}\times \{id\}$ and thus $U_k(p^{(\varphi,\psi,\rho)})=U_k(p)=\mathbb{P}_k(N)=\psi \mathbb{P}_k(N)=\psi U_k(p).$

Let next $p$ be such that $\mathrm{Stab}_U(p)\nleq S_h\times \{id\}\times \{id\}$. Then also $\mathrm{Stab}_U(p^{(\varphi,\psi,\rho)})\nleq S_h\times \{id\}\times \{id\}$. We find the link between $\psi_p$ and $\psi_{p^{(\varphi,\psi,\rho)}}.$ Let $(\varphi_1,\psi_p,\rho_0)\in \mathrm{Stab}_U(p).$ Then, using the fact that $\Omega$ is abelian, we have $$(\varphi_1,\psi_p,\rho_0)^{(\varphi,\psi,\rho)}=(\varphi_1^{\varphi},\psi_p^{\psi},\rho_0^{\rho})=(\varphi_1^{\varphi},\psi_p^{\psi},\rho_0)\in \mathrm{Stab}_U(p)^{(\varphi,\psi,\rho)}=\mathrm{Stab}_U(p^{(\varphi,\psi,\rho)}).$$
Thus, $\psi_{p^{(\varphi,\psi,\rho)}}=(\psi_p)^{\psi}.$ Using $\mathbb{P}_k(N)=\psi \mathbb{P}_k(N)$ and \eqref{equal-new}, it follows that
$$U_k(p^{(\varphi,\psi,\rho)})=\{x\in  \mathbb{P}_k(N):  \psi_{p^{(\varphi,\psi,\rho)}}\, x=x\}=\{\psi x\in  \mathbb{P}_k(N): x\in \mathbb{P}_k(N),\ \psi\psi_{p}\psi^{-1}\psi\, x=\psi x\}=$$
$$\{\psi x\in  \mathbb{P}_k(N): x\in \mathbb{P}_k(N),\ \psi_{p}\, x= x\}=\psi U_k(p).$$

$(ii)$ Let $(n,k)\notin T$. Then we have $k\notin\{1,n-1\}$ and $n\geq 4$. Moreover, $n$ is odd or $n$ is even and $k$ is even.
Thus, by Lemma \ref{decisiva}, $U_k$ is decisive. Fix now $p\in \mathcal{P}$. We show that $| U_k(p)|\geq 2$.
 If $U_k(p)=\mathbb{P}_k(N),$ we have $|U_k(p)|=\binom{n}{k}\geq n\geq 2.$ If instead
$U_k(p)=\{x\in  \mathbb{P}_k(N):  \psi_p\, x=x\},$ then $U_k(p)$ is made up by all the $k$-subsets of $N$ which can be formed as union of $\psi_p$-orbits. Since $|U_k(p)|\geq 1$ we have at least one of them, say $x_1\in U_k(p)$. Let first $n$ be odd. Then $n\geq 5$ and there are $\frac{n-1}{2}\geq 2$ orbits of  $\psi_p$ of size $2$ and one orbit of size $1.$ Since $k\geq 2,$ there is at least one orbit of size $2$ included in $x_1$. On the other hand, not all the orbits of size $2$ are included in $x_1$, otherwise $k=|x_1|=n-1.$ Now exchange one orbit of size $2$ included in $x_1$ with one orbit of size $2$ left out. This builds another $k$-subset $x_2$ of $N$ belonging to $U_k(p)$. Let next $n$ be even and $k$ be even. Here $n\geq 4$ and we have $\frac{n}{2}\geq 2$ orbits of  $\psi_p$ all of size $2$. Obviously we have used at least one orbit to build $x_1$ but not all and we can exchange one orbit included in $x_1$ with one  left out building another $k$-subset $x_2$ of $N$ belonging to $U_k(p)$.

$(iii)$ By $(i)$ we have that condition \eqref{sccU1} for $U$-consistency is satisfied; by $(ii)$ we also have that, trivially, condition \eqref{sccU2} for $U$-consistency is satisfied because it is never the case to have $U_k(p)$ a singleton for any $p\in \mathcal{P}$. Thus $U_k\in \mathfrak{C}_k^U$.
\end{proof}

\begin{proof}[Proof of the implication  $(i)\Rightarrow (ii)$ of Theorem \ref{main1}] By \eqref{PU2vuoto}, we get the desired result proving that if $\mathcal{P}^U_2\neq\varnothing$ and $(n,k)\not\in T$, then $\mathfrak{F}^U_{k,U_k}=\varnothing$. Indeed, assume that $\mathcal{P}^U_2\neq\varnothing$ and $(n,k)\not\in T$ and consider $U_k$. By Proposition \ref{simmetrica}, we have that $U_k\in \mathfrak{C}_k^U$. Consider $(p^j)_{j\in\mathcal{P}^U}\in \mathfrak{S}(U)$ and pick $j\in\mathcal{P}^U_2$ so that $\mathrm{Stab}_U(p^j)\nleq S_h\times \{id\}\times \{id\}$. Then, by definition of $U_k$, we have
\[
U_k(p^j)=\{x\in  \mathbb{P}_k(N):  \psi_{p^j}\, x=x\}=\{x\in  \mathbb{P}_k(N):  \psi_{j}\, x=x\}.
\]
Therefore we surely have that, for every $x\in U_k(p)$, $\psi_j\,x=x$ and \eqref{crisi} is satisfied, so that $A^2_{U_k}(p^j)=\varnothing$ and, by Theorem \ref{fu-min-count2},  $\mathfrak{F}^U_{k,U_k}=\varnothing.$
\end{proof}

\section{Additional material for reader's convenience}\label{final}
We collect here the detailed proofs of those results which are an adaptations of results recently published by the two authors. 

\subsection{Proof of Proposition \ref{U,V-corr}}

Let us prove first that $ \mathfrak{P}^{U_1}\cap \mathfrak{P}^{U_2}=\mathfrak{P}^{\langle U_1,U_2\rangle}$.
Since $\langle U_1,U_2\rangle\leq G$ contains both $U_1$ and $U_2$, we  immediately get $\mathfrak{P}^{\langle U_1,U_2\rangle}\subseteq \mathfrak{P}^{U_1}\cap \mathfrak{P}^{U_2}.$ Let us now fix $C\in\mathfrak{P}^{U_1}\cap \mathfrak{P}^{U_2}$ and prove
 $C\in \mathfrak{P}^{\langle U_1,U_2\rangle}$.
Define, for every  $t\in \mathbb{N}$, the set $\langle U_1,U_2\rangle_t$ of the elements in $\langle U_1,U_2\rangle$ that can be written as product of $t$ elements of $U_1\cup U_2$. Then 
 we have $\langle U_1,U_2\rangle=\bigcup_{t\in\mathbb{N}}\langle U_1,U_2\rangle_t$ and to get $C\in \mathfrak{P}^{\langle U_1,U_2\rangle}$ it is enough to show the two following facts:
\begin{itemize}
\item[$(a)$] for every $t\in\mathbb{N},$
\begin{equation}\label{A1}
\mbox{for every } p\in\mathcal{P} \mbox{ and } g=(\varphi, \psi, id)\in \langle U_1,U_2\rangle_t,\mbox{ \eqref{sccU1} holds true;}
\end{equation}
\item[$(b)$]  for every $t\in\mathbb{N},$ $ p\in\mathcal{P}$ and $g=(\varphi, \psi, \rho_0)\in \langle U_1,U_2\rangle_t$, \eqref{sccU2} holds true.
\end{itemize}
For every $g=(\varphi,\psi,\rho)\in G$, define $\overline{g}=(\varphi,\psi,id)\in G$. We start
showing that, for every $t\in \mathbb{N}$,
\begin{equation}\label{B}
 g\in \langle U_1,U_2\rangle_t\ \mbox{ implies} \ \overline{g}\in \langle U_1,U_2\rangle_t.
 \end{equation}
 If $\rho=id$, there is nothing to prove. So assume $\rho=\rho_0$.
Since, for every  $i\in\{1,2\}$, we have that $U_i=Z_i\times  R_i$ with $Z_i\le S_h\times S_h$ and $R_i\le  \Omega$, then \eqref{B} surely holds for $t=1$. If $t\geq 2$, pick $g=g_1\cdots g_t=(\varphi,\psi,\rho_0)\in \langle U_1,U_2\rangle_t$, where $g_1,\ldots,g_t\in U_1\cup U_2$. Since $\rho_0$ has order two, the number of $j\in\{1,\dots,t\}$ such that the third component of $g_j$ is $\rho_0$ is odd. Pick $j\in\{1,\dots,t\}$ such that $g_j=(\varphi_j,\psi_j,\rho_0)$.
By the case $t=1$, we have that  $\overline{g}_j=(\varphi_j,\psi_j,id)\in U_1\cup U_2$, so that
$\overline{g}=g_1\ldots g_{j-1}\overline{g}_jg_{j+1}\dots g_t\in \langle U_1,U_2\rangle_t$ and its first and second components are equal to those of $g.$ Moreover, the number of factors  in $\overline{g}$ having as third component $\rho_0$ is even, which gives $\overline{g}=(\varphi,\psi,id).$

We now show $(a)$, by induction on $t$. If $t=1$, we have $g\in\langle U_1,U_2\rangle_1= U_1\cup U_2$ and so \eqref{A1} is guaranteed by $C\in\mathfrak{P}^{U_1}\cap \mathfrak{P}^{U_2}$.
Assume \eqref{A1}  up to some $t\in\mathbb{N}$ and show that it holds also for $t+1$. Let $p\in\mathcal{P}$ and $g=(\varphi,\psi,id)\in \langle U_1,U_2\rangle_{t+1}$. Then there exist $g_*=(\varphi_*,\psi_*,\rho_*)\in \langle U_1,U_2\rangle_{t}$ and $g_1=(\varphi_1,\psi_1,\rho_1)\in U_1\cup U_2$
such that $g=g_1g_*=(\varphi_1\varphi_*,\psi_1\psi_*,\rho_1\rho_*)$.
We want to show that $C(p^{g})=\psi_1\psi_*C(p)$.
Note that $g=\overline{g}_1\overline{g}_*$ and that, by \eqref{B},  $\overline{g}_*\in \langle U_1,U_2\rangle_{t}$
and $\overline{g}_1\in U_1\cup U_2.$
Then, using \eqref{action-e} and applying the inductive hypothesis for \eqref{A1} both to $\overline{g}_1$ and to $\overline{g}_*$,  we get
$C(p^{g})=C(p^{\overline{g}_1\overline{g}_*})=C((p^{\overline{g}_*})^{\overline{g}_1})=\psi_1C(p^{\overline{g}_*})=\psi_1\psi_*C(p)$.

We next show $(b)$.
Let $t\in\mathbb{N}$,
$p\in\mathcal{P}$, $g=(\varphi, \psi, \rho_0)\in \langle U_1,U_2\rangle_{t}$ and $|C(p)|=1$. We need to show that $C(p^g)\neq \psi C(p).$
First of all note that, since $\langle U_1,U_2\rangle$ contains an element with third component $\rho_0$, then we necessarily have $R_1=\Omega$ or  $R_2=\Omega$, so that $(id,id,\rho_0)\in U_1\cup U_2$.
Moreover, we can express $g$ as $g=\overline{g}\,(id,id,\rho_0)$ and,  by \eqref{B}, $\overline{g} \in \langle U_1,U_2\rangle_{t}$.
Thus, by \eqref{action-e}  and $(a)$, we have $C(p^g)=C((p^{(id,id,\rho_0)})^{\overline{g}})=\psi C(p^{(id,id,\rho_0)}).$
On the other hand, since $(id,id,\rho_0)\in U_1\cup U_2$ and $C\in\mathfrak{P}^{U_1}\cap \mathfrak{P}^{U_2}$, we get
 $C(p^{(id,id,\rho_0)})\neq C(p)$ and so, by \eqref{equal-new}, $C(p^g)=\psi C(p^{(id,id,\rho_0)})\neq \psi C(p)$ as required.

The proof of the equality $ \mathfrak{C}_k^{U_1}\cap \mathfrak{C}_k^{U_2}=\mathfrak{C}_k^{\langle U_1,U_2\rangle}$ is formally the same.

\subsection{Proof of Proposition \ref{rappresentanti2}}

Let $p\in \mathcal{P}$ and show that $C(p)=C'(p)$. We know there exist $j\in\mathcal{P}^U$ and $(\varphi,\psi,id)\in U$ such that $p=p^{j\,(\varphi,\psi,id)}$.
Then,
\begin{equation}\label{chain}
C(p)=C(p^{j\,(\varphi,\psi,id)})=\psi C(p^{j})=\psi C'(p^{j})=C'(p^{j\,(\varphi,\psi,id)})=C'(p).
\end{equation}

\subsection{Proof of Proposition \ref{rappresentanti3}}  Let $p\in \mathcal{P}$ and show that $C(p)=C'(p)$. Let $j \in \mathcal{P}^U$ be the unique orbit such that $p\in j .$
If there exists $(\varphi,\psi,id)\in U$ such that $p=p^{j\,(\varphi,\psi,id)}$, then we get $C(p)=C'(p)$ operating as in \eqref{chain}.
So, assume that,
\begin{equation}\label{cond}
\mbox{for every } (\varphi,\psi,\rho)\in U  \mbox{ such that } p=p^{j\,(\varphi,\psi,\rho)},  \mbox{ we have } \rho=\rho_0.
\end{equation}
We show that \eqref{cond} implies $\mathrm{Stab}_U(p^j)\le  S_h\times S_n\times \{id\}$.  Indeed, suppose by contradiction that there exists $(\varphi_1,\psi_1,\rho_0)\in \mathrm{Stab}_U(p^j)$. Pick  $(\varphi,\psi,\rho_0)\in U$ such that $p=p^{j\,(\varphi,\psi,\rho_0)}$ and note that, by \eqref{action-e},
\[p=p^{j\,(\varphi,\psi,\rho_0)}=(p^{j\,(\varphi_1,\psi_1,\rho_0)})^{(\varphi,\psi,\rho_0)}=p^{j\,(\varphi\varphi_1,\psi\psi_1,id)}
\]
which contradicts \eqref{cond}. As a consequence, $j\in \mathcal{P}_1^U$ and thus, by assumption, $C(p^{j\,(\varphi_*,\psi_*,\rho_0)})=C'(p^{j\,(\varphi_*,\psi_*,\rho_0)})$. Pick  again $(\varphi,\psi,\rho_0)\in U$ such that $p=p^{j\,(\varphi,\psi,\rho_0)}$ and note that, by \eqref{action-e},
\[p=p^{j\,(\varphi,\psi,\rho_0)}=(p^{j\,(\varphi_*,\psi_*,\rho_0)})^{(\varphi\varphi_*^{-1},\psi\psi_*^{-1},id)}
\]
so that, since $C$ and $C'$ are $U$-consistent, we finally obtain
\[
C(p)=C\left((p^{j\,(\varphi_*,\psi_*,\rho_0)})^{(\varphi\varphi_*^{-1},\psi\psi_*^{-1},id)}\right)=\psi\psi_*^{-1}C(p^{j\,(\varphi_*,\psi_*,\rho_0)})
\]
\[
=\psi\psi_*^{-1}C'(p^{j\,(\varphi_*,\psi_*,\rho_0)})=C'\left((p^{j\,(\varphi_*,\psi_*,\rho_0)})^{(\varphi\varphi_*^{-1},\psi\psi_*^{-1},id)}\right)=C'(p).
\]

\subsection{Proof of $(iii)\Rightarrow (i)$ in Theorem \ref{main3}}

$(iii)\Rightarrow (i)$ By Theorem \ref{main}, in order to prove that
that $\mathfrak{F}^{*U}_C\neq\varnothing$, it is sufficient to show that,  for every $p\in\mathcal{P}$,  $S(p)\cap C(p)\neq\varnothing$.
By Lemma \ref{s}$\,(i)$, since $U$ is regular, for every $p\in\mathcal{P}$, we have  $S(p)\neq\varnothing$.
Now, for every $p\in\mathcal{P}$ such that $\mathrm{Stab}_U(p)\le S_h \times \{id\}\times \{id\}$, we have $S(p)=\mathbf{L}(N)$ so that trivially
$S(p)\cap C(p)\neq\varnothing$. Thus, we are left with proving that,
for every $p\in\mathcal{P}$ such that $\mathrm{Stab}_U(p)\not \le S_h \times \{id\}\times \{id\}$,  we have $S(p)\cap C(p)\neq\varnothing$.

\vspace{2mm}

\noindent {\it From now till the end of the section, let us fix  $p\in\mathcal{P}$ and $(\varphi,\psi,\rho_0)\in \mathrm{Stab}_U(p)$. Recall that $U$ is regular;
$\psi$  is a conjugate  of $\rho_0$; $R(p)$ is irreflexive, acyclic such that $\{q\in\mathbf{L}(N): R(p)\subseteq q\}\subseteq C(p)$ and condition $(a)$ is satisfied. We have to prove that $S(p)\cap C(p) \neq\varnothing$.}
\vspace{2mm}

We are going to exhibit an element of the set $S(p)\cap C(p)$, namely the linear order $q$ defined in \eqref{q-magic}. The construction of $q$ is quite tricky and relies on some preliminary lemmas concerning the properties of the relation $R(p)$ and of its transitive extension $R^*(p)$ defined in \eqref{sigmac}. Thus, the first part of the proof is devoted to the study of such relations.

Since $\psi$ is a conjugate of $\rho_0$,  we have that $\psi$ has the same type of $\rho_0$ and, in particular, $|\psi|=2$. Let
$(\hat{x}_j)_{j=1}^r\in N^r$ be a system of representatives of the $\psi$-orbits. Thus $r$ is the number of $\psi$-orbits and we have
\[
O(\psi)=\left\{
\{\hat{x}_j,\psi(\hat{x}_j)\}: j\in \ldbrack r \rdbrack
\right\}.
\]
Note that if $n$ is even, then $r=\frac{n}{2}$; $\psi$ has no fixed point; $|\{\hat{x}_j,\psi(\hat{x}_j)\}|=2$ for all $j\in \ldbrack r \rdbrack$. If instead $n$ is odd, then $r=\frac{n+1}{2}$; $\psi$ has a unique fixed point, say $\hat{x}_r$; $|\{\hat{x}_j,\psi(\hat{x}_j)\}|=2$ for all $j\in \ldbrack r-1 \rdbrack$.

$R(p)$ acyclic implies $R(p)$ asymmetric.
Thus, $x\succeq _{R(p)} y$ is equivalent to $x\succ_{R(p)} y$ and implies $x\neq y.$
We will write, for compactness, $x\succ y$ instead of $x\succ_{R(p)} y$. Then, for every $x,y\in N$, we translate condition $(a)$ into
\begin{equation}\label{prop-R}
x\succ y\quad  \Leftrightarrow\quad  \psi(y)\succ \psi(x).
\end{equation}
Given $x,y\in N$ with $x\neq y$, a chain $\gamma$ for $R(p)$ (or a $R(p)$-chain) from $x$ to $y$ is an ordered sequence  $x_1,\ldots, x_l,$ with $l\ge 2,$ of distinct elements of $N$ such that $x_1=x$, $x_l=y,$ and, for every $j\in\ldbrack l-1 \rdbrack$, $(x_j,x_{j+1})\in R(p)$. The number $l-1$ is called the length of the chain, $x$ its starting point and $y$ its end point.

Consider now the following relation on $N$,
\begin{equation}\label{sigmac}
R^*(p)=\{(x,y)\in N^2:
\mbox{there exists a}\  R(p)\mbox{-chain from } x\   \mbox{to}\  y\},
\end{equation}
and note that $R^*(p)\supseteq R(p).$

\begin{lemma}\label{asymmetry}
$R^*(p)$ is asymmetric and transitive. Moreover, for every $x,y\in N$, $(x,y)\in R^*(p)$ if and only if $(\psi(y),\psi(x))\in R^*(p)$.
\end{lemma}

\begin{proof}
Let us prove first that $R^*(p)$ is asymmetric.
Let $(x,y)\in R^*(p)$. Then, there exists a $R(p)$-chain from $ x$  to $ y$, that is, there exist $l\ge 2$ distinct $x_1,\ldots,x_l\in N$ such that $x=x_1$, $y=x_l$ and, for every $j\in\ldbrack l-1 \rdbrack$, $x_j\succ x_{j+1}$.
Assume, by contradiction, that  $(y,x)\in R^*(p).$ Then there exist $m\ge 2$ distinct $y_1,\ldots,y_m\in N$ such that $y=y_1$, $x=y_m$ and, for every $j\in
\ldbrack m-1 \rdbrack$, $y_j\succ y_{j+1}$. Consider now the set
$
A=\{j\in\{2,\ldots,m\}: y_j\in\{x_1,\ldots,x_{l-1}\}\}.
$
Clearly, because $y_m=x_1,$ we have $m\in A\neq \varnothing$. Let us define then $m^*=\min A$, so that there exists $l^*\in\ldbrack l-1 \rdbrack$ such that $y_{m^*}=x_{l^*}$. Then it is easy to check that $x_{l^*},x_{l^*+1},\ldots, x_l,y_2,\ldots, y_{m^*}$ is a sequence of at least three elements in $N$, with no repetition up to the $x_{l^*}=y_{m^*},$ which is a cycle in $R(p)$ and the contradiction is found.

Let us prove now that $R^*(p)$ is transitive. Let $(x,y),(y,z)\in R^*(p)$.
Then, by the definition of $R^*(p)$,  there exist $l\ge 2$ distinct $x_1,\ldots,x_l\in N$ such that $x=x_1$, $y=x_l$ and, for every $j\in\ldbrack l-1 \rdbrack$, $x_j\succ x_{j+1}$; moreover, there are $m\ge 2$ distinct $y_1,\ldots,y_m\in N$ such that $y=y_1$, $z=y_m$ and, for every $j\in\ldbrack m-1 \rdbrack$, $y_j\succ y_{j+1}$. Consider then the sequence $\gamma$ of alternatives  $x_1,\ldots,x_l,y_2,\ldots,y_m$. We show that those alternatives are all distinct. Assume that there exist $i\in \ldbrack l \rdbrack$ and $j\in \ldbrack m \rdbrack$ with $x_i=y_j$. Then we have a $R(p)$-chain with starting point $x_i$ and end point $y$ as well as a $R(p)$-chain with starting point $y$ and end point $y_j=x_i$, that is, $(x_i,y)\in R^*(p)$ and $(y,x_i)\in R^*(p),$ against the asymmetry.
It follows that $\gamma$ is a chain from $x$ to $z$, so that $(x,z)\in R^*(p)$.

We are left with proving that $(x,y)\in R^*(p)$ if and only if $(\psi(y),\psi(x))\in R^*(p)$. Let $(x,y)\in R^*(p)$ and consider $l\ge 2$ distinct $x_1,\ldots,x_l\in N$ such that $x=x_1$, $y=x_l$ and, for every $j\in\ldbrack l-1 \rdbrack$, $x_j\succ x_{j+1}$. Defining, for every $j\in\ldbrack l \rdbrack$, $y_j=\psi(x_{l-j+1})$ and using \eqref{prop-R}, it is immediately checked that $\psi(y)=y_1$, $\psi(x)=y_l$ and that, for every $j\in \ldbrack l-1 \rdbrack$, $y_{j}\succ y_{j+1}$. In other words, we have a $R^*(p)$-chain from $\psi(y)$ to $\psi(x),$ that is,
 $(\psi(y),\psi(x))\in R^*(p)$. The other implication is now a trivial consequence of $|\psi|=2$.
\end{proof}

In what follows, we write $x\hookrightarrow y$ instead of $(x,y)\in R^*(p)$ and $x\not\hookrightarrow y$ instead of $(x,y)\notin R^*(p)$.
As consequence of Lemma \ref{asymmetry}, for every $x,y,z\in N$, the following relations hold true:
$x\not\hookrightarrow x$;  $x\hookrightarrow y$ implies  $y\not\hookrightarrow x$; $x\hookrightarrow y$ and $y\hookrightarrow z$ imply  $x\hookrightarrow z$; $x\hookrightarrow y$ is equivalent to $\psi(y)\hookrightarrow \psi(x)$.

Define now, for every $z\in N$, the subset of $N$
\[
\Gamma(z)=\{x\in N: x\hookrightarrow  z\}.
\]
Note that $z\notin \Gamma(z)$ and that it may happen that $\Gamma(z)=\varnothing$. This is the case exactly when, for every $x\in N$, the relation $x\succ  z$ does not hold.
Define now the subset of $N$
\[
\Gamma=\mbox{$\bigcup_{z\in N}$}[\Gamma(z)\cap \Gamma(\psi(z))].
\]

\begin{lemma}\label{catena-estesa}
Let $x,y\in N.$ If $y\in \Gamma$ and $x\hookrightarrow  y$, then $x\in \Gamma.$
\end{lemma}

\begin{proof} Let $x, y\in N$ and $x\hookrightarrow  y$ with $y\in \Gamma$. Thus,
there exists $z\in N$ such that $y\hookrightarrow z$ and $y\hookrightarrow \psi(z).$  By transitivity we conclude that also $x\hookrightarrow z$ and $x\hookrightarrow \psi(z)$, that is, $x\in \Gamma$.
\end{proof}

\begin{lemma}\label{psiR-lemma}The following facts hold true:\vspace{-2mm}
\begin{itemize}
\item[$(i)$] $\Gamma\,\cap\,\psi(\Gamma)=\varnothing;$\vspace{-2mm}
\item[$(ii)$] if $n$ is odd, then $\hat{x}_r\notin \Gamma\,\cup\, \psi(\Gamma)$.  Moreover, if $x\in \Gamma$, then $\hat{x}_r\not \hookrightarrow  x;$\vspace{-2mm}
\item[$(iii)$] for every $x\in N$, $|\{x,\psi(x)\}\cap \Gamma|\leq 1.$\vspace{-2mm}
\end{itemize}
\end{lemma}

\begin{proof} $(i)$ Assume that there exists $x\in N$ with $x\in \Gamma$ and $x\in \psi(\Gamma).$ Since $|\psi|=2$, this gives $\psi(x)\in \Gamma,$ so that
there exist $z,y\in N$ with
\begin{equation}\label{x}
x\hookrightarrow  z,\quad x\hookrightarrow  \psi(z),\quad \psi(x)\hookrightarrow  y,\quad \psi(x)\hookrightarrow  \psi(y).
\end{equation}
By Lemma \ref{asymmetry} applied to the second and fourth relation in \eqref{x}, we also get
\begin{equation}\label{psix2}
z\hookrightarrow  \psi(x),\quad y\hookrightarrow  x.
\end{equation}
From \eqref{x} and \eqref{psix2} and by transitivity of $\hookrightarrow $, we deduce that $\psi(x)\hookrightarrow   x$ and $x \hookrightarrow  \psi(x)$, against the asymmetry of $R^*(p)$ established in Lemma \ref{asymmetry}.

$(ii)$ Assume that $\hat{x}_r\in \Gamma\,\cup\,\psi(\Gamma)$. Then, by (i), we have $\hat{x}_r=\psi(\hat{x}_r)\in \Gamma\,\cap\,\psi(\Gamma)=\varnothing,$ a contradiction.
Next let $x\in \Gamma$ and $\hat{x}_r\hookrightarrow  x.$ Then, by Lemma \ref{catena-estesa}, we also have $\hat{x}_r\in \Gamma$, a contradiction.

$(iii)$ Assume there is $x\in N$ such that both $x$ and $\psi(x)$ belong to $\Gamma.$ Then $x\in \psi(\Gamma)$ and so $x\in \Gamma\,\cap\,\psi(\Gamma)=\varnothing,$ against $(i)$.
\end{proof}

\begin{lemma}\label{fisso} Let $n$ be odd and $x\in N$. Then:\vspace{-2mm}
\begin{itemize}
\item[$(i)$]$x\hookrightarrow  \hat{x}_r$ implies $\{x,\psi(x)\}\cap \Gamma=\{x\}$;\vspace{-2mm}
\item[$(ii)$] $\hat{x}_r\hookrightarrow  x$ implies $\{x,\psi(x)\}\cap \Gamma=\{\psi(x)\}.$
\end{itemize}
\end{lemma}

\begin{proof} $(i)$ From $x\hookrightarrow  \hat{x}_r$, we get $x\in \Gamma(\hat{x}_r)\subseteq \Gamma$ and thus Lemma \ref{psiR-lemma}(i) gives $\psi(x)\notin \Gamma$. It follows that $\{x,\psi(x)\}\cap \Gamma=\{x\}$.

$(ii)$  From $\hat{x}_r\hookrightarrow  x$, using Lemma \ref{asymmetry}, we obtain $\psi(x)\hookrightarrow  \hat{x}_r$ and (i) applies to $\psi(x)$, giving $\{x,\psi(x)\}\cap \Gamma=\{\psi(x)\}.$
\end{proof}

Given $X\subseteq N$, define
\[
\Theta(p,X)=\left\{l\in \mathbf{L}(X) :   R(p)\cap X^2\subseteq l\right\}.
\]
Since $R$ is acyclic,  then $R(p)\cap X^2$ is an acyclic relation on $X$ and thus it admits a linear extension, so that $\Theta(p,X)\neq \varnothing.$
Note that if $Y\subseteq X \subseteq N$ and $l\in \Theta(p,X),$ then the restriction of $l$ to $Y$ belongs to $\Theta(p,Y).$
Moreover, since $C_R$ refines $C$, we have $ \Theta(p,N)\subseteq C(p).$

For every $j\in \ldbrack r \rdbrack$, consider $\{\hat{x}_j,\psi(\hat{x}_j)\}\cap \Gamma$. By Lemma \ref{psiR-lemma}$\,(iii)$ the order of this set cannot exceed 1. Define then the sets
\[
J=\{j\in \ldbrack r \rdbrack: |\{\hat{x}_j,\psi(\hat{x}_j)\}\cap \Gamma|=1\}, \quad J^*=\{j\in \ldbrack r \rdbrack\setminus J: |\{\hat{x}_j,\psi(\hat{x}_j)\}|= 2\}.
\]
Of course, for every $j\in \ldbrack r \rdbrack\setminus J$, we have $\{\hat{x}_j,\psi(\hat{x}_j)\}\cap \Gamma=\varnothing.$ Note also that, when $n$ is even, we have
$J\cup J^*=\ldbrack r \rdbrack$; if $n$ is odd, by Lemma \ref{psiR-lemma}$\,(ii)$, we have $J\cup J^*=\ldbrack r-1 \rdbrack$.
For every $j\in J$, let us call $y_j$ the unique element in the set $\{\hat{x}_j,\psi(\hat{x}_j)\}\cap \Gamma$ so that $\Gamma=\{y_j:j\in J\}.$

Consider now the subset of $N$ defined by
\[
T=\{y_j : j\in J\}\cup \,\mbox{$\bigcup_{j\in J^*}$}\{\hat{x}_j,\psi(\hat{x}_j)\}
\]
and note that $T\supseteq \Gamma.$ Let $l\in \Theta(p,T)$ and for $j\in J^*,$ let $y_j$ be the maximum of $\{\hat{x}_j,\psi(\hat{x}_j)\}$ with respect to $l$, so that $y_j>_l \psi(y_j)$. Define
\[
M=\{y_j:j\in J\cup J^*\}.
\]
Observe that if $n$ is even, then $M\cup \psi(M)=N$; if $n$ is odd, then $M\cup \psi(M)=N\setminus\{\hat{x}_r\}$. Moreover, we have $|M|=\lfloor \frac{n}{2}\rfloor$, $M\cap \psi(M)=\varnothing$ and $\Gamma\subseteq M \subseteq T$.
The restriction of $l\in \Theta(p,T)$ to $M$ is a linear order $g\in \Theta(p,M)$. We can assume that
\[
a_1\succ_g \ldots \succ_g a_{\lfloor \frac{n}{2}\rfloor}
\]
where $M=\{a_1,\ldots, a_{\lfloor \frac{n}{2}\rfloor}\}.$
 Finally,  consider the linear order $q$ on $N$ such that
\begin{equation}\label{q-magic}
\begin{array}{ll}
a_1\succ_q\ldots\succ_q a_{\lfloor \frac{n}{2}\rfloor}\succ_q
\psi\big(a_{\lfloor \frac{n}{2}\rfloor}\big)\succ_q \ldots\succ_q \psi(a_1)
&\mbox{ if }n\mbox{ is even}\\
\vspace{-2mm}\\
a_1\succ_q \ldots \succ_q a_{\lfloor \frac{n}{2}\rfloor}\succ_q
\hat{x}_r\succ_q
\psi\big(a_{\lfloor \frac{n}{2}\rfloor}\big)\succ_q  \ldots\succ_q  \psi(a_1)
&\mbox{ if }n\mbox{ is odd},\\
\end{array}
\end{equation}
and prove that $q\in S(p)\cap C(p)$.
Note that in the upper part of $q$ there are the alternatives in $M$ and in the lower one those in $\psi(M)$; in the odd case, the fixed point $\hat{x}_r$ of $\psi$ is ranked in the middle, at the position
$r=\frac{n+1}{2}$.

By construction, we have that
 $\psi q\rho_0=q,$ which implies, due to the regularity of $U$, that $q\in S(p)$.
As a consequence, we have that, for every $x,y\in N$,
\begin{equation}\label{consequence}
y\succ_q x \  \hbox{ if and only if } \psi(x)\succ_q \psi(y),
\end{equation}
because, since $\psi^2=id$,
\[
y\succ_q x\quad\Leftrightarrow \quad y\succ_{\psi q\rho_0} x\quad\Leftrightarrow \quad \psi^2(x)\succ_{\psi q} \psi^2(y)\quad\Leftrightarrow \quad \psi(x)\succ_{ q} \psi(y).
\]

In order to complete the proof we need to show that $q\in C(p)$. Since we know that $ \Theta(p,N)\subseteq C(p)$, we proceed  showing that $q\in \Theta(p,N).$ We need to prove that $R(p)\subseteq q$, that is, that  for every $x,y\in N$, $x\succ  y$ implies $x\succ_q y$.
Since when $n$ is even we have $N=M\cup \psi(M)$ and when $n$ is odd we have $N=M\cup \psi(M)\cup\{\hat{x}_r\},$ we reduce to prove that, for every $x,y\in M$:\vspace{-1mm}
\begin{itemize}
\item [$(a)$] $x\succ  y$ implies $x\succ_q y$;\vspace{-1mm}
\item [$(b)$] $x\succ \psi(y)$ implies $x\succ_q \psi(y)$;\vspace{-1mm}
\item [$(c)$] $\psi(x)\succ \psi(y)$ implies $\psi(x)\succ_q \psi(y)$;\vspace{-1mm}
\item [$(d)$] $\psi(x)\succ  y$ implies $\psi(x)\succ_q y$,\vspace{-1mm}
\end{itemize}
and, in the odd case, showing further that, for every $x\in M$:\vspace{-1mm}
\begin{itemize}
\item [$(e)$] $x\succ  \hat{x}_r$ implies $x\succ_q \hat{x}_r$;\vspace{-1mm}
\item [$(f)$] $\hat{x}_r\succ  x$ implies $\hat{x}_r\succ_q x$;\vspace{-1mm}
\item [$(g)$] $\psi(x)\succ  \hat{x}_r$ implies $\psi(x)\succ_q \hat{x}_r$;\vspace{-1mm}
\item [$(h)$] $\hat{x}_r\succ  \psi(x)$ implies $\hat{x}_r\succ_q \psi(x).$\vspace{-1mm}
\end{itemize}
Fix then $x,y\in M$ and prove first $(a)$, $(b)$, $(c)$ and $(d)$. Note that $x\neq \psi(x)$ and $y\neq\psi(y),$ because we have observed that if $\psi$ admits the fixed point $\hat{x}_r$, then $\hat{x}_r\notin M.$\vspace{-1mm}
\begin{itemize}
\item [$(a)$] Let $x\succ  y$. Since $g\in \Theta(p,M)$, we have $x\succ_g y$ and then also $x\succ_q y$.\vspace{-1mm}
\item [$(b)$] By the construction of $q$, for every $x,y\in M,$ we have $x\succ_q \psi(y).$\vspace{-1mm}
\item [$(c)$] Let $\psi(x)\succ \psi(y)$. Then, by  \eqref{prop-R}, we also have $y\succ  x$ and by $(a)$ we get $y>_q x$, which  by (\ref{consequence}) implies
$\psi(x)\succ_q \psi(y)$.\vspace{-1mm}
\item [$(d)$] Let us prove that the condition $\psi(x)\succ  y$ never realizes. Indeed, assume by contradiction $\psi(x)\succ  y$. Then, by  \eqref{prop-R}, we also have $\psi(y)\succ  x$.
If $y\in \Gamma$, then, by  \eqref{prop-R},  we have $\psi(x)\in \Gamma\subseteq M$ and thus $x\in M\cap \psi(M)=\varnothing$, a contradiction.
Thus, we reduce to the case $y\not\in \Gamma$. If $x\in \Gamma$, then, again by Lemma \ref{catena-estesa}, $\psi(y)\in \Gamma\subseteq M$ and we get the  contradiction $y\in M\cap \psi(M)=\varnothing$.
 If instead $x\not\in \Gamma$, then $x,\psi(x),y,\psi(y)\in T$.
As a consequence,  $\psi(x)\succ  y$ implies $\psi(x)\succ_l y$ and $\psi(y)\succ  x$ implies $\psi(y)\succ_l x$. Moreover, as $y$ is the maximum of $\{y,\psi(y)\}$ with respect to $l,$ we also have $y\succ_l\psi(y)$. By transitivity of $f$, we then get $\psi(x)\succ_l x,$ which is a contradiction because $x$ is the maximum of $\{x,\psi(x)\}$ with respect to $l$.\vspace{-1mm}
\end{itemize}
Assume now that $n$ is odd. Then fix $x\in M$ and prove $(e), (f), (g)$ and $(h)$.\vspace{-1mm}
\begin{itemize}
\item [$(e)$] This case is trivial, because, by the construction of $q$, for every $x\in M$, we have $x\succ_q \hat{x}_r.$\vspace{-1mm}
\item [$(f)$] Let us prove that it cannot be $\hat{x}_r\succ  x$. Assume, by contradiction, $\hat{x}_r\succ  x$. Then, by Lemma \ref{fisso}(ii) we get $\psi(x)\in \Gamma\subseteq M,$ against $M\cap \psi(M)=\varnothing.$\vspace{-1mm}
\item [$(g)$] Let us prove that it cannot be $\psi(x)\succ  \hat{x}_r$.  Assume, by contradiction, $\psi(x)\succ  \hat{x}_r$. Then, by Lemma \ref{fisso}(i), we get
$\psi(x)\in \Gamma\subseteq M$ against $M\cap \psi(M)=\varnothing$.\vspace{-1mm}
\item [$(h)$] This case is trivial because, by the construction of $q$, we have $\hat{x}_r\succ_q\psi(x)$ for all $x\in M.$
\end{itemize}

\subsection{Proof of Proposition \ref{existence}}

Given $p\in \mathcal{P},$ there exist a unique $j\in \mathcal{P}^U$ such that $p=p^{j\,(\varphi,\psi,\rho)}$ for some $(\varphi,\psi,\rho)\in U$. We claim  that, if for some $j\in \mathcal{P}^U$ there exist $(\varphi_1,\psi_1,\rho_1),(\varphi_2,\psi_2,\rho_2)\in U$ such that $p^{j\,(\varphi_1,\psi_1,\rho_1)}=p^{j\,(\varphi_2,\psi_2,\rho_2)}$, then $\psi_1 q_{j} \rho_1 =\psi_2 q_{j} \rho_2$. Indeed, by (\ref{action-e}), we have that $p^{j\,(\varphi_1,\psi_1,\rho_1)}=p^{j\,(\varphi_2,\psi_2,\rho_2)}$ implies
$(\varphi^{-1}_2\varphi_1,\psi_2^{-1}\psi_1,\rho_2^{-1}\rho_1)\in \mathrm{Stab}_U(p^{j}).$
Since $q_{j}\in S(p^{j})$ and $\Omega$ is abelian, we have that
$q_{j}=\psi^{-1}_2\psi_1 q_{j} \rho_2^{-1}\rho_1=\psi^{-1}_2\psi_1 q_{j} \rho_1\rho_2^{-1}$, and thus $\psi_1 q_{j} \rho_1 =\psi_2 q_{j} \rho_2$.

As a consequence, the resolute {\sc spc} $f$ defined, for every $p\in\mathcal{P}$, by $f(p)=\psi q_j \rho$, where $j\in \mathcal{P}^U$ and $(\varphi,\psi,\rho)\in U$ are such that $p=p^{j\,(\varphi,\psi,\rho)}$, is consistent. Moreover, for every $j\in \mathcal{P}^U$, $f(p^j)=q_j$.
Let us prove that $f\in \mathfrak{P}^{*U}$. Consider $p\in \mathcal{P}$ and $(\varphi,\psi,\rho)\in U$. Let $p=p^{j\,(\varphi_1,\psi_1,\rho_1)}$, for some $j\in \mathcal{P}^U$ and
$(\varphi_1,\psi_1,\rho_1)\in U$.
By the definition of $f$ and by \eqref{action-e} and using again the fact that $\Omega$ is abelian,  we have
\[
f(p^{(\varphi,\psi,\rho)})=f\left(\left(p^{j\,(\varphi_1,\psi_1,\rho_1)}\right)^{(\varphi,\psi,\rho)}\right)
=f(p^{j\,(\varphi\varphi_1,\psi\psi_1,\rho\rho_1)})=(\psi\psi_1)q_j(\rho\rho_1)
\]
\[
=(\psi\psi_1)q_j(\rho_1\rho)=\psi(\psi_1q_j\rho_1)\rho=\psi f(p^{j\,(\varphi_1,\psi_1,\rho_1)})\rho=\psi f(p)\rho.
\]
In order to prove the uniqueness of $f$, it suffices to note that if $f'\in\mathfrak{F}^{*U}$ is such that, for every $j\in \mathcal{P}^U$, $f'(p^j)=q_j$, then we have  $f=f'$ by
Proposition \ref{rappresentanti1}.

\subsection{Proof of Proposition \ref{f-fu}}

Let $C\in\mathfrak{P}^U$. Consider $f\in \mathfrak{F}$ defined as follows. Given $p\in \mathcal{P}$, consider the unique $j\in\mathcal{P}^U$ such that $p\in j$ and the nonempty set $U_p=\{(\varphi,\psi,id)\in U: p=p^{j\,(\varphi,\psi,id)}\}$. Pick $(\varphi,\psi,id)\in U_p$ and let
$f(p)=\psi x_j$. We need to prove that the value of $f(p)$ does not depend on the particular element chosen in $U_p$.
Indeed, let $(\varphi_1,\psi_1,id),(\varphi_2,\psi_2,id)\in U_p$ and note that $(\varphi_2^{-1}\varphi_1,\psi_2^{-1}\psi_1,id)\in \mathrm{Stab}_U(p^j)$.
Since $U$ is regular, that gives  $\psi_1=\psi_2$ and then $\psi_1 x_j=\psi_2x_j$.

We show that $f$ satisfies all the desired properties.
Since $U\le G$, we have $(id,id,id)\in U$ and thus the definition of $f$ immediately implies
 $f(p^j)=x_j$.
Let us prove that $f\in\mathfrak{F}^U$.
Consider $p\in\mathcal{P}$ and $(\varphi,\psi,id)\in U$ and show that $f(p^{(\varphi,\psi,id)})=\psi f(p)$. Let $p=p^{j\,(\varphi_1,\psi_1,id)}$ for suitable $j\in\mathcal{P}^U$ and $(\varphi_1,\psi_1,id)\in U$. Thus,
$
f(p)=\psi_1x_j
$
and, by \eqref{action-e},
$
f(p^{(\varphi,\psi,id)})=f(p^{j\,(\varphi\varphi_1,\psi\psi_1,id)})=\psi\psi_1  x_j=\psi f(p).
$

Let us next prove that $f\in\mathfrak{F}_{C}$. Consider then $p\in\mathcal{P}$ and show that $f(p)\in C(p)$. Let $p=p^{j\,(\varphi_1,\psi_1,id)}$ for suitable $j\in\mathcal{P}^U$ and $(\varphi_1,\psi_1,id)\in U$. Thus,
$
f(p)=\psi_1x_j
$ and, since
$C$ is $U$-consistent,
$
\psi_1x_j\in \psi_1C(p^j)=C(p^{j\,(\varphi_1,\psi_1,id)})=C(p).
$
Finally, in order to prove uniqueness, let $f'\in\mathfrak{F}^U_C$ such that
 $f'(p^j)=x_j$ for all $j\in\mathcal{P}^U$. Then $f'$ and $f$ coincides on $(p^j)_{j\in\mathcal{P}^U}\in\mathfrak{S}(U)$ and Proposition \ref{rappresentanti2} applies giving $f'=f.$

Let now $C\in\mathfrak{C}^U_k$. The proposition can be proved using the same argument.

\subsection{Proof of Theorem \ref{fu-min-2}}

Let $C\in\mathfrak{P}^U$. Consider $f\in\mathfrak{F}^U_{C}$ and note that,
 for every $j\in\mathcal{P}^U$, $f(p^j)\in C(p^j)$ and $\Phi\left((f(p^j))_{j\in\mathcal{P}^U}\right)=f$. Thus $\Phi$ is bijective and we have
 $|\mathfrak{F}^U_{C}|=\left|\times_{j\in\mathcal{P}^U}C(p^j)	\right|= \prod_{j\in\mathcal{P}^U}\left|C(p^j)\right|.$
 Finally note that if $C$ is decisive, then, for every $j\in\mathcal{P}^U$,  $C(p^j)\neq \varnothing$ so that $\mathfrak{F}^U_{C}\neq\varnothing$.

Let now $C\in\mathfrak{C}^U_k$. The theorem can be proved using the same argument.

\subsection{Proof of Proposition \ref{fu-min-ex}}

Let $C\in\mathfrak{P}^U$.
Given $j\in\mathcal{P}_2^U$, consider the set $K^U(p^j)=\left\{\sigma\in S_n: \psi_j=\sigma \rho_0 \sigma^{-1}  \right\}$, where $\psi_j$ is defined in \eqref{psi-j}.
Since $U$ is regular, $K^U(p^j)$ is nonempty so that we can choose an element $\sigma_j$ in $K^U(p^j)$. Note that, for every $j\in\mathcal{P}_2^U$ and  $(\varphi,\psi,\rho)\in \mathrm{Stab}_U(p^j)$, we have $\psi=\sigma_j\rho\sigma_j^{-1}$.

Let us consider then $f\in\mathfrak{F}$ defined, for every $p\in\mathcal{P}$, as follows. Given $p\in \mathcal{P}$, consider the unique $j\in\mathcal{P}^U$ such that $p\in j$ and the nonempty set $U_p=\{(\varphi,\psi,\rho)\in U: p=p^{j\,(\varphi,\psi,\rho)}\}$. Pick $(\varphi,\psi,\rho)\in U_p$ and let
\[
f(p)=
\left\{
\begin{array}{ll}
\psi y_j&\mbox{ if }j\in\mathcal{P}_1^U \mbox{ and }\rho=id\\
\vspace{-2mm}\\
\psi\psi_*^{-1}z_j&\mbox{ if }j\in\mathcal{P}_1^U \mbox{ and } \rho=\rho_0\\
\vspace{-2mm}\\
\psi \sigma_{j} \rho \sigma_{j}^{-1}x_j&\mbox{ if }j\in\mathcal{P}_2^U
\end{array}
\right.
\]
We need to prove that the value of $f(p)$ does not depend on the particular element chosen in $U_p$.
Indeed, let $(\varphi_1,\psi_1,\rho_1),(\varphi_2,\psi_2,\rho_2)\in U_p$ and recall that 
$(\varphi_2^{-1}\varphi_1,\psi_2^{-1}\psi_1,\rho_2^{-1}\rho_1)\in \mathrm{Stab}_U(p^j)$.
\begin{itemize}
\item[-] If $j\in\mathcal{P}_1^U$, then $(\varphi_2^{-1}\varphi_1,\psi_2^{-1}\psi_1,\rho_2^{-1}\rho_1)\in \mathrm{Stab}_U(p^j)$ implies
 $\rho_2=\rho_1$ and $\psi_1=\psi_2$. As a consequence, if $\rho_1=\rho_2=id$, then
$\psi_1y_j=\psi_2y_j$, while if $\rho_1=\rho_2=\rho_0$, then
$(\psi_1\psi_*^{-1})z_j=(\psi_2\psi_*^{-1})z_j$.
\item[-] If  $j\in\mathcal{P}_2^U$, then $(\varphi_2^{-1}\varphi_1,\psi_2^{-1}\psi_1,\rho_2^{-1}\rho_1)\in \mathrm{Stab}_U(p^j)$ implies
$\psi_2^{-1}\psi_1= \sigma_j \rho_2^{-1}\rho_1 \sigma_j^{-1}$, that is, $\psi_1 \sigma_j \rho_1 \sigma_j^{-1}=\psi_2 \sigma_j \rho_2 \sigma_j^{-1}$,
as $\rho=\rho^{-1}$ for all $\rho\in\Omega$. Then we get
$
\psi_1 \sigma_j \rho_1 \sigma_j^{-1}x_j=\psi_2 \sigma_j \rho_2 \sigma_j^{-1}x_j.
$
\end{itemize}

We show that $f$ satisfies all the desired properties.
Since $U\le G$, we have $(id,id,id)\in U$ and thus the definition of $f$ immediately implies
 $f(p^j)=y_j$ and $f(p^{j\,(\varphi_*,\psi_*,\rho_0)})=z_j$ for all $j\in\mathcal{P}_1^U$,
and $f(p^j)=x_j$ for all $j\in\mathcal{P}_2^U$.
We prove that $f\in\mathfrak{F}^U$.
Consider $p\in\mathcal{P}$ and $(\varphi,\psi,\rho)\in U$ and show that if $\rho=id$, then $f(p^{(\varphi,\psi,\rho)})=\psi f(p)$; if $\rho=\rho_0$, then  $f(p^{(\varphi,\psi,\rho)})\neq\psi f(p)$. Let $p=p^{j\,(\varphi_1,\psi_1,\rho_1)}$ for suitable $j\in\mathcal{P}^U$ and $(\varphi_1,\psi_1,\rho_1)\in U$.
\begin{itemize}
\item[-] If $j\in\mathcal{P}_1^U$ and $\rho_1=id$, then
$
f(p)=\psi_1 y_j.
$
By \eqref{action-e},
if $\rho=id$, then
$
f(p^{(\varphi,\psi,\rho)})=f(p^{j\,(\varphi\varphi_1,\psi\psi_1,id)})=\psi\psi_1 y_j=\psi f(p),
$
while if $\rho=\rho_0$, then
$
f(p^{(\varphi,\psi,\rho)})=f(p^{j\,(\varphi\varphi_1,\psi\psi_1,\rho_0)})=\psi\psi_1 \psi_*^{-1}z_j\neq \psi\psi_1 y_j=\psi f(p),
$
since $z_j\neq \psi_* y_j$ because $(y_j,z_j)\in A_C^1(p^j)$.
\item[-] If $j\in\mathcal{P}_1^U$ and $\rho_1=\rho_0$, then
$
f(p)=\psi_1 \psi^{-1}_*z_j.
$
By \eqref{action-e}, if $\rho=id$, then
$
f(p^{(\varphi,\psi,\rho)})=f(p^{j\,(\varphi\varphi_1,\psi\psi_1,\rho_0)})=\psi\psi_1 \psi^{-1}_*z_j=\psi f(p),
$
while if $\rho=\rho_0$, then
$
f(p^{(\varphi,\psi,\rho)})=f(p^{j\,(\varphi\varphi_1,\psi\psi_1,id)})=\psi\psi_1 y_j\neq \psi\psi_1 \psi^{-1}_*z_j=\psi f(p),
$
since $z_j\neq \psi_* y_j$ because $(y_j,z_j)\in A_C^1(p^j)$.
\item[-] If $j\in\mathcal{P}_2^U$, then
$
f(p)=\psi_1 \sigma_{j} \rho_1 \sigma_{j}^{-1} x_j
$
and, by \eqref{action-e},
$$
f(p^{(\varphi,\psi,\rho)})=f(p^{j\,(\varphi\varphi_1,\psi\psi_1,\rho\rho_1)})=\psi\psi_1 \sigma_{j} \rho \rho_1\sigma_{j}^{-1} x_j.
$$
As a consequence, if $\rho=id$, we get $f(p^{(\varphi,\psi,\rho)})=\psi f(p) $.
If instead $\rho=\rho_0$, we have that  $f(p^{(\varphi,\psi,\rho)})\neq\psi f(p)$ if and only if
$
\psi\psi_1 \sigma_{j} \rho_0 \rho_1 \sigma_{j}^{-1} x_j\neq
\psi\psi_1 \sigma_{j} \rho_1 \sigma_{j}^{-1} x_j
$
if and only if
$
\sigma_{j} \rho_0  \sigma_{j}^{-1} x_j\neq
x_j.
$
However, the last relation holds true since $\sigma_{j} \rho_0  \sigma_{j}^{-1} =\psi_j$ and $\psi_jx_j\neq x_j$ because $x_j\in A_C^2(p^j)$.
\end{itemize}

Let us next prove that $f\in\mathfrak{F}_C$.
Consider $p\in\mathcal{P}$ and show that  $f(p)\in C(p)$. Let $p=p^{j\,(\varphi_1,\psi_1,\rho_1)}$ for suitable $j\in\mathcal{P}^U$ and $(\varphi_1,\psi_1,\rho_1)\in U$.
\begin{itemize}
	\item[-] If $j\in\mathcal{P}_1^U$ and $\rho_1=id$, then $f(p)=\psi_1 y_j$ and, by the $U$-consistency of $C$,
$\psi_1 y_j\in\psi_1C(p^j)=C(p^{j\,(\varphi_1,\psi_1,id)})=C(p)$.
\item[-] If $j\in\mathcal{P}_1^U$ and $\rho_1=\rho_0$, then $f(p)=\psi_1\psi_*^{-1}z_j$ and, by \eqref{action-e} and the $U$-consistency of $C$,
\[
\psi_1\psi_*^{-1}z_j\in \psi_1\psi_*^{-1}C(p^{j\,(\varphi_*,\psi_*,\rho_0)})=
\]
\[
C\left((p^{j\,(\varphi_*,\psi_*,\rho_0)})^{(\varphi_1\varphi_*^{-1},\psi_1\psi_*^{-1},id)}\right)=C(p^{j\,(\varphi_1,\psi_1,\rho_0)})=C(p).
\]
\item[-] If $j\in\mathcal{P}_2^U$ and $\rho_1=id$, then $f(p)=\psi_1x_j$ and, by the $U$-consistency of  $C$,
$\psi_1x_j\in \psi_1 C(p^j)
=C(p^{j\,(\varphi_1,\psi_1,id)})=C(p)$.
\item[-] If $j\in\mathcal{P}_2^U$ and $\rho_1=\rho_0$, then let
 $(\varphi_2,\psi_2,\rho_0)\in U$ be such that $p^{j\,(\varphi_2,\psi_2,\rho_0)}=p^j$. By \eqref{action-e}, we have
\[
p=p^{j\,(\varphi_1,\psi_1,\rho_0)}=(p^{j\,(\varphi_2,\psi_2,\rho_0)})^{(\varphi_1\varphi_2^{-1},\psi_1\psi_2^{-1},id)}=p^{j\,(\varphi_1\varphi_2^{-1},\psi_1\psi_2^{-1},id)}.
\]
Thus, $f(p)=\psi_1\psi_2^{-1}x_j$ and, by the $U$-consistency of  $C$,
$$\psi_1\psi_2^{-1}x_j\in \psi_1\psi_2^{-1} C(p^j)
=C(p^{j\,(\varphi_1\varphi_2^{-1},\psi_1\psi_2^{-1},id)})=C(p).$$
\end{itemize}

Finally, in order to prove uniqueness, let $f'\in\mathfrak{F}^U_C$ such that
 $f'(p^j)=y_j$ and $f'(p^{j\,(\varphi_*,\psi_*,\rho_0)})=z_j$ for all $j\in\mathcal{P}_1^U$,
and $f'(p^j)=x_j$ for all $j\in\mathcal{P}_2^U$. Then $f,f'\in \mathfrak{P}^U$ realize $f(p^j)=f'(p^j)$ for all $j\in \mathcal{P}^U$ and $f(p^{j\,(\varphi_*,\psi_*,\rho_0)})=f'(p^{j\,(\varphi_*,\psi_*,\rho_0)})$ for all $j\in \mathcal{P}_1^U$. Hence, the thesis follows from Proposition \ref{rappresentanti2}.

Let now $C\in\mathfrak{C}^U_k$. The proposition can be proved using the same argument.

\end{document}